\documentclass[11pt, oneside]{amsart}   	
\usepackage{geometry}                		
\usepackage{graphicx}
\usepackage{color}
\usepackage{amssymb}
\usepackage{amsmath}
\usepackage{amsthm}
\usepackage{amsrefs}
\usepackage{appendix}

\allowdisplaybreaks
\numberwithin{equation}{section}

\newcommand{\R}{\mathbb{R}}

\newcommand{\N}{\mathbb{N}}
\newcommand{\Sym}{\operatorname{Sym}}
\newcommand{\Ric}{\operatorname{Ric}}
\newcommand{\Rm}{\operatorname{Rm}}

\newcommand{\inj}{\operatorname{inj}}

\newcommand{\supp}{\operatorname{Supp}}

\newcommand{\vol}{\operatorname{vol}}
\newcommand{\id}{\operatorname{id}}
\newcommand{\loc}{\text{loc}}

\renewcommand{\flat}{\text{flat}}
\newcommand{\diam}{\operatorname{diam}}
\newcommand{\met}{\operatorname{met}}
\newcommand{\G}{\operatorname{G}}
\newcommand{\Rf}{\operatorname{Rf}}
\renewcommand{\L}{\mathcal{L}}

\theoremstyle{plain}
\newtheorem{theorem*}{Theorem}
\newtheorem{corollary*}{Corollary}
\newtheorem{question*}{Question}
\newtheorem{definition*}{Definition}
\newtheorem{theorem}{Theorem}[section]
\newtheorem{lemma}[theorem]{Lemma}
\newtheorem{corollary}[theorem]{Corollary}

\newtheorem{proposition}[theorem]{Proposition}
\theoremstyle{definition}
\newtheorem{definition}[theorem]{Definition}
\theoremstyle{remark}
\newtheorem{remark}[theorem]{Remark}

 \title[lower scalar curvature bounds via regularizing Ricci flow]{Pointwise lower scalar curvature bounds for $C^0$ metrics via regularizing Ricci flow}
\author{Paula Burkhardt-Guim}
\address{Dept. of Mathematics \\ University of California, Berkeley}
\email{paulab@math.berkeley.edu}
\date{\today}							

\begin{document}
\begin{abstract}
In this paper we propose a class of local definitions of weak lower scalar curvature bounds that is well defined for $C^0$ metrics. We show the following: that our definitions are stable under greater-than-second-order perturbation of the metric, that there exists a reasonable notion of a Ricci flow starting from $C^0$ initial data which is smooth for positive times, and that the weak lower scalar curvature bounds are preserved under evolution by the Ricci flow from $C^0$ initial data.
\end{abstract}
\maketitle

\section{Introduction}

It is natural to ask whether there exists a useful notion of scalar curvature for singular metric spaces. Gromov has introduced (see \cite{Gro}) a definition of lower scalar curvature bounds for certain singular spaces, which has the advantage that it is well-defined for $C^0$ metrics, rather than requiring higher regularity. As a result, Gromov  was able to prove in \cite{Gro} that the space of $C^2$ Riemannian metrics with scalar curvature bounded below was closed with respect to $C^0$-convergence. In \cite{Bam}, Bamler provided an alternative proof of this fact, using Ricci flow and the results of Koch and Lamm \cite{KL1}, and making use of the fact that, for smooth Ricci flows, lower bounds on the scalar curvature are preserved.

Bamler's approach and the preservation of (constant) lower bounds on the scalar curvature under the Ricci flow leads one to ask whether it is possible to formulate a local notion of lower bounds on the scalar curvature for singular spaces in terms of Ricci flow (for a discussion of classical Ricci flow, see \S \ref{sec:preliminaries}). A satisfactory notion of a pointwise lower bound on the scalar curvature should satisfy the following requirements: For any constant $\kappa$, we should have
\begin{enumerate}
\item\label{item:secondorder} \textbf{Stability under greater-than-second-order perturbation:} If $g'$ and $g''$ are two $C^0$ metrics that agree to greater than second order around a point $x_0$, i.e. $|g'(x) - g''(x)| \leq cd^{2+\eta}(x,x_0)$ for some $c, \eta>0$ and all $x$ in a neighborhood of $x_0$, then $g'$ should have scalar curvature bounded below by $\kappa$ in the weak sense at $x_0$ if and only if $g''$ does. Moreover, if $g'$ and $g''$ are $C^0$ metrics on different manifolds which merely agree to greater than second order under pullback by a locally defined diffeomorphism, the conclusion should still hold.
\item \textbf{Preservation under the Ricci flow:} If $g$ is a $C^0$ metric on a closed manifold that has scalar curvature bounded below by $\kappa$ in the weak sense at every point, and $\tilde g_t$ is a solution to the Ricci flow starting from $g$ in some sense, then $\tilde g_t$ should have scalar curvature bounded below by $\kappa$ at every point for all $t>0$ for which the flow is defined. This is true for Ricci flows starting from smooth initial data.
\item \textbf{Agreement with the classical notion for $C^2$ metrics:} If $g$ is a $C^2$ metric with scalar curvature bounded below by $\kappa$ at $x_0$ in the generalized sense for $C^0$ metrics, then $g$ should have scalar curvature bounded below by $\kappa$ at $x_0$ in the classical sense.
\end{enumerate}

We now explain what it means to have Ricci flow starting from a metric that is only $C^0$. In \cite{Sim} Simon showed that, for a complete initial metric, there is a smooth, time-dependent family of metrics defined on a positive time interval and converging uniformly to the initial data, which solves the Ricci-DeTurck flow, an evolution equation closely related to the Ricci flow (we discuss the Ricci-DeTurck flow in greater detail in \S \ref{sec:preliminaries}). Additionally, in \cite{KL1} and \cite{KL2} Koch and Lamm developed a natural notion of a solution to various geometric flows starting from nonsmooth, or ``rough'', initial data, namely, a solution to the corresponding integral equation for these geometric flows. For positive times, certain integral solutions from the rough initial data have high regularity. These results suggest that one might define a weak notion of lower scalar curvature bounds for $C^0$ metrics by finding a solution to the flow starting from the $C^0$ data, and then checking that, for small positive times, the lower scalar curvature bound is satisfied in the classical sense. In order to state our notion of local lower scalar curvature bounds for $C^0$ metrics, we first show that there is a Ricci flow starting from $C^0$ initial data in the following sense:

\begin{theorem}\label{thm:introthm0}
Let $M$ be a closed manifold and $g_0$ a $C^0$ metric on $M$. Then there exists a time-dependent family of smooth metrics $(\tilde g_t)_{t\in(0,T]}$ and a continuous surjection $\chi:M\to M$ such that the following are true:
\begin{enumerate}
\item[(a)] The family $(\tilde g_t)_{t\in(0,T]}$ is a Ricci flow, and
\item[(b)] There exists a smooth family of diffeomorphisms $(\chi_t)_{t\in (0,T]}: M\to M$ such that
\begin{equation*}
\chi_t\xrightarrow[t\to 0]{C^0}\chi\,  \text{ and } \, ||(\chi_t)_*\tilde g_t - g_0||_{C^0(M)}\xrightarrow[t\to 0]{}0.
\end{equation*}
\end{enumerate}
Moreover, for any $x\in M$, $\diam\{\chi_s(x): s\in(0,t]\} \leq C\sqrt{t}$  for some constant $C>0$ independent of $x$, and any two such families are isometric, in the sense that if $\tilde g_t'$ is another such family with corresponding continuous surjection $\chi'$, then there exists a stationary diffeomorphism $\alpha:M\to M$ such that $\alpha^* \tilde g_t = \tilde g_t'$ and $\chi\circ\alpha = \chi'$.
\end{theorem}

We call the pair $((\tilde g_t)_{t\in(0,T]}, \chi)$ a \emph{regularizing Ricci flow} for $g_0$, and use it to make the following definition:
\begin{definition}\label{def:RRFscalar}
Let $M^n$ be a closed manifold and $g_0$ a $C^0$ metric on $M$. For $0 < \beta < 1/2$ we say that $g_0$ has scalar curvature bounded below by $\kappa$ at $x$ in the \emph{$\beta$-weak sense} if there exists a regularizing Ricci flow $((\tilde g_t)_{t\in(0,T]}, \chi)$ for $g_0$ such that, for some point $y\in M$ with $\chi(y) = x$, we have
\begin{equation}\label{eqn:RRFscalar}
\inf_{C>0}\left(\liminf_{t\searrow 0}\left(\inf_{B_{\tilde g(t)}(y,Ct^\beta)}R^{\tilde g}(\cdot, t)\right)\right) \geq \kappa,
\end{equation}
where $B_{\tilde g(t)}(y,Ct^\beta)$ denotes the ball of radius $Ct^\beta$ about $y$, measured with respect to the metric $\tilde g(t)$, and $R^{\tilde g}(\cdot,t)$ denotes the scalar curvature of $\tilde g_t$ .
\end{definition}

\begin{remark}
In fact, we will show in \S \ref{sec:equivalentdefs} that Definition \ref{def:RRFscalar} is independent of choice of $y$, so it is equivalent to require that (\ref{eqn:RRFscalar}) hold at $y$ for \emph{all} $y$ with $\chi(y) = x$.
\end{remark}

The objective of this paper is to show that Definition \ref{def:RRFscalar} satisfies items ($1$), ($2$), and ($3$). It is clear that ($3$) is satisfied, since if $g_0$ is $C^2$, then the regularizing Ricci flow is the usual Ricci flow with $\chi = \id$, and 
\begin{equation}
\inf_{C>0}\left(\liminf_{t\searrow 0}\left(\inf_{B_{\tilde g(t)}(y,Ct^\beta)}R^{\tilde g}(\cdot, t)\right)\right) = \lim_{t\to 0}R^{\tilde g}(x, t).
\end{equation}

In order to show that Definition \ref{def:RRFscalar} satisfies item ($1$), we study the stability of the difference of scalar curvatures of regularizing Ricci flows from metrics that agree to greater than second order. We show:
\begin{theorem}\label{cor:introcor1}
Suppose $g'$ and $g''$ are two $C^0$ metrics on closed manifolds $M'$ and $M''$ respectively, and that there is a locally defined diffeomorphism $\phi: U\to V$ where $U$ is a neighborhood of $x_0'$ in $M'$ and $V$ is a neighborhood of $x_0''$ in $M''$ with $\phi(x_0') = x_0''$. Suppose furthermore that $g'$ and $\phi^*g''$ agree to greater than second order around $x_0'$, i.e. $|g'(x) - \phi^*g''(x)| \leq cd^{2+\eta}(x,x_0)$ for some $c, \eta>0$ and all $x$ in a neighborhood of $x_0'$. Then there exist regularizing Ricci flows $(\tilde g_t', \chi')$ and $(\tilde g_t'', \chi'')$ for $g'$ and $g''$ respectively such that, for $1/(2+\eta) < \beta <1/2$, $C>0$, and $t$ sufficiently small depending on $C$, $\beta$, and $\eta$, we have
\begin{equation}
\sup_{B(x_0', Ct^\beta)}|R^{(\chi_t')_*\tilde g'_t} - \phi^*R^{(\chi_t'')_*\tilde g''_t}|\leq ct^\omega,
\end{equation}
where $\omega$ is some positive exponent, $c$ is a constant that does not depend on $t$ or $C$, $R^{(\chi_t')_*\tilde g'_t}$ and $R^{(\chi_t'')_*\tilde g''_t}$ denote the scalar curvatures with respect to $(\chi_t')_*\tilde g_t'$ and $(\chi_t'')_*\tilde g_t''$ respectively, and $(\chi_t')$ and $(\chi_t'')$ are the smooth families of diffeomorphisms for $\tilde g_t'$ and $\tilde g_t''$ respectively, whose existence is given by (b) in Theorem \ref{thm:introthm0}.

In particular, Definition \ref{def:RRFscalar} holds for $g'$ at $x_0'$ if and only if it holds for $g''$ and $x_0''$.
\end{theorem}
Theorem \ref{cor:introcor1} follows from some results in \S \ref{sec:fixedptconstruction} and \S \ref{sec:RRF}, as we shall discuss.

Definition \ref{def:RRFscalar} may be reformulated in terms of Ricci or Ricci-DeTurck flows, and also has a natural formulation in terms of subspaces of the space of germs of Riemannian metrics at a point: By Theorem \ref{cor:introcor1}, Definition \ref{def:RRFscalar} descends to the space of germs of metrics at a point, and further descends to the quotient space induced on the space of germs of metrics at $x_0$ by the relation $[g] \sim [g']$ if $g$ and $g'$ agree to greater than second order at $x_0$. Moreover, we show that Definition \ref{def:RRFscalar} behaves appropriately under the Ricci flow, and thus satisfies item ($2$):

\begin{theorem}\label{thm:introthm2}
Suppose that $g_0$ is a $C^0$ metric on a closed manifold $M$, and suppose there is some $\beta \in (0,1/2)$ such that $g_0$ has scalar curvature bounded below by $\kappa$ in the $\beta$-weak sense at all points in $M$. Suppose also that $(\tilde g(t))_{t\in (0,T]}$ is a Ricci flow starting from $g_0$ in the sense of Theorem \ref{thm:introthm0}. Then the scalar curvature of $\tilde g(t)$, $R(\tilde g(t))$, satisfies $R(\tilde g(t)) \geq \kappa$ everywhere on $M$, for all $t\in (0,T]$.
\end{theorem}

Theorem \ref{thm:introthm2} implies:
\begin{corollary}\label{cor:C2approximation}
If $(M,g)$ is a closed Riemannian manifold with $C^0$ metric $g$, and if there exists $\beta\in(0,1/2)$ such that, at every point in $M$, $g$ has scalar curvature bounded below by $\kappa$ in the $\beta$-weak sense, then there exists a sequence of $C^2$ metrics on $M$ with scalar curvature bounded below by $\kappa$ that converges uniformly to $g$.
\end{corollary}

Theorem \ref{thm:introthm2} also implies:
\begin{theorem}\label{thm:kappaapproximation}
Let $g$ be a $C^0$ metric on a closed manifold $M$ which admits a uniform approximation by $C^0$ metrics $g^i$ such that, for some $\beta\in(0,1/2)$, $g^i$ has scalar curvature bounded below by $\kappa_i$ in the $\beta$-weak sense everywhere on $M$, where $\kappa_i$ is some sequence of numbers such that $\kappa_i\xrightarrow[i\to \infty]{} \kappa$ for some number $\kappa$. Then $g$ has scalar curvature bounded below by $\kappa$ in the $\beta$-weak sense. In particular, any regularizing Ricci flow $(\tilde g(t))_{t\in (0,T]}$ for  $g$ satisfies $R(\tilde g(t)) \geq \kappa$ for all $t\in (0,T]$, so $g$ admits a uniform approximation by smooth metrics with scalar curvature bounded below by $\kappa$.
\end{theorem}

As a corollary of Theorem \ref{thm:kappaapproximation}, we may answer the following question, posed by Gromov in \cite{Gro}:
\begin{question*}[{\cite[Page $1119$]{Gro}}]\label{q:Gromov}
Let $g$ be a continuous Riemannian metric on a closed manifold $M$ which admits a $C^0$-approximation by smooth Riemannian metrics $g_i$ with $R(g_i)\geq -\varepsilon_i\xrightarrow[i\to \infty]{}0$. Does $M$ admit a smooth metric of nonnegative scalar curvature?
\end{question*}

By setting $\kappa_i = -\varepsilon_i$ in Theorem \ref{thm:kappaapproximation}, we obtain the following response:
\begin{corollary}\label{cor:GromovQ}
If $(M,g)$ is as in Question \ref{q:Gromov}, then any regularizing Ricci flow $(\tilde g_t)_{t\in (0,T]}$ for $g$ satisfies $R(\tilde g_t) \geq 0$ for all $t\in (0,T]$. In particular, $M$ admits a smooth metric of nonnegative scalar curvature, and moreover, $g$ admits a uniform approximation by smooth metrics with nonnegative scalar curvature.
\end{corollary}

We use similar methods to show a torus rigidity result, motivated by the scalar torus rigidity theorem, which was first proven by Schoen and Yau \cite{SY} for dimensions $\leq 7$, and later proven by Gromov and Lawson \cite{GL} for all dimensions, and which says that any Riemannian manifold with nonnegative scalar curvature that is diffeomorphic to the torus must be isometric to the flat torus. We show:
\begin{corollary}\label{cor:torusrigidity}
Suppose $g$ is a $C^0$ metric on the torus $\mathbb{T}$, and that there is some $\beta \in (0,1/2)$ such that $g$ has nonnegative scalar curvature in the $\beta$-weak sense everywhere. Then $(\mathbb{T}, g)$ is isometric as a metric space to the standard flat metric on $\mathbb{T}$.
\end{corollary}

\begin{remark}
Corollary \ref{cor:torusrigidity} is in fact the optimal result, i.e. it is not possible to show that there is a \emph{Riemannian} isometry between $g$ and the standard flat metric. In the case where $g^1$ and $g^2$ are smooth metrics, a metric space isometry is automatically a smooth Riemannian isometry. However, there exist examples of $C^0$ isometries between $C^0$ metrics which are not $C^1$. Moreover, the regularizing Ricci flow is invariant under $C^0$ isometry, in the sense that if $\varphi: (M_2, g^2) \to (M_1, g^1)$ is a metric space isometry, then, for any two regularizing Ricci flows $((\tilde g^1(t))_{t\in (0,T^1]}, \chi^1)$ and $((\tilde g^2(t))_{t\in (0,T^2]}, \chi^2)$ for $g^1$ and $g^2$ respectively, there is a stationary diffeomorphism $\alpha:M_2 \to M_1$ such that $\alpha^*\tilde g^1(t) = \tilde g^2(t)$ for all $t\in (0,\min\{T^1,T^2\}]$ and $\chi^1\circ\alpha = \varphi \circ \chi^2$; this is Corollary \ref{cor:metricuniqueness}.
\end{remark}

We now briefly discuss the role of this paper in the context of possible definitions of scalar curvature bounded below for $C^0$ metrics. Let $M$ be a closed manifold. We define the following classes of $C^0$ metrics on $M$, each of which is a class of $C^0$ metrics that have scalar curvature bounded below by $\kappa$ in some reasonable generalized sense. Let $C^0_{\beta}(M,\kappa)$ denote the space of $C^0$ metrics on $M$ that satisfy Definition \ref{def:RRFscalar} everywhere in $M$, for a given value of $\beta$. Let $C^0_{\met}(M, \kappa)$ denote the $C^0$-completion of $C^2$-metrics whose scalar curvature is bounded below by $\kappa$ in the classical sense, i.e. a $C^0$ metric $g$ on $M$ is an element of $C^0_{\met}(M, \kappa)$ if and only if there exists a sequence of $C^2$ metrics $g^i$ on $M$ such that the $g^i$ converge uniformly to $g$ and satisfy $R(g^i) \geq \kappa$. Define $C^0_{\Rf}(M,\kappa)$ to be the space of $C^0$ metrics on $M$ whose corresponding regularizing Ricci flows have scalar curvature bounded below by $\kappa$ for all positive times, i.e. $g\in C^0_{\Rf}(M,\kappa)$ if and only if for all regularizing Ricci flows $(\tilde g_t, \chi)$ for $g$ we have $R(\tilde g(t)) \geq \kappa$ everywhere for all $t>0$ for which the flow is defined. This is a natural class of metrics to study in light of \cite{Bam}. Finally, let $C^0_{\G}(M, \kappa)$ denote the space of $C^0$ metrics on $M$ that have scalar curvature bounded below by $\kappa$ in the sense of Gromov's paper \cite{Gro}. 

Gromov's formulation is essentially that a Riemannian manifold $(M,g)$ has nonnegative scalar curvature if it does not contain a cube with strictly mean convex faces, such that the dihedral angles are acute, or more generally, that it has scalar curvature bounded below by $\kappa$ if its product with an appropriate space form has nonnegative scalar curvature in the same weak sense; see \cite[p. $1119$]{Gro}. This is a natural definition because, if such a cube were to exist in a manifold with (classical) nonnegative scalar curvature, Gromov has proposed (see \cite[pp. $1144-1145$]{Gro}) that one could glue together $2n$ copies of the cube and obtain a non-flat torus of nonnegative scalar curvature, contradicting the scalar torus rigidity theorem (\cite[Corollary 2]{SY} and \cite[Corollary A]{GL}).

We now discuss the question of equivalence of these different definitions of lower scalar curvature bounds for $C^0$ metrics. Certainly we have $C^0_{\Rf}(M,\kappa) \subset C^0_{\beta}(M,\kappa)$. Moreover, Theorem \ref{thm:introthm2} implies that $C^0_{\beta}(M,\kappa) \subset C^0_{\Rf}(M,\kappa)$, so $C^0_{\Rf}(M,\kappa) = C^0_{\beta}(M,\kappa)$.

By Corollary \ref{cor:C2approximation}, $C^0_{\beta}(M,\kappa)\subset C^0_{\met}(M,\kappa)$. Moreover, Theorem \ref{thm:kappaapproximation} implies that $C^0_{\met}(M,\kappa)\subset C^0_{\Rf}(M,\kappa)$. Thus we have that $C^0_{\met}(M,\kappa) = C^0_{\Rf}(M,\kappa) = C^0_{\beta}(M,\kappa)$.

That \cite{Bam} provides a Ricci flow proof of Gromov's $C^0$-limit Theorem \cite[Page $1118$]{Gro} suggests a relationship between $C^0_{\Rf}(M,\kappa)$ and $C^0_{\G}(M,\kappa)$. We have that $C^0_{\met}(M,\kappa) \subset C^0_{\G}(M,\kappa)$, since $C^0_{\G}(M,\kappa)$ contains all $C^2$ metrics $g$ with $R(g)\geq \kappa$ and is closed in $C^0$. Thus, $C^0_{\Rf}(M,\kappa) \subset C^0_{\G}(M,\kappa)$. It is an open question whether, if a  $C^0$ metric on a closed manifold has scalar curvature bounded below in the sense of \cite{Gro}, it necessarily has scalar curvature bounded below under the Ricci flow:

\begin{question*}
Suppose that $M$ is closed. Is $C^0_{\G}(M,\kappa)\subset C^0_{\Rf}(M,\kappa)$?
\end{question*}

One feature of Gromov's definition is that it may be localized around a point $x\in M$, by requiring only that there exist a neighborhood of $x$ such that no cube within the neighborhood that contains $x$ has strictly mean convex faces and acute dihedral angles; see \cite[p. $1144$]{Gro}. In light of this, for $x\in M$ define $C^0_{\G}(x, \kappa)$ to be the space of germs of $C^0$ metrics on $M$ at $x$ that have scalar curvature bounded below by $\kappa$ at $x$ in the sense of \cite{Gro}. In particular, \cite{Bam} suggests that there should be a way of localizing $C^0_{\Rf}(M,\kappa) = C^0_{\beta}(M,\kappa)$. Define $C^0_{\beta}(x,\kappa)$ to be the space of germs of $C^0$ metrics on $M$ at $x$ that have scalar curvature bounded below by $\kappa$ at $x$ in the sense of Definition \ref{def:RRFscalar}. Theorem \ref{cor:introcor1} suggests that this is a reasonable of localization of $C^0_{\Rf}(M,\kappa)$, because if a metric satisfies Definition \ref{def:RRFscalar} at all points for a uniform value of $\beta$, then it is in $C^0_{\Rf}(M,\kappa)$

\begin{question*}
Let $M$ be a closed manifold and $x\in M$. Do we have $C^0_{\G}(x, \kappa) = C^0_{\beta}(x,\kappa)$?
\end{question*}

Moreover, towards the aim of showing that Definition \ref{def:RRFscalar} is equivalent to Gromov's, we might ask:
\begin{question*}
Suppose $g$ is a $C^0$ metric on a closed manifold $M$ and that the germ of $g$ at $x$ is in $C^0_{\G}(x,\kappa)$. Suppose that $g'$ is another $C^0$ metric on $M$ that agrees with $g$ to greater than second order about $x$. Is the germ of $g'$ at $x$ also in $C^0_{\G}(x,\kappa)$?
\end{question*}

Theorem \ref{cor:introcor1} says that the perturbation of a metric by greater than second order does not affect $\beta$-weak lower bounds on the scalar curvature. Thus, it is natural to ask whether one can characterize those metrics, up to higher order perturbation, that have nonnegative scalar curvature in the sense of Definition \ref{def:RRFscalar}:

\begin{question*}
Suppose $g = g_{ij}dx^i\otimes dx^j$ is a metric on a neighborhood of the origin in $\R^n$, and that we may write
\begin{equation}
g_{ij}(x) = \delta_{ij} + r^2G_{ij}(\frac{x}{r}) + O(|x|^{2+\eta}),
\end{equation}
where the $G_{ij}$ are functions on $\mathbb{S}^{n-1}\subset \R^n$ satisfying $x^ix^jG_{ij}(x) = 0$. Is there an explicit characterization of metrics of this form that have nonnegative scalar curvature at the origin, in the sense of Definition \ref{def:RRFscalar}?
\end{question*}

We now explain the structure of the rest of the paper.

 In \S \ref{sec:preliminaries} we recall some facts about Ricci and Ricci-DeTurck flow, and the evolution of scalar curvature under these flows. We then state some estimates for the heat kernel on a Ricci flow background, and derive relevant bounds for the heat kernel on a Ricci-DeTurck flow background. Finally, we record some useful analytic facts.

In \S \ref{sec:initialderivatives} we define appropriate vector spaces and use an argument of Koch and Lamm to show that there exists a solution to the integral equation for the Ricci-DeTurck flow from $C^0$ initial data, and that two solutions with close initial data remain close for positive times. We also apply parabolic interior estimates to show that one may take these solutions to be smooth for positive times, and that time slices of such solutions converge uniformly to their initial data as $t\to 0$.

In \S \ref{sec:fixedptconstruction} we use weighted norms to study the stability under Ricci-DeTurck flow of the difference of two $C^0$ metrics which initially agree to greater than second order. We then study the behavior of their second derivatives as $t$ tends to $0$ and (essentially) prove Theorem \ref{cor:introcor1}.

In \S \ref{sec:RRF} we prove Theorem \ref{thm:introthm0} and further discuss Theorem \ref{cor:introcor1}.

 In \S \ref{sec:equivalentdefs} we provide several equivalent formulations of Definition \ref{def:RRFscalar} and show invariance of the definition under greater than second order perturbation of the metric. We also show that Definition \ref{def:RRFscalar} is independent choice of $y$ and choice of regularizing Ricci flow. Finally, we show that Definition \ref{def:RRFscalar} descends to the space of germs of $C^0$ metrics on $M$.

In \S \ref{sec:maxprinciple} we prove Theorem \ref{thm:introthm2}, Theorem \ref{thm:kappaapproximation}, and Corollary \ref{cor:torusrigidity}.

\subsection*{Acknowledgements}
I would like to thank my advisor, Richard Bamler, for introducing me to this project, and for his help and encouragement. I would also like to thank Christina Sormani, for posing the problem of torus rigidity for Definition \ref{def:RRFscalar} to me, and for showing relevant references to me. Finally, I would like to thank Chao Li, for showing the work of Simon \cite{Sim} to me, and for many helpful comments on a previous version of this paper.

This material is based upon work supported by the National Science Foundation Graduate Research Fellowship Program under Grant No. DGE $1752814$. Any opinions, findings, and conclusions or recommendations expressed in this material are those of the author(s) and do not necessarily reflect the views of the National Science Foundation.

\section{Preliminaries}\label{sec:preliminaries}

\subsection{Ricci and Ricci-DeTurck flow}\label{sec:RFandRDT}
If $M$ is a smooth manifold and $g_0$ is a smooth Riemannian metric on $M$, the Ricci flow is a solution to
\begin{equation}
\begin{cases}
\partial_t \bar g&= -2\Ric(\bar g) \text{ in } M \times (0,T)\\
\bar g(0) &= g_0,
\end{cases}
\end{equation}
where $\bar g(t)$ is a smooth, time-dependent family of Riemannian metrics on a $M$ for $t\in (0,T)$. Moreover, if $M$ is closed and $g_0$ is smooth, then a short-time solution to the Ricci flow always exists and is unique; see \cite[Theorems $5.2.1$, $5.2.2$]{Top}. If $\bar g(t) = \bar g_t$ is a Ricci flow, then the parabolically rescaled family of metrics $\hat{\bar g}(x,t) := \lambda \bar g(x, t/\lambda)$ also solves the Ricci flow equation. Moreover, if $\bar g(t)$ is a Ricci flow with $|\Rm(\bar g(t))| \leq C$ on $[0,T]$, then 
\begin{equation}\label{eq:distancedistortion}
e^{-2Ct}\bar g(0) \leq \bar g(t) \leq e^{2Ct}\bar g(0)
\end{equation}
for all $t\in[0,T]$; see \cite[Lemma $5.3.2$]{Top}.

By taking $T$ sufficiently small depending on the flow $(\bar g(t))_{t\in [0,T]}$ we may assume that
\begin{equation}\label{eq:almosteucl}
\frac{1}{10}\omega_n r^n \leq \vol_{\bar g_t}(B_{\bar g_t}(x, r)) \leq 10\omega_n r^n
\end{equation}
for all $x\in M$ and all $r\leq \sqrt{T}$, where $\omega_n$ denotes the volume on the unit ball in $n$-dimensional Euclidean space.

Moreover, we may also take $T$ sufficiently small depending on the flow so that around every point there is an exponential coordinate ball for $\bar g_0$ of radius $2\sqrt{T}$, and in any such coordinate ball we have
\begin{equation}\label{eq:boundedcoeffs}
|\partial^m(\bar g^{ij}(t) - \delta^{ij})| \leq 100
\end{equation}
for any multiindex $m$ with $|m| \leq 3$ and all $t\leq T$.

We will also be concerned with the Ricci-DeTurck flow, first introduced by DeTurck in \cite{De}, which is related to the Ricci flow via pullback by a family of diffeomorphisms, and depends on a choice of background metric. Throughout this paper, we consider Ricci-DeTurck flows on a Ricci flow background. We now recall some facts about this setting, which may be found (in a more general setting) in \cite[Appendix A]{BK}. 

Define the following operator, which maps symmetric $2$-forms on $M$ to vector fields:
\begin{equation}\label{eq:Xoperator}
X_{\bar g}(g):= \sum_{i=1}^n(\nabla^{\bar g}_{e_i}e_i - \nabla^{g}_{e_i}e_i),
\end{equation}
where $\{e_i\}_{i=1}^n$ is any local orthonormal frame with respect to $g$. Then the Ricci-DeTurck equation is
\begin{equation}\label{eq:RDTeq}
\partial_t g(t) = -2\Ric(g(t)) - \L_{X_{\bar g(t)}(g(t))}g(t),
\end{equation}
where $\bar g(t)$ is a background Ricci flow. As mentioned, if $g(t)$ solves (\ref{eq:RDTeq}) then it is related to a Ricci flow via pullback by diffeomorphisms. More precisely, if $g(t)$ solves (\ref{eq:RDTeq}) and $\chi_t: M\to M$ is a family of diffemorphisms satisfying
\begin{equation}\label{eq:diffeoseq}
\begin{cases}
X_{\bar g(t)}(g(t))f &= \frac{\partial}{\partial t}(f\circ\chi_t) \text{ for all } f\in C^\infty(M)\\
\chi_0 &= \id,
\end{cases}
\end{equation}
 then $\chi_t^*g(t)$ solves the Ricci flow equation with initial condition $g(0)$.
If $g_t$ is a solution to the Ricci-DeTurck equation, and we write $g_t = \bar g_t + h_t$, where $\bar g_t$ is again a (smooth) background Ricci flow, then the evolution equation for $h_t$ is
\begin{equation}\label{eq:hevolution}
\partial_t h_t + Lh_t = Q[h_t],
\end{equation}
where the linear part, $L$, is 
\begin{equation}\label{eq:Lis}
L h_t := \Delta^{\bar g_t}h_t + 2\Rm^{\bar g_t}(h_t) := \Delta^{\bar g_t}h_t + 2{\bar g}^{pq}R_{pij}^{m}h_{q m}dx^i\otimes dx^j,
\end{equation}
and $Q$ denotes the quadratic term 
\begin{equation}
\begin{split}
\left(Q_{\bar g_t}[h_t]\right)_{ij}&:= \left((\bar g+h)^{pq} - \bar g^{pq}\right)\left(\nabla^2_{pq}h_{ij} +R_{pij}^mh_{mq} + R_{pji}^mh_{mq}\right)
\\& \qquad + \left(\bar g^{pq} - (\bar g+h)^{pq}\right)\left(R_{ipq}^mh_{mj} + R_{jpq}^mh_{im}\right)
\\& -\frac{1}{2}(\bar g+h)^{pq}(\bar g+h)^{m\ell}\big(-\nabla_i h_{pm}\nabla_jh_{q\ell} - 2\nabla_mh_{ip}\nabla_qh_{j\ell}
\\& \qquad + 2\nabla_mh_{ip}\nabla_{\ell}h_{jq} + 2\nabla_ph_{i\ell}\nabla_jh_{qm} + 2\nabla_ih_{pm}\nabla_qh_{j\ell}\big)
\\&= \nabla_p (\left((\bar g+h)^{pq} - \bar g^{pq}\right)\nabla_qh_{ij}) 
\\& \qquad - \left(\nabla_p\left((\bar g+h)^{pq} - \bar g^{pq}\right)\right)\nabla_qh_{ij} +  \left((\bar g+h)^{pq} - \bar g^{pq}\right)\left(R_{pij}^mh_{mq} + R_{pji}^mh_{mq}\right)
\\& \qquad + \left(\bar g^{pq} - (\bar g+h)^{pq}\right)\left(R_{ipq}^mh_{mj} + R_{jpq}^mh_{im}\right)
\\& -\frac{1}{2}(\bar g+h)^{pq}(\bar g+h)^{m\ell}\big(-\nabla_i h_{pm}\nabla_jh_{q\ell} - 2\nabla_mh_{ip}\nabla_qh_{j\ell}
\\& \qquad + 2\nabla_mh_{ip}\nabla_{\ell}h_{jq} + 2\nabla_ph_{i\ell}\nabla_jh_{qm} + 2\nabla_ih_{pm}\nabla_qh_{j\ell}\big),
\end{split}
\end{equation}
where $\nabla$ denotes the covariant derivative with respect to $\bar g_t$, and the last equality follows from the Leibniz rule.
Moreover, we may write
\begin{equation}
Q[h_t] = Q^0_t + \nabla^* Q^1_t,
\end{equation}
where
\begin{equation}\label{eq:Q0is}
\begin{split}
Q^0_t&:=-\frac{1}{2}(\bar g+h)^{pq}(\bar g+h)^{m\ell}\big(-\nabla_i h_{pm}\nabla_jh_{q\ell} - 2\nabla_mh_{ip}\nabla_qh_{jm}
\\& \qquad + 2\nabla_mh_{ip}\nabla_{\ell}h_{jq} + 2\nabla_ph_{i\ell}\nabla_jh_{qm} + 2\nabla_ih_{pm}\nabla_qh_{j\ell}\big)
\\& \qquad - \left(\nabla_p((\bar g+h)^{pq} - \bar g^{pq})\right)\nabla_qh_{ij} +  \left((\bar g+h)^{pq} - \bar g^{pq}\right)\left(R_{pij}^mh_{mq} + R_{pji}^mh_{mq}\right)
\\& = (\bar g + h)^{-1}\star (\bar g + h)^{-1} \star \nabla h \star \nabla h + ((\bar g + h)^{-1} - \bar g^{-1})\star \Rm^{\bar g_t}\star h
\end{split}
\end{equation}
and
\begin{equation}\label{eq:Q1is}
\nabla^*Q^1_t:= \nabla_p (\left((\bar g+h)^{pq} - \bar g^{pq}\right)\nabla_qh_{ij}) = \nabla(((\bar g + h)^{-1} - \bar g^{-1})\star \nabla h),
\end{equation}
where here we use the notation $A\star B$ for two tensor fields $A$ and $B$ to mean a linear combination of products of the coefficients of $A$ and $B$, and $(\bar g + h)^{-1}$ and $\bar g^{-1}$ denote tensor fields with coefficients $(\bar g + h)^{ij}$ and $\bar g^{ij}$ respectively. For any $(0,2)$-tensors $\bar g$ and $h$ satisfying $|h|_{\bar g} \leq \gamma<1$, we have
\begin{equation}\label{eq:inverseexpansion1}
|(\bar g + h)^{-1}|_{\bar g} \leq c(n)(1 - |h|_{\bar g})^{-1} \leq c(n,\gamma),
\end{equation}
\begin{equation}\label{eq:inverseexpansion2}
|(\bar g + h)^{-1} - \bar g^{-1}|_{\bar g} \leq c(n)|h|_{\bar g}/(1-|h|_{\bar g}) \leq c(n,\gamma)|h|_{\bar g},
\end{equation}
and
\begin{equation}\label{eq:inverseexpansion3}
|(\bar g + h')^{-1} - (\bar g + h'')^{-1}| \leq |(\bar g + h')^{-1}|||h' - h''||(\bar g + h'')^{-1}|
\end{equation}
To see why this is true, fix a point in $M$, and choose coordinates about this point so that $\bar g_{ij} = \delta_{ij}$. Then, using the usual expansion for matrices, we find
\begin{equation*}
|(\bar g + h)^{-1}| \leq c(n)\sum_{i=0}^{\infty}|h|^i \leq c(n)\frac{1}{1 - |h|},
\end{equation*}
and
\begin{equation*}
|(\bar g + h')^{-1} - (\bar g + h'')^{-1}| = |[\bar g - (\bar g + h')^{-1}(\bar g + h'')](\bar g + h'')^{-1}| = |(\bar g + h')^{-1}[(\bar g + h') - (\bar g + h'')](\bar g + h'')^{-1}|.
\end{equation*}

We make use of the following pointwise estimates for $Q^0$ and $Q^1$: if $|h|_{\bar g} \leq \gamma < 1$ then (\ref{eq:inverseexpansion1}) and (\ref{eq:inverseexpansion2}) imply
\begin{equation}\label{eq:ptwiseR0}
|Q_t^0| \leq c(n,\gamma)(|\nabla h|^2 + |\Rm^{\bar g_t}||h|^2)
\end{equation}
and
\begin{equation}\label{eq:ptwiseR1}
|Q_t^1|\leq c(n,\gamma)|h||\nabla h|.
\end{equation}

We will now bound the amount that the diffeomorphisms $\chi_t$ which solve the differential equation in (\ref{eq:diffeoseq}) perturb the points of $M$ (cf. \cite[Lemma A.$18$]{BK}):

\begin{lemma}\label{lemma:driftbound}
Assume that $g(s)$ and $\bar g(s)$ are smooth families of Riemannian metrics in the setting above, defined for $t\in (0,T]$ and $t\in [0,T]$ respectively. Suppose that $\chi_s$ are a family of diffeomorphisms solving the differential equation from (\ref{eq:diffeoseq}). Suppose also that on this time interval there is some constant $c_0$ such that $|\nabla^{\bar g} g_s|_{\bar g} \leq c_0s^{-1/2}$. Then
\begin{equation}\label{eq:Xbound}
|X_s|_{\bar g_t}\leq \frac{c}{\sqrt{s}}
\end{equation}
for any $t\in[0,T]$, and, for all $p\in M$ and all $0<t_1<t_2$, we have
\begin{equation}\label{eq:distancesnotmuchmoved}
d_{\bar g(t_1)}(\chi_{t_1}(p), \chi_{t_2}(p)) \leq c(\sqrt{t_2}-\sqrt{t_1}),
\end{equation}
where $c$ depends on $c_0$, $T\sup_{[0,T]}|\Rm|(\bar g)$, and the dimension.
\end{lemma}

\begin{proof}
We first bound the norm of the operator from (\ref{eq:Xoperator}), $|X_s|_{\bar g_s}$, at a point $p\in M$. Let the $(e_i)$ be normal coordinates for $g_s$ at $p$. We have
\begin{align*}
\bar g_s\left(\sum_{i}\nabla^{\bar g_s}_{e_i}e_i - \nabla^{g_s}_{e_i}e_i, \sum_j\nabla^{\bar g_s}_{e_j}e_j - \nabla^{g_s}_{e_j}e_j\right) &= \sum_{i,j} \bar g_s(\nabla^{\bar g_s}_{e_i}e_i, \nabla^{\bar g_s}_{e_j}e_j) - 2 \bar g_s(\nabla^{\bar g_s}_{e_i}e_i, \nabla^{g_s}_{e_j}e_j) + \bar g_s(\nabla^{g_s}_{e_i}e_i, \nabla^{g_s}_{e_j}e_j)
\\& =: \sum_{i,j}I + II + III.
\end{align*}
We estimate each of these terms separately. We have
\begin{align*}
I = \bar g_s(\nabla^{\bar g_s}_{e_i}e_i, \nabla^{\bar g_s}_{e_j}e_j)&= \bar g_s(\bar\Gamma_{ii}^k e_k, \bar\Gamma_{jj}^\ell e_{\ell}) 
\\& = \sum_{k=1}^{n}\bar\Gamma_{ii}^k\bar\Gamma_{jj}^{k} \leq \frac{c}{s},
\end{align*}
where $\bar \Gamma_{ab}^c$ are the Christoffel symbols with respect $\bar g$ of the $(e_i)$. If $\Gamma_{ab}^c$ denote the Christoffel symbols with respect to $g$ of the $(e_i)$, then
\begin{align*}
II = \bar g_s(\nabla^{\bar g_s}_{e_i}e_i, \nabla^{g_s}_{e_j}e_j) &= \partial_i\bar g_s(e_i,\nabla^{g_s}_{e_j}e_j) - \bar g_s(e_i, \nabla^{\bar g_s}_{e_i}\nabla^{g_s}_{e_j}ej)
\\& = \partial_i\left(\Gamma_{jj}^k\bar g_s(e_i,e_k)\right) - \bar g_s(e_i, \nabla_i^{\bar g_s}\Gamma_{jj}^k e_k)
\\& = \partial_i\left(\Gamma_{jj}^k\bar g_s(e_i,e_k)\right) - \bar g_s(e_i, \partial_i\Gamma_{jj}^ke_k + \Gamma_{jj}^k\bar\Gamma_{ik}^{\ell}e_\ell)
\\& = \partial_i\Gamma_{jj}^i- \partial_i\Gamma_{jj}^i = 0.
\end{align*}
Finally,
\begin{align*}
III = \bar g_s(\nabla^{g_s}_{e_i}e_i, \nabla^{g_s}_{e_j}e_j) = \bar g_s( \Gamma_{ii}^ke_k, \Gamma_{jj}^{\ell}e_{\ell}) = 0,
\end{align*}
so
\begin{equation}
|X_s|_{\bar g_s} \leq \frac{c}{\sqrt{s}}.
\end{equation}
Then (\ref{eq:distancedistortion}) implies (\ref{eq:Xbound}).

We now estimate the drift. Fix some $p\in M$, and write $p = \chi_s^{-1}(q)$ for some $q$. Then, by \cite[(A.9)]{BK}, we have
\begin{equation*}
|\partial_s\chi_s(p)|_{\bar g_t} = |(\partial_s\chi_s)(\chi_s^{-1}(q))|_{\bar g_t} = |X_{\bar g_s}(g_s)(q)|_{\bar g_t} \leq \frac{c}{\sqrt{s}}.
\end{equation*}.
Therefore, for $0<t_1<t_2$ and $p\in M$,
\begin{align*}
d_{\bar g(t_1)}(\chi_{t_1}(p), \chi_{t_2}(p)) &\leq \int_{t_1}^{t_2}\left|\frac{\partial}{\partial s}d_{\bar g_{t_1}}(\chi_{t_1}(p), \chi_s(p))\right|_{\bar g_{t}}ds
\\& \leq \int_{t_1}^{t_2}|\nabla^{\bar g_{t_1}}d_{\bar g_{t_1}}(\chi_{t_1}(p),\cdot)|_{\bar g_t}|\partial_s\chi_s(p)|_{\bar g_{t}}ds
\\& \leq c(n)\int_{t_1}^{t_2}|\partial_s\chi_s(p)|_{\bar g_{t}}ds \leq c\int_{t_1}^{t_2}\frac{1}{\sqrt{s}}ds = c(\sqrt{t_2}-\sqrt{t_1}).
\end{align*}
\end{proof}

\subsection{Maximum principle and evolution of the scalar curvature under Ricci and Ricci-DeTurck flow}

The scalar curvature under the Ricci flow evolves by
\begin{equation}
\partial_t R = \Delta^{\bar g(t)} R + 2|\Ric|^2;
\end{equation}
see \cite[Proposition $2.5.4$]{Top}.
Making an orthogonal decomposition, we may conclude that
\begin{equation}\label{eq:diffineqR}
\partial_t R \geq \Delta^{\bar g(t)}R + \frac{2}{n}R^2;
\end{equation}
this is \cite[Corollary $2.5.5$]{Top}.

If instead, $g(t)$ is a Ricci-DeTurck flow, then recall that $g(t) = (\chi_t^{-1})^*\bar g(t)$ for some Ricci flow $\bar g(t)$, where the family $(\chi_t)$ satisfies (\ref{eq:diffeoseq}) and $X$ is the corresponding vector field. Therefore, pushing forward (\ref{eq:diffineqR}) by $(\chi_t)$ we find that, under the Ricci-DeTurck flow,
\begin{equation}\label{eq:RDTRevolution}
\partial_t R \geq \Delta^{g(t)}R - \langle X, \nabla R\rangle + \frac{2}{n}R^2;
\end{equation}
see also \cite[p.$6$]{Bam}. It follows that, if $g(t)$ is a Ricci or Ricci-DeTurck flow on a closed manifold, starting from a smooth initial metric, and defined on the interval $[0,T]$, then if $R^{g_0} \geq \kappa$, we have (cf. \cite[Theorem $3.2.1$]{Top})
\begin{equation}\label{eq:Rpreservation}
R^{g(t)} \geq \frac{\kappa}{1 - \left(\frac{2\kappa}{n}t\right)} \text{ for all } t\in [0,T].
 \end{equation} 
 A more general bound that follows from the maximum principle is the following (cf. \cite[Corollary $3.2.5$]{Top}): suppose $g(t)$ is a Ricci or Ricci-DeTurck flow on a closed manifold, for $t\in (0,T]$. Then
\begin{equation}\label{eq:scalarlowerbound}
R \geq -\frac{n}{2t}
\end{equation}
for all $t\in (0,T]$.

\subsection{Heat kernel for the Ricci flow}

If $E$ is a vector bundle over $M$ and $L$ is a second order linear differential operator acting on sections of $E$, then the heat kernel associated to $L$ is the family $(k_t) \in C^{\infty}(M\times M \times \R_+\to E\otimes E^*)$ satisfying 
\begin{equation}
(\partial_t + L)k_t(\cdot, y) = 0, \qquad \lim_{t\searrow 0}k_t(\cdot, y) = \delta_y\id_{E_y}.
\end{equation}
for all $y\in M$. 

 If $h(x,t)$ solves
\begin{equation}
(\partial_t + L)h_t = Q[h_t], \qquad h(0) = h_0
\end{equation}
and $k_t(x,y)$ is the corresponding heat kernel on the background $\bar g(t)$,
then we have the representation formula
\begin{equation}\label{eq:representationformula}
h_t(x) = \int_M k_t(x,y)h_0(y)d_{\bar g_0}(y) + \int_{M\times[0,t]}k_{t-s}(x,y)Q[h_s](y)d_{\bar g_s}(y)ds.
\end{equation}
We will be interested in the cases where $E = \Sym^2(T^*M)$ and $E = M\times \R$, and $L$ is an operator on a Ricci or Ricci-DeTurck flow background. We will sometimes write $k(x,t;y,s)$ for $k_{t-s}(x,y)$, where $s\leq t$. Henceforth we will denote by $\bar \Phi$ the scalar heat kernel for the operator $\partial_t - \Delta$ on a Ricci flow background.

It is a computation to show:
\begin{lemma}\label{lemma:rescaledheatkernel}
Let $\bar g(t)$ be a Ricci flow on a manifold $M$ with complete time slices, defined on $[T_1, T_2]$. Now consider the rescaled Ricci flow on $M$, $\hat{\bar g}(t) := \lambda^{-1}g(\lambda t + T_1)$, which is a Ricci flow defined for $t\in [0,\lambda^{-1}(T_2 - T_1)]$, as discussed in \S\ref{sec:RFandRDT}. Let $\bar \Phi$ be the heat kernel for the Ricci flow background, and $\hat{\bar \Phi}$ be the heat kernel for the rescaled Ricci flow background $\hat{\bar g}(t)$. Then
\begin{equation}
\hat{\bar \Phi}(x,t; y,s) = \lambda^{n/2}\bar \Phi(x,\hat t, y, \hat s),
\end{equation}
where $\hat t:= \lambda t + T_1$ and $\hat s := \lambda s + T_1$.
\end{lemma}

The following is a consequence of \cite[Lemma $26.23$]{CCG+} and is essentially \cite[Theorem $26.25$]{CCG+} for the Ricci flow:
\begin{theorem}\label{thm:heatkernelbound}
Suppose that $\bar g(t)$ is a Ricci flow on $M\times [0,T]$ with complete time slices that satisfies $\sup_{[0,T]}|\Ric|< \infty$. Then there exist constants $C = C(n,T\sup|\Ric|)< \infty$ and $D = D(T \sup|\Ric|) < \infty$ such that, if $\bar \Phi$ is the heat kernel on the background $\bar g(t)$, then
\begin{equation}
\bar \Phi(x,t; y,s) \leq \frac{C\exp\left(\frac{-d^2_{\bar g(0)}(x,y)}{D(t-s)}\right)}{\vol_{\bar g(0)}^{1/2}B_{\bar g(0)}\left(x,\sqrt{\frac{t-s}{2}}\right)\vol_{\bar g(0)}^{1/2}B_{\bar g(0)}\left(y,\sqrt{\frac{t-s}{2}}\right)}.
\end{equation}
\end{theorem}
\begin{proof}
If we allow the constants to have the dependencies $C = C(n, T, \sup|\Ric|)$ and $D = D(T,\sup|\Ric|)$ (i.e. the constants may depend on the individual values of $T$ and $\sup |\Ric|$ rather than the product) then proof is identical to that of \cite[Theorem $26.25$]{CCG+}.

Now suppose that we have two Ricci flows $\bar g'$ and $\bar g''$ on $M$ that have complete times slices on $[0,T']$ and $[0,T'']$ respectively, such that $T'\sup_{[0,T']}|\Ric(\bar g')| = T''\sup_{[0,T'']}|\Ric(\bar g'')|$. Denote by $\bar \Phi'$ and $\bar \Phi''$ the heat kernels of $\bar g'$ and $\bar g''$ respectively. We now use the notation of Lemma \ref{lemma:rescaledheatkernel}. We take $T_1 = 0$ and $\lambda = T''/T'$, and parabolically rescale $\bar g''$ by $\lambda^{-1}$ as described in Lemma \ref{lemma:rescaledheatkernel} to obtain the Ricci flow $\hat{\bar g}''$ with heat kernel $\hat{\bar \Phi}''$. The rescaled flow $\hat{\bar g}''$ is defined on $[0, \lambda^{-1}T''] = [0,T']$ and satisfies 
\begin{equation*}
\sup_{[0,T']}|\Rm|(\hat{\bar g}'') = \frac{T''\sup_{[0,T'']}|\Rm|(\bar g'')}{T'} = \frac{T'\sup_{[0,T']}|\Rm|(\bar g')}{T'} = \sup_{[0,T']}|\Rm|(\bar g').
\end{equation*}

As discussed in the first paragraph, there are constants $C = C(n, T, \sup|\Ric|)$ and $D = D(T,\sup|\Ric|)$ such that 
\begin{equation*}
\bar \Phi'(x,t; y,s) \leq \frac{C\exp\left(\frac{-d^2_{\bar g'(0)}(x,y)}{D(t-s)}\right)}{\vol_{\bar g'(0)}^{1/2}B_{\bar g'(0)}\left(x,\sqrt{\frac{t-s}{2}}\right)\vol_{\bar g'(0)}^{1/2}B_{\bar g'(0)}\left(y,\sqrt{\frac{t-s}{2}}\right)}.
\end{equation*}
Because $\hat{\bar g}''$ satisfies the same upper bonds on the time interval and the curvature as $\bar g'$, we have, for the same values of $C$ and $D$:
\begin{equation}
\lambda^{n/2}\bar\Phi''(x,\hat t, y, \hat s) = \hat{\bar \Phi}''(x,t;y,s)\leq \frac{C_3\exp\left(\frac{-d^2_{\hat{\bar g}''(0)}(x,y)}{C_4(t-s)}\right)}{\vol_{\hat{\bar g}''(0)}^{1/2}B_{\hat{\bar g}''(0)}\left(x,\sqrt{\frac{t-s}{2}}\right)\vol_{\hat{\bar g}''(0)}^{1/2}B_{\hat{\bar g}''(0)}\left(y,\sqrt{\frac{t-s}{2}}\right)}.
\end{equation}
where the first equality is due to Lemma \ref{lemma:rescaledheatkernel}. 
The rest is a computation. We have
\begin{equation}
d_{\hat{\bar g}''(0)}(x,y) = \sqrt{\lambda}^{-1}d_{\bar g''(0)}(x,y)
\end{equation}
for all $x,y\in M$. Therefore, 
\begin{equation*}\label{eq:initialkernelbound}
\exp\left(\frac{-d^2_{\hat{\bar g}''(0)}(x,y)}{D( t- s)}\right) = \exp\left(\frac{-d^2_{\bar g''(0)}(x,y)}{D(\hat t- \hat s)}\right)
\end{equation*}
and
\begin{equation*} 
B_{\hat{\bar g}''(0)}\left(x, \sqrt{\frac{( t - s)}{2}}\right) = B_{\bar g''(0)}\left(x, \sqrt{\frac{(\hat t - \hat s)}{2}}\right).
\end{equation*}
Therefore, 
\begin{align*}
\vol_{\hat{\bar g}''(0)}B_{\hat{\bar g}''(0)}\left(x, \sqrt{\frac{t - s}{2}}\right) &= \int_{B_{\hat{\bar g}''(0)}\left(x, \sqrt{\frac{t - s}{2}}\right) } \sqrt{\det((\hat{\bar g}''(0))_{ij})}dx^1 \wedge \ldots \wedge dx^n
\\& = \int_{B_{\bar g''(0)}\left(x, \sqrt{\frac{(\hat t - \hat s)}{2}}\right)} \lambda^{-n/2}\sqrt{\det((\bar g''(0))_{ij})}dx^1 \wedge \ldots \wedge dx^n
\\&=  \lambda^{-n/2}\vol_{\bar g''(0)}B_{\bar g''(0)}\left(x, \sqrt{\frac{\hat t - \hat s}{2}}\right).
\end{align*}
Inserting these computations into (\ref{eq:initialkernelbound}), we find
\begin{equation}
\bar \Phi''(x,\hat t; y,\hat s) \leq  \frac{C\exp\left(\frac{-d^2_{\bar g''(0)}(x,y)}{D(\hat t-\hat s)}\right)}{\vol_{\bar g''(0)}^{1/2}B_{\bar g''(0)}\left(x,\sqrt{\frac{\hat t-\hat s}{2}}\right)\vol_{\bar g''(0)}^{1/2}B_{\bar g''(0)}\left(y,\sqrt{\frac{\hat t-\hat s}{2}}\right)}.
\end{equation}
for the same constants $C$ and $D$.
\end{proof}

\begin{lemma}\label{lemma:timeintervalhkb}
Fix $t_0>0$ and let $\bar g(t)$ be a Ricci flow with complete time slices on the time interval $[\tfrac{t_0}{2}, t_0]$. Suppose also that $\bar g(t)$ satisfies a curvature bound of the form $|\Rm|\leq \tfrac{c}{t}$. There exist constants $C = C(n,c) < \infty$ and $D(c) < \infty$ such that for all $x, y\in M, s< t\in [\tfrac{t_0}{2}, t_0]$, we have
\begin{equation}
{\bar \Phi}(x,t;y,s)\leq \frac{C\exp\left(\frac{-d^2_{\bar g(t_0/2)}(x,y)}{D(t-s)}\right)}{\vol_{\bar g(t_0/2)}^{1/2}B_{\bar g(t_0/2)}\left(x,\sqrt{\frac{t-s}{2}}\right)\vol_{\bar g(t_0/2)}^{1/2}B_{\bar g(t_0/2)}\left(y,\sqrt{\frac{t-s}{2}}\right)}.
\end{equation}
\end{lemma}

\begin{proof}
This follows immediately from Theorem \ref{thm:heatkernelbound}, after shifting the time interval to begin at $0$.
\end{proof}

If instead we consider the representation formula (\ref{eq:representationformula}) for $(0,2)$-tensors, we may still derive exponential estimates on the corresponding heat kernel. To this end, we now record a version of parabolic interior estimates that will be useful for our setting.  We first define the following norms; see also \cite[Page $422$]{Bam2}: Fix $\alpha \in (0,\tfrac{1}{2})$. If $Q\subset \R^n\times [0,T]$ is a parabolic domain, and 
\begin{equation}
r:= \min\{r': Q\subset B(p, r')\times[t- (r')^2, t] \text{ for some } (p, t) \in M\times [0,T]\},
\end{equation}
then let
\begin{equation}
||u||_{C^{2m,2\alpha;m,\alpha}(Q)}:= \sum_{|\iota|+2k\leq 2m}r^{|\iota|+2k}(||\nabla^{\iota}\partial_t^ku||_{C^0(Q)} + r^{2\alpha}[\nabla^{\iota}\partial_t^k u]_{2\alpha,\alpha}).
\end{equation}
\begin{lemma}\label{lemma:interiorests}
For $k= 1, 2, 3$, there exists $c_k = c(n, \alpha, k)$ such that the following is true:

 If $(\bar g(t))_{t\in [0,T]}$ is a smooth Ricci flow and $T$ is sufficiently small so that (\ref{eq:boundedcoeffs}) holds, and if $u(x,t)$ is a time-dependent family of $(0,2)$-tensors satisfying $(\partial_t - \Delta^{\bar g(t)})u = 0$, then we have
 \begin{equation}
 ||\nabla^k u||_{C^0(B_{\bar g_0}(x, r) \times [t - r^2, t])} \leq c_kr^{-k}||u||_{C^0(B_{\bar g_0}(x, 2r)\times [t- 4r^2, t])}
 \end{equation}
 for all $r^2 < t \leq T$.
 
\end{lemma}
\begin{remark}
They key point here is that, by making $T$ small enough so that (\ref{eq:boundedcoeffs}) holds, the constants $c_k$ do not depend on the choice of Ricci flow $\bar g(t)$.
\end{remark}
\begin{proof}[Proof of Lemma \ref{lemma:interiorests}]
First suppose that (\ref{eq:boundedcoeffs}) is true for $T= 1$, and $B_{\bar g_0}(x, 1)$ is contained in an exponential coordinate chart for $\bar g_0$. Then the result is immediate from \cite[Theorem $8.12.1$]{Kry}, and is independent of $\bar g(t)$ since (\ref{eq:boundedcoeffs}) implies that $|\partial^m\bar g^{ij}| \leq 101$ for $|m| \leq 3$. Now fix arbitrary $0 < r \leq \sqrt{T}$ and $x\in M$. Pick exponential coordinates for $\bar g_0$ based at $x$, so that $B_{\bar g_0}(x,r)$ is contained within an exponential coordinate chart. We parabolically rescale $u$ and $\bar g$ by $r^{-2}$ to find that $\hat u$ and $\hat{\bar g}$ satisfy $(\partial_t - \Delta^{\hat{\bar g}(t)})\hat u = 0$. 
Moreover, by (\ref{eq:boundedcoeffs}), we have $|\partial^m\hat{\bar g}^{ij}(t)| \leq 101r^2 \leq 101$, so for the same constant $c$ we find
\begin{align*}
||\nabla^k u||_{C^{2,2\alpha;1,\alpha}(B_{\bar g_0}(x, r) \times [t - r^2, t])} &= ||\nabla ^k \hat u||_{C^{2,2\alpha;1,\alpha}(B_{\bar g_0}(x, 1) \times [tr^{-2} - 1, tr^{-2}])}
\\& \leq c||\hat u||_{C^0(B_{\bar g_0}(x,2)\times [tr^{-2}- 4, tr^{-2}])} = c||u||_{C^0(B_{\bar g_0}(x,2r)\times [t- 4r^2, t])}.
\end{align*}
\end{proof}

 Let $\bar K$ denote the heat kernel for $(0,2)$-tensors on a Ricci flow background corresponding to the differential operator $\partial_t + L$, where $L$ is the linear operator from (\ref{eq:hevolution}). If we assume that $\bar g(t)$ satisfies the curvature bound $|\Rm|(\bar g(t))\leq c$, then, by Kato's inequality, $\partial_t|\bar K_t(\cdot, y)| \leq \Delta|\bar K_t|(\cdot, y) + c|\bar K_t|(\cdot, y)$ in the barrier sense for all fixed $y\in M$, so $|\bar K_t(x,y)|\leq e^{ct}\bar \Phi_t(x,y)$, where $\bar \Phi$ denotes the scalar heat kernel for the Ricci flow on $M$, as before. We now record some estimates for $\bar K_t$ (cf. \cite[p.$32$]{Xphdthesis}).

\begin{corollary}\label{cor:exponentialboundnonscalar}
Let $\bar g(t)$ be a Ricci flow with complete time slices and heat kernel $\bar K$. Suppose $\bar g(t)$ satisfies the curvature bound $|\Rm|(\bar g(t)) \leq c$. Then, for $k= 0,1,2,3$, there exist $C_k = C(n,Tc, k)$ and $D = D(Tc)$ such that, for $t$ sufficiently small so that (\ref{eq:almosteucl}) and (\ref{eq:boundedcoeffs}) hold, we have
\begin{equation}\label{eq:ptwiseheatkernel}
\begin{split}
|\bar K_t|(x,y) &\leq C_0t^{-n/2}\exp\left(-\frac{d^2_{\bar g(0)}(x,y)}{Dt}\right) \text{ and }\\
|\nabla^k \bar K_t|(x,y) & \leq C_kt^{-(k+n)/2}\exp\left(-\frac{d^2_{\bar g(0)}(x,y)}{Dt}\right),
\end{split}
\end{equation}
where $\nabla$ is the covariant derivative with respect to $\bar g(t)$.
\end{corollary}

\begin{proof}
The first statement follows immediately from Theorem \ref{thm:heatkernelbound} and (\ref{eq:almosteucl}), since $|\bar K_t|(x,y) \leq e^{ct}|\bar \Phi_t(x,y)| < e^{cT}|\bar \Phi_t(x,y)|$. The estimates on the higher derivatives follow from Lemma \ref{lemma:interiorests}: for fixed $y\in M$, let $u(x,t) = \bar K_t(x,y)$, so that Lemma \ref{lemma:interiorests} implies
\begin{align*}
|\nabla^k u(x,t)| & \leq ||\nabla^k u||_{C^0(B_{\bar g_0}(x, \sqrt{t/8})\times[\tfrac{7t}{8}, t])} \leq ct^{-k/2}||u||_{C^0(B_{\bar g_0}(x, \sqrt{t/2})\times[\tfrac{t}{2}, t])}
\\& \leq Ct^{-k/2} ||\bar \Phi_s(z,y)||_{C^0((z,s)\in B_{\bar g_0}(x, \sqrt{t/2})\times[\tfrac{t}{2}, t])}
\\& \leq Ct^{-(k+n)/2}\exp\left(-\frac{d^2_{\bar g_0}(x,y)}{2Dt}\right),
\end{align*}
with $C$ adjusted, where the last line is as follows:  we have, by Jensen's inequality, $d_{\bar g_0}^2(z,y) \geq \tfrac{1}{2}d_{\bar g_0}^2(x,y) - d_{\bar g_0}^2(z,x)$, so for all $z\in B_{\bar g_0}(x, \sqrt{t/2})$, Theorem \ref{thm:heatkernelbound} and (\ref{eq:almosteucl}) imply
\begin{equation*}
\begin{split}
|\bar \Phi_s(z,y)| &\leq Ct^{-n/2}\exp\left(-\frac{d^2_{\bar g_0}(z,y)}{2Dt}\right) 
\\& \leq Ct^{-n/2}\exp\left(-\frac{d^2_{\bar g_0}(x,y)}{2Dt}\right)\exp\left(\frac{d^2_{\bar g_0}(z,x)}{4Dt}\right)  \leq Ct^{-n/2}\exp\left(-\frac{d^2_{\bar g_0}(x,y)}{2Dt}\right).
\end{split}
\end{equation*}
\end{proof}

\begin{lemma}
Let $\bar g(t)$ be a Ricci flow with complete time slices and heat kernel $\bar K$. Suppose $\bar g(t)$ satisfies the curvature bound $|\Rm|(\bar g(t)) \leq c$. Let $\nabla$ denote the covariant derivative with respect to $\bar g(t)$. Take $T$ sufficiently small so that (\ref{eq:almosteucl}) and (\ref{eq:boundedcoeffs}) hold. Then, for $k= 0, 1, 2, 3$ there exist $C_k = C(n,Tc, k)$ and $D = D(Tc)$ such that, for $t\leq T$, we have
\begin{equation}\label{eq:ptwiseheatkernelderivbounds}
|\nabla^k\bar K_t(x, y)| \leq C_k(d_{\bar g(0)}(x,y) + \sqrt{t})^{-n - k}
\end{equation}

\noindent Moreover, for $0< r \leq \sqrt{T}$, if $d_{\bar g(0)}(x,y) \geq r$ or $s \leq \tfrac{r^2}{2}$ we have
\begin{equation}\label{eq:ptwiseheatkernelawayfromball}
\begin{split}
|\bar K_{r^2 - s}|(x,y) &\leq C_0r^{-n}\exp\left(-\frac{d^2_{\bar g(0)}(x,y)}{2Dr^2}\right)\\
|\nabla \bar K_{r^2 - s}|(x,y) &\leq C_1r^{-n-1}\exp\left(-\frac{d^2_{\bar g(0)}(x,y)}{2Dr^2}\right)\\
|\nabla^2\bar K_{r^2 - s}|(x,y) &\leq C_2r^{-n-2}\exp\left(-\frac{d^2_{\bar g(0)}(x,y)}{2Dr^2}\right).
\end{split}
\end{equation}

\noindent Also,
\begin{equation}\label{eq:heatkernelawayfrom0}
\begin{split}
\int_{M\setminus B_{\bar g_0}(x,r)}|\bar K_t(x,y)|d_{\bar g_0}(y) &\leq C_0'\exp\left(-\frac{r^2}{2Dt}\right)\\
\int_{M\setminus B_{\bar g_0}(x, r)}|\nabla \bar K_t(x,y)|d_{\bar g_0}(y) &\leq \frac{C_1'}{\sqrt{t}}\exp\left(-\frac{r^2}{2Dt}\right),
\end{split}
\end{equation}
where $C_k' = C'(n,Tc, k)$.
\end{lemma}

\begin{proof}
The bounds (\ref{eq:ptwiseheatkernelderivbounds}) follow from Corollary \ref{cor:exponentialboundnonscalar} by observing that 
\begin{equation*}
\exp\left(-\frac{d^2_{\bar g(0)}(x,y)}{Dt}\right) \leq \left(\frac{c(n+k)}{\sqrt{\frac{d^2(x,y)}{t}} + 1}\right)^{n+k} = \left(\frac{c(n+k)\sqrt{t}}{d(x,y) + \sqrt{t}}\right)^{n+k}.
\end{equation*}

The bounds (\ref{eq:ptwiseheatkernelawayfromball}) follow from Corollary \ref{cor:exponentialboundnonscalar} immediately in the case that $s\leq r^2/2$. On the other hand, if $d(x,y) \geq r$, then 
\begin{align*}
(r^2 - s)^{\frac{-n-k}{2}}\exp\left(-\frac{d^2_{\bar g(0)}(x,y)}{D(r^2-s)}\right) & \leq(r^2 - s)^{\frac{-n-k}{2}}\left(\frac{c(n+k)\sqrt{r^2 -s}}{d_{\bar g(0)}(x,y) + \sqrt{r^2-s}}\right)^{n+k}\exp\left(-\frac{d^2_{\bar g(0)}(x,y)}{2D(r^2-s)}\right)
\\& = \left(\frac{c(n+k)}{d_{\bar g(0)}(x,y) + \sqrt{r^2-s}}\right)^{n+k}\exp\left(-\frac{d^2_{\bar g(0)}(x,y)}{2D(r^2-s)}\right)
\\& \leq \left(\frac{c(n+k)}{r + \sqrt{r^2-s}}\right)^{n+k}\exp\left(-\frac{d^2_{\bar g(0)}(x,y)}{2D(r^2-s)}\right)
\\& \leq \left(\frac{c(n+k)}{r}\right)^{n+k}\exp\left(-\frac{d^2_{\bar g(0)}(x,y)}{2D(r^2-s)}\right).
\end{align*}

The bounds (\ref{eq:heatkernelawayfrom0}) follow from Corollary \ref{cor:exponentialboundnonscalar} by integration and rescaling the metric, as follows:
\begin{align*}
\int_{M\setminus B_{\bar g(0)}(x,r)}C_0t^{-n/2}\exp\left(-\frac{d^2_{\bar g(0)}(x,y)}{Dt}\right)d\bar g_0  
& \leq C_0\exp\left(-\frac{r^2}{2Dt}\right)\int_{M\setminus B_{\bar g(0)}(x,r/\sqrt{t})}\exp\left(-\frac{d^2_{\bar g(0)}(x,y)}{2D}\right)d\hat{\bar g}_0
\\& \leq C_0\exp\left(-\frac{r^2}{2Dt}\right)\int_0^{\infty}\vol(B_s^{-Tc})\exp\left(-\frac{s^2}{2D}\right)ds
\\& \leq C_0'\exp\left(-\frac{r^2}{2Dt}\right),
\end{align*}
where $\hat{\bar g} = t^{-1}\bar g$, and $B_s^{-Tc}$ denotes a ball of radius $s$ in a space of constant curvature $-Tc$, since $|\Rm|(\hat{\bar g}) = t|\Rm|(\bar g) \leq T|\Rm|(\bar g)$. The penultimate line is due to Bishop-Gromov. The integral bound on the covariant derivative of $\bar K$ follows similarly.
\end{proof}

\subsection{Scalar heat kernel for the Ricci-DeTurck flow}
In this section we push forward the analysis on $\bar \Phi$ by diffeomorphisms to estimate the scalar heat kernel for the Ricci-DeTurck flow.

\begin{lemma}\label{lemma:RDTexpbound}
Let $(g(t))_{t\in(0,T]}$ be a solution to the Ricci-DeTurck equation with respect to a smooth background Ricci flow $(\bar g(t))_{t\in [0,T]}$, such that $g(t)$ is uniformly $(1+b)$-bilipschitz to $\bar g(t)$ on $[0,T]$. Take $t_0$ sufficiently small so that (\ref{eq:almosteucl}) holds for the Ricci flow $\chi_t^*g_t$, where $\chi_t$ are as in (\ref{eq:diffeoseq}), and suppose that $g(t)$ is smooth on $[\tfrac{t_0}{2}, t_0]$. Let $\Phi(x,t;y,s)$ be the heat kernel for the operator $\partial_t - \Delta^{g(t)} - \nabla^{g(t)}_{X_{\bar g}(g)}$, where $X$ is as in (\ref{eq:Xoperator}). Suppose also that $g(t)$ satisfies a bound of the form $|(\nabla^{\bar g})^{m}g(t)|_{\bar g} \leq \tfrac{c}{t^{m/2}}$ for $m=1,2$ on $[\tfrac{t_0}{2}, t_0]$. Then there exist constants $C = C(n,c) < \infty$ and $D(c, b, T\sup_{[0,T]}|\Rm|(\bar g(t))) < \infty$ such that, for all $s,t\in [\tfrac{t_0}{2}, t_0]$, we have
\begin{equation}
\Phi(x,t;y,s) \leq \frac{C}{(t-s)^{n/2}}\exp\left(\frac{-d^2_{\bar g_0}(x,y)}{D(t-s)}\right)
\end{equation}
\end{lemma}

\begin{proof}
We pull back Lemma \ref{lemma:timeintervalhkb} by diffeomorphisms, as in \cite[Lemma $3$]{Bam}. Let $\bar \Phi$ be the heat kernel with respect to the Ricci flow $\chi_t^*g(t)$, so that $\Phi(x,t;y,s) = \bar \Phi(\chi_t^{-1}(x), t; \chi_s^{-1}(y),s)$. 

First observe that, by Jensen's inequality,
\begin{equation}\label{eq:Jensentriangle}
d^2_{\chi_{t_0/2}^*g(t_0/2)}(\chi_t^{-1}(x), \chi_t^{-1}(y))  \geq \tfrac{1}{2}d^2_{\chi_{t_0/2}^*g(t_0/2)}(\chi_t^{-1}(x), \chi_t^{-1}(y)) - d^2_{\chi_{t_0/2}^*g(t_0/2)}(\chi_t^{-1}(y), \chi_s^{-1}(y)).
\end{equation}

Then, by (\ref{eq:almosteucl}), (\ref{eq:Jensentriangle}), and (\ref{eq:distancesnotmuchmoved}), and Lemma \ref{lemma:timeintervalhkb} applied to $\chi_t^*g_t$,  we have
 \begin{align*}
 \Phi(x,t;y,s) & = \bar \Phi(\chi_t^{-1}(x), t; \chi_s^{-1}(y),s) \leq  \frac{C}{(t-s)^{n/2}}\exp\left(-\frac{d_{\chi_{t_0/2}^*g(t_0/2)}^{2}(\chi_t^{-1}(x), \chi_s^{-1}(y))}{D(t-s)}\right)
\\& \leq \frac{C}{(t-s)^{n/2}}\exp\left(-\frac{d^{2}_{\chi_{t_0/2}^*g(t_0/2)}(\chi_t^{-1}(x), \chi_t^{-1}(y))}{D(t-s)}\right)\exp\left(\frac{d^{2}_{\chi_{t_0/2}^*g(t_0/2)}(\chi_t^{-1}(y), \chi_s^{-1}(y))}{D(t-s)}\right)
\\& \leq \frac{C}{(t-s)^{n/2}}\exp\left(-\frac{d^{2}_{\chi_{t}^*g(t)}(\chi_t^{-1}(x), \chi_t^{-1}(y))}{D(t-s)}\right)\exp\left(\frac{d^{2}_{\chi_{t}^*g(t)}(\chi_t^{-1}(y), \chi_s^{-1}(y))}{D(t-s)}\right)
\\& = \frac{C}{(t-s)^{n/2}}\exp\left(-\frac{d^{2}_{g(t)}(x, y)}{D(t-s)}\right)\exp\left(\frac{d^{2}_{g(t)}(y, \chi_t\circ\chi_s^{-1}(y))}{D(t-s)}\right)
\\& = \frac{C}{(t-s)^{n/2}}\exp\left(-\frac{d^{2}_{g(t)}(x, y)}{D(t-s)}\right)\exp\left(\frac{d^{2}_{g(t)}(\chi_s\circ\chi_s^{-1}(y), \chi_t\circ\chi_s^{-1}(y))}{D(t-s)}\right)
\\& \leq \frac{C}{(t-s)^{n/2}}\exp\left(-\frac{d^{2}_{\bar g_0}(x, y)}{D(t-s)}\right)\exp\left(\frac{d^{2}_{\bar g_0}(\chi_s\circ\chi_s^{-1}(y), \chi_t\circ\chi_s^{-1}(y))}{D(t-s)}\right)
\\& \leq \frac{C}{(t-s)^{n/2}}\exp\left(-\frac{d^{2}_{\bar g_0}(x, y)}{D(t-s)}\right)\exp\left(\frac{(\sqrt{t} - \sqrt{s})^2}{D(t-s)}\right) \leq \frac{C}{(t-s)^{n/2}}\exp\left(-\frac{d^{2}_{\bar g_0}(x, y)}{D(t-s)}\right),
\end{align*}
after adjusting the constants.
\end{proof}

\begin{corollary}\label{cor:heatkernelboundforballs}
Let $(g(t))_{t\in(0,T]}$ be a solution to the Ricci-DeTurck equation with respect to a smooth background Ricci flow $(\bar g(t))_{t\in [0,T]}$, where we take $T$ sufficiently small so that (\ref{eq:almosteucl}) holds. Suppose that $g(t)$ is smooth on $[\tfrac{t_0}{2}, t_0]$ and uniformly $(1+b)$-bilipschitz to $\bar g(t)$ on $[0,T]$. Let $\Phi(x,t;y,s)$ be the heat kernel for the operator $\partial_t - \Delta^{g(t)} - \nabla^{g(t)}_{X_{\bar g}(g)}$, where $X$ is as in (\ref{eq:Xoperator}). Suppose also that $g(t)$ satisfies a bound of the form $|(\nabla^{\bar g})^{m}g(t)|_{\bar g} \leq \tfrac{c}{t^{m/2}}$ for $m=1,2$ on $[\tfrac{t_0}{2}, t_0]$. Then there exist constants $C = C(n,c, T\sup_{[0,T]}|\Rm|(\bar g(t))) < \infty$ and $D(c, b, T\sup_{[0,T]}|\Rm|(\bar g(t))) < \infty$ such that, for any $r>0$, and all $s,t\in [\tfrac{t_0}{2}, t_0]$ with $s<t$, we have
\begin{equation}
\int_{M\setminus B_{\bar g_0}(x,r)}\Phi(x,t; y, s)dg_s(y) \leq C\exp\left(-\frac{r^2}{D(t-s)}\right).
\end{equation}
\end{corollary}
\begin{proof}
The proof is similar to that of (\ref{eq:heatkernelawayfrom0}). We have, by Lemma \ref{lemma:RDTexpbound}, (\ref{eq:distancedistortion}), and Bishop-Gromov,
\begin{align*}
\int_{M\setminus B_{\bar g_0}(x,r)}\Phi(x,t; y, s)dg_s(y) &\leq C\int_{M\setminus B_{\bar g_0}(x,r)}\frac{\exp\left(-\frac{d^2_{\bar g(t_0/2)}(x,y)}{D(t-s)}\right)}{(t-s)^{n/2}}
 \\& \leq  C\exp\left(-\frac{d^2_{\bar g_0}(x,y)}{2D(t-s)}\right)\int_{0}^{\infty}\vol(B_{s'}^{-\sup|\Rm|(\tilde g)(t-s)})\exp\left(-\frac{s'^2}{2D}\right)ds'
 \\& \leq C\exp\left(-\frac{d^2_{\bar g_0}(x,y)}{2D(t-s)}\right).
\end{align*}
\end{proof}

\subsection{Analytic preliminaries}

In \S \ref{sec:initialderivatives} we make use of the following result concerning higher order derivatives of a tensor field on a manifold.
\begin{lemma}\label{lemma:derivativeinterpolation}
There exists $C = C(n,q, \ell, K,\rho)$ such that if $(M^n,g)$ is a closed manifold with $|\nabla^k\Rm| \leq K$ for all $0\leq k \leq \ell$ and $\inj(M)\geq \rho > 0$, then, for any $C^{\ell}$ $(0,q)$-tensor field $f$ and all $0<r\leq 1$, $0<a\leq r$, and all $x\in M$, we have
\begin{equation}
\begin{split}
a||\nabla f||_{L^\infty(B(x,r))} + a^2||\nabla^2 f||_{L^\infty(B(x,r))} + \cdots + &a^{\ell - 1}||\nabla^{\ell - 1}f||_{L^\infty(B(x,r))} 
\\& \leq C||f||_{L^{\infty}(B(x, r+a))} + Ca^{\ell}||\nabla^{\ell}f||_{L^\infty(B(x, r+a))}.
\end{split}
\end{equation}
In particular,
\begin{equation}\label{eq:stronginterpolation}
||\nabla^k f||_{L^\infty(B(x,r))} \leq \frac{C}{a^k}||f||_{L^\infty(B(x, r+a))} + Ca^{\ell - k}||\nabla^\ell f||_{L^\infty(B(x, r+a))}.
\end{equation}
\end{lemma}

\begin{proof}
We first show the case for $a = r = 1$. We proceed by contradiction, so suppose the lemma is false and let $C^i \to \infty$. Pick a sequence of counterexamples $(M_i, g_i, x_i, f_i)$ such that
\begin{equation}\label{eq:counterexamples}
\begin{split}
||\nabla^{g_i} f_i||_{L^\infty(B_{g_i}(x_i,1))} &+ ||(\nabla^{g_i})^2 f_i||_{L^\infty(B_{g_i}(x_i, 1))} \cdots + ||(\nabla^{g_i})^{\ell - 1} f_i||_{L^\infty(B_{g_i}(x_i, 1))} 
\\& \qquad \qquad \qquad > C_i ||f_i||_{L^\infty(B_{g_i}(x_i,2))} + C_i||(\nabla^{g_i})^\ell f_i||_{L^\infty(B_{g_i}(x_i,2))}.
\end{split}
\end{equation}

Due to the bounds $|\Rm|(g_i) \leq K, \inj(x_i)\geq \rho$, we have smooth pointed convergence, determined by maps $\phi_i: M_\infty \to M_i$ which are diffeomorphisms onto their images, such that, on a subsequence, we have $(M_i, g_i, x_i)\to (M_{\infty}, g_{\infty}, x_{\infty})$. Pass to this subsequence, so that we may assume, without loss of generality, that $M_i = M_{\infty}$ for all $i$, by replacing $f_i$ by $f_i\circ \phi_i$. Moreover, (\ref{eq:counterexamples}) implies that
\begin{align*}
||\nabla f_i||_{L^\infty(B(x_\infty,1))} &+ ||\nabla^2 f_i||_{L^\infty(B(x_\infty, 1))} \cdots + ||\nabla^{\ell - 1} f_i||_{L^\infty(B(x_\infty, 1))} 
\\ & \geq ||\nabla^{g_i} f_i||_{L^\infty(B_{g_i}(x_i,1))} + ||(\nabla^{g_i})^2 f_i||_{L^\infty(B_{g_i}(x_i, 1))} \cdots + ||(\nabla^{g_i})^{\ell - 1} f_i||_{L^\infty(B_{g_i}(x_i, 1))} - \varepsilon_i
\\& > C_i ||f_i||_{L^\infty(B_{g_i}(x_i,2))} + C_i||(\nabla^{g_i})^\ell f_i||_{L^\infty(B_{g_i}(x_i,2))} - \varepsilon_i
\\& \geq C_i||f_i||_{L^\infty(B(x_\infty,2))} + C_i||\nabla^\ell f_i||_{L^\infty(B(x_\infty,2))} -\varepsilon_i - \delta_i
\end{align*}
where $\varepsilon_i$ and $\delta_i$ are some numbers that tend to $0$ as $i\to \infty$, $\nabla$ denotes the covariant derivative with respect to $g_\infty$, and $B(p, r)$ denotes the ball of radius $r$ with respect to $g_\infty$ about $p$.

Multiplying $f_i$ by the correct constant, we may assume that
\begin{equation}\label{eq:rescaleto1}
||\nabla f_i||_{L^\infty(B(x_\infty, 1))} + \cdots + ||\nabla^{\ell - 1}f_i||_{L^\infty(B(x_\infty, 1))} = 1.
\end{equation}
Since $\tfrac{1}{C_i}\to 0$, (\ref{eq:rescaleto1}) implies that $||f_i||_{L^\infty(B(x_\infty, 2))} + ||\nabla^{\ell}f_i||_{L^\infty(B(x_\infty, 2))} \leq 1/C_i + \varepsilon_i + \delta_i$, so certainly $||f_i||_{L^\infty(B(x_\infty, 2))}$ and $||\nabla^{\ell}f_i||_{L^\infty(B(x_\infty, 2))}$ are uniformly bounded with respect to $i$. We wish to show that the $\nabla^kf_i$ are uniformly bounded and uniformly equicontinuous within an exponential coordinate chart for $g_\infty$. Let $U\subset B(x_{\infty}, 2)$ be one such chart.

For all $k < \ell$ the sequence $\{\nabla^kf_i\}$ is uniformly bounded on $U$ as follows: Fix $0\leq k < \ell$ and assume that $\{ \nabla^{k+1} f_i\}$ is uniformly bounded on $U$, so that the argument may be applied iteratively. Let $y\in U$, and let $\gamma$ be a minimizing geodesic from $x_i$ to $y$, parametrized by arclength. Then, using (\ref{eq:rescaleto1}) and the mean value theorem, we have
\begin{align*}
|\nabla^kf_i(y)| &\leq \max_{t\in [0, d(x_i,y)]}\frac{d}{dt}\left|\nabla^{k}f_i(\gamma(t))\right|d(x_i,y) + |\nabla^{k} f_i(x_i)|
\\& \leq \max_{t\in [0,d(x_i,y)]}\frac{1}{|\nabla^kf_i(\gamma(t))|}\langle\nabla^kf_i(\gamma(t)), \nabla_{\dot{\gamma}(t)}\nabla^kf_i(\gamma(t))\rangle d(x_i,y) + 1
\\& \leq \sup_{i\in \N}||\nabla^{k+1}f_i||_{L^\infty(U)}\diam(U) + 1,
\end{align*}
so $\{\nabla^kf_i\}$ is uniformly bounded on $U$. We now show that the $\nabla^k f_i$ are uniformly equicontinuous within $U$. For multiiindices $k$ and $m$ we have that 
\begin{align*}
\nabla^k f_i &= \nabla^k\left(f_{a_1a_2\cdots a_q}dx^{a_1}\otimes dx^{a_2}\otimes \cdots \otimes dx^{a_q}\right)
\\&= \sum_{|m|=0}^{|k|}\left(\partial^m f_{a_1a_2\cdots a_q}\right)\nabla^{k-m}\left(dx^{a_1}\otimes dx^{a_2}\otimes \cdots \otimes dx^{a_q}\right)
\end{align*}
Then we may inductively show that $\partial^k f_{a_1a_2\cdots a_q}$ are uniformly bounded on $U$ because $|\nabla^k f_i|$, $\partial^m  f_{a_1a_2\cdots a_q}$ for $|m| < |k|$, and Christoffel symbols and first $k$ derivatives of the Christoffel symbols are all uniformly bounded with respect to $i$ on $U$. Then, for $y$ and $y'$ within $U$, we have
\begin{equation*}
\begin{split}
|\partial^kf_i(y) - \partial^kf_i(y')| &\leq \max_{a_1, a_2, \ldots, a_q}||\partial^{k+1}(f_i)_{a_1a_2\cdots a_q}||_{L^\infty(U)}|y - y'|
\\& \leq c|y-y'|,
\end{split}
\end{equation*}
where $c$ is some constant independent of $i$, and depends on upper bounds for the first $k$ derivatives of $\Rm(g_\infty)$. Therefore, the Arzel\'a-Ascoli theorem implies that, on a subsequence, $f_i$ converges in $C^{\ell - 1}_{\loc}(U)$ to some limiting tensor. After covering $B(x_\infty, 1.5)$ by exponential coordinate charts, we find that the $f_i$ converge on a subsequence to a limiting $f_\infty$ in $C^{\ell - 1}(B(x_\infty, 1.1))$. Pass to this subsequence to find
\begin{align*}
0&= \lim_{i\to\infty}\frac{1}{C_i} + \varepsilon_i + \delta_i \geq \lim_{i\to \infty} ||f_i||_{L^{\infty(B(x_\infty, 1.1))}} = ||f_{\infty}||_{L^\infty(B(x_\infty, 1.1))},
\end{align*}
so $f_\infty \equiv 0$ on $B(x_\infty, 1.1)$. On the other hand, 
\begin{align*}
1 &= \lim_{i\to \infty}\left( ||\nabla f_i||_{L^\infty(B(x_\infty,1))} + \cdots + ||\nabla^{\ell - 1}f_i||_{L^\infty(B(x_\infty,1))}\right)
\\&= ||\nabla f_\infty||_{L^\infty(B(x_\infty,1))} + \cdots + ||\nabla^{\ell - 1}f_\infty||_{L^\infty(B(x_\infty,1))},
\end{align*}
so we cannot have $||\nabla f_{\infty}||_{L^{\infty}(B(x_\infty,1))} = ||\nabla^2 f_{\infty}||_{L^{\infty(B(x_\infty,1))}}= \cdots = ||\nabla ^{\ell - 1}f_{\infty}||_{L^\infty(B(x_\infty,1))} = 0$. This is a contradiction, so we find that there exists $C$ such that, for any $(M,g,x, f)$,
\begin{equation}\label{eq:interpolationat1}
||\nabla f||_{L^\infty(B(x,1))} + ||\nabla^2 f||_{L^\infty(B(x,1))} + \cdots + ||\nabla^{\ell - 1}f||_{L^\infty(B(x,1))} \leq C||f||_{L^{\infty}(B(x,2))} + C||\nabla^{\ell}f||_{L^\infty(B(x,2))}.
\end{equation}

We now show the case for $0 < a = r < 1$. If $(M,g)$ is any Riemannian manifold satisfying the hypotheses of the lemma, $0<r<1$, $x\in M$, and $f$ a $C^{\ell}$ $(0,q)$-tensor field on $M$, then the rescaled metric $g':= r^{-2}g$ satisfies $|\Rm|(g') \leq r^2K < K$ and $\inj(M, g') \geq \rho/r > \rho$, so we may apply (\ref{eq:interpolationat1}) to $g'$ to find
\begin{equation}\label{eq:interpolationatr}
\begin{split}
r^{q+1}||\nabla f||_{L^\infty(B_g(x,r))} &+ r^{q+2}||\nabla^2f||_{L^\infty(B_g(x,r))} + \cdots + r^{q + \ell - 1}||\nabla^{\ell - 1}f||_{L^\infty(B_g(x,r))}
\\& = ||\nabla f||_{L^\infty(B_{g'}(x,1))} + ||\nabla^2 f||_{L^\infty(B_{g'}(x,1))} + \cdots + ||\nabla^{\ell - 1}f||_{L^\infty(B_{g'}(x,1))} 
\\& \leq C||f||_{L^{\infty}(B_{g'}(x,2))} + C||\nabla^{\ell}f||_{L^\infty(B_{g'}(x,2))}
\\& = Cr^q||f||_{L^\infty(B_g(x, 2r))} + Cr^{q + \ell}||\nabla^{\ell}f||_{L^\infty(B_g(x, 2r))},
\end{split}
\end{equation}
where $C$ is the constant from (\ref{eq:interpolationat1}).

We now handle the case where we have $a<r$. Let $y\in B(x,r)$. Then $y\in B(x',a)$ for some $x'$ with $d(x,x') \leq r-a$ (by choosing $x'$ along a minimizing geodesic from $x$ to $y$). Note that $B(x', 2a)\subset B(x, r+a)$ so, by (\ref{eq:interpolationatr}), we have
\begin{equation*}
\begin{split}
a|\nabla f|(y) + a^2|\nabla^2f|(y) + \cdots + a^{\ell-1}|\nabla^{\ell - 1}f|(y) &\leq C||f||_{L^\infty(B(x', 2a))} + Ca^{\ell}||\nabla^{\ell}f||_{L^\infty(B(x', 2a))}
\\& \leq C||f||_{L^\infty(B(x, r+a))} + Ca^{\ell}||\nabla^{\ell}f||_{L^\infty(B(x, r+a))},
\end{split}
\end{equation*}
whence follows the result.
\end{proof}

We now record the following version of Young's convolution inequality for heat kernels; this is essentially \cite[Theorem $0.3.1$]{Sogge}.

\begin{lemma}\label{lemma:Youngsfork}
Let $k_t(\cdot,\cdot)$ be a heat kernel for some operator on $M$, and let $Q$ be a time-dependent family of tensor fields of the appropriate shape on $M$. Suppose $\tfrac{1}{p} + \tfrac{1}{q} = \tfrac{1}{r} + 1$. Let $U\subset M$, and let $I$ and $J$ be time intervals. Suppose also that
\begin{equation*}
\max\left\{\sup_{(x,t)\in U\times J}\left(\int_I\int_M|k_{t-s}(x,y)|^p dyds\right)^{1/p}, \sup_{(y,s)\in M\times I}\left(\int_J\int_U|k_{t-s}(x,y)|^pdxdt\right)^{1/p}\right\} \leq B.
\end{equation*}
If $Q\in L^q(M\times I)$, then
\begin{equation}
\left|\left|\int_{M\times I}k_{t-s}(x,y)Q(y,s)dyds\right|\right|_{L^r((x,t)\in U\times J)} \leq B ||Q||_{L^q(M\times I)}.
\end{equation}
\end{lemma}

\section{Regularity of solutions to the integral equation}\label{sec:initialderivatives}

The purpose of this section is to discuss properties of solutions to the integral equation that corresponds to the Ricci-DeTurck flow. For convenience, we consider the Ricci-DeTurck flow on a smooth Ricci flow background, $\bar g(t)$. Let $\bar K$ denote the $(0,2)$-tensor heat kernel associated to the background $\bar g(t)$ for the operator $\partial_t + L$, where $L$ is given by (\ref{eq:Lis}).

We study the integral equation that corresponds to the perturbation equation (\ref{eq:hevolution}):
\begin{equation}\label{eq:integraleq}
\begin{split}
h(x,t) &= \int_M \bar K(x,t;y, 0)h_0(y)d\bar g_0(y) + \int_{M\times[0,t]}\bar K(x,t; y, s)Q^0_s(y) + \nabla^* \bar K(x,t;y,s)Q^1_s(y)d\bar g_s(y) 
\\& =: F[h, h_0].
\end{split}
\end{equation}

Having discussed the existence of a solution to the integral equation (\ref{eq:integraleq}), we will then show that such solutions are smooth away from $t=0$, and satisfy certain derivative bounds. To do this, we appeal to derivative estimates for smooth Ricci-DeTurck flows and properties of the operator $F[\cdot,\cdot]$ defined by (\ref{eq:integraleq}).

The existence of a solution to (\ref{eq:integraleq}) is given by a fixed point argument as in \cite{KL1}. As such, we introduce their norms below. We also define weighted analogs of these norms, which are not used in this section, but which we will use in the next section to study the evolution under the Ricci-DeTurck flow of the difference of two metrics that agree to greater-than-second order.

\begin{definition}\label{def:norms}
Let $\bar g(t)$ be a smooth Ricci flow on $M$ defined on the time interval $[0,T]$, where we take $T<1$.
Throughout we take all covariant derivatives with respect to $\bar g(t)$, and measure all balls and take all absolute values with respect to $\bar g(0)$, which is uniformly bilipschitz to $\bar g(t)$ on this interval by (\ref{eq:distancedistortion}). We define the following localized versions of the Banach spaces introduced in \cite{KL1}.

\begin{align*}
X(B(x,r))&=\bigg\{h\in L^{\infty}_{\loc}\cap \dot{W}^{1,2}_{\loc}\cap \dot{W}^{1, n+4}_{\loc}: ||h||_{X(B(x,r))} = \sup_{0<t<T}||h_t||_{L^\infty(B(x,r))} 
\\& \qquad + \left(r^{-n/2}||\nabla h||_{L^2(B(x,r)\times(0,r^2))} + r^{\tfrac{2}{n+4}}||\nabla h||_{L^{n+4}(B(x,r)\times(\tfrac{r^2}{2},r^2))}\right)< \infty\bigg\},\\
 \text{ and }\\
Y(B(x,r)) &= \{f: ||f||_{Y(B(x,r))} < \infty\}, \text{ where } 
\\ ||f||_{Y(B(x,r))} &= \inf\{||f_0||_{Y^0(B(x,r))} + ||f_1||_{Y^1(B(x,r))}: f_0\in Y^0(B(x,r)), f_1\in Y^1(B(x,r)), 
\\& \qquad \qquad \text{ and } f = f_0 + \nabla^*f_1\},\\
 ||f||_{Y^0(B(x,r))} &= r^{-n}||f||_{L^1(B(x,r)\times(0,r^2))} + r^{\tfrac{4}{n+4}}||f||_{L^{\tfrac{n+4}{2}}(B(x,r)\times(\tfrac{r^2}{2}, r^2))}\\
 ||f||_{Y^1(B(x,r))} &= r^{-n/2}||f||_{L^2(B(x,r)\times(0,r^2))} + r^{\tfrac{2}{n+4}}||f||_{L^{n+4}(B(x,r)\times(\tfrac{r^2}{2}, r^2))}.
\end{align*}
More precisely, we take the Banach spaces to be the completions with respect to their respective norms of the smooth time-dependent families of $(0,2)$-tensor fields on $M$ for which the aforementioned norms are finite.

The following three norms are due to Koch and Lamm (see, for instance, \cite[Page $225$]{KL1}); the difference between their setting and ours being that they employ a stationary background metric, while ours evolves by the Ricci flow.
\begin{align*}
X_T &= \bigg\{h : ||h||_{X} = \sup_{\substack{x\in M\\0<r^2<T}} ||h||_{X(B(x,r))} < \infty\bigg\} \text{ and }
Y_T = \{f: ||f||_{Y} < \infty\}, \text{ where } 
\\ ||f||_{Y} &= \inf\{||f_0||_{Y^0_T} + ||f_1||_{Y^1}: f_0\in Y^0_T, f_1\in Y^1_T, f = f_0 + \nabla^*f_1\}, \text{and where } \\
||f||_{Y^0} &= \sup_{\substack{x\in M \\ 0<r^2<T}}||f||_{Y^0(B(x,r))} \text{ and }
||f||_{Y^1} = \sup_{\substack{x\in M\\ 0 < r^2 <T}}||f||_{Y^1(B(x,r))}.
\end{align*}
In most cases we will suppress the $T$, and simply write $X$ and $Y$ to mean $X_T$ and $Y_T$ respectively. For $\gamma >0$, we set $X^\gamma := \{f\in X : ||f||_X \leq \gamma\}$.

We now define some weighted norms, which we will use in the next section. These are similar to the unweighted norms that we have just introduced, but they are equipped with a greater-than-second-order weight designed to offset the evolution of metrics that agree to greater than second order near a point.

For $a\geq 0, t\in [0,T], x\in M$, and some given $\eta>0$, let $w_a(x,t)$ denote the greater-than-second-order weight given by
\begin{equation}\label{eq:weight}
w_a(x,t) := \max\{(d(x_0, x) + \sqrt{t} + a)^{-2-\eta}, 1\}.
\end{equation}
Observe that $w_0(\cdot, \cdot)$ is not defined at $(x_0, 0)$. We say that two initial metrics $g_0'$ and $g_0''$ \emph{agree to greater than second order} around $x_0$ if there is some $\eta$ such that $||w_0(\cdot, 0)(g_0' - g_0'')||_{L^\infty(B(x_0,R))} < \infty$ for some $R>0$.
For a fixed point $x_0\in M$, exponent $\eta>0$, and time $T$, we define the weighted Banach spaces
\begin{align*}
\tilde X_a(B(x,r))&=\bigg\{h \in L^{\infty}_{\loc}\cap \dot{W}^{1,2}_{\loc}\cap \dot{W}^{1, n+4}_{\loc}: ||h||_{\tilde X_a(B(x,r))} := \sup_{0<t<T}||w_a(t)h_t||_{L^\infty(B(x,r))} 
\\& \qquad + w_{a}(x,r^2)\left(r^{-n/2}||\nabla h||_{L^2(B(x,r)\times(0,r^2))} + r^{\tfrac{2}{n+4}}||\nabla h||_{L^{n+4}(B(x,r)\times(\tfrac{r^2}{2},r^2))}\right)< \infty\bigg\},
\\ \tilde Y_a(B(x,r))&= \{f: ||f||_{\tilde Y_a(B(x,r))}< \infty\}, \text{ where } 
\\ ||f||_{\tilde Y_a(B(x,r))} &= \inf\{||f_0||_{\tilde Y^0_a(B(x,r))} + ||f_1||_{\tilde Y^1_a(B(x,r))}: f_0\in \tilde Y^0_a(B(x,r)), f_1\in \tilde Y^1_a(B(x,r)),
\\& \qquad \qquad \qquad \text{ and } f = f_0 + \nabla^*f_1\}\\
||f||_{\tilde Y^0_a(B(x,r))} &= w_{a}(x,r^2)\left(r^{-n}||f||_{L^1(B(x,r)\times(0,r^2))} + r^{\tfrac{4}{n+4}}||f||_{L^{\tfrac{n+4}{2}}(B(x,r)\times(\tfrac{r^2}{2}, r^2))}\right)\\
||f||_{\tilde Y^1_a(B(x,r))} &= w_{a}(x,r^2)\left(r^{-n/2}||f||_{L^2(B(x,r)\times(0,r^2))} + r^{\tfrac{2}{n+4}}||f||_{L^{n+4}(B(x,r)\times(\tfrac{r^2}{2}, r^2))}\right).
\end{align*}

Because the representation formula (\ref{eq:integraleq}) requires integration over the entire manifold, it is often necessary to consider weighted norms that are not localized to a particular neighborhood of $x_0$. To this end, we define
\begin{align*}
\tilde X_a &= \bigg\{h : ||h||_{\tilde X_a} = \sup_{\substack{x\in M\\0<r^2<T}} ||h||_{\tilde X_a(B(x,r))} < \infty\bigg\} \text{ and }
\tilde Y_a = \{f: ||f||_{\tilde Y_a} < \infty\}, \text{ where } 
\\ ||f||_{\tilde Y_a} &= \inf\{||f_0||_{\tilde Y^0_a} + ||f_1||_{\tilde Y^1_a}: f_0\in \tilde Y^0_a, f_1\in \tilde Y^1_a, f = f_0 + \nabla^*f_1\}, \text{ and where } \\
||f||_{\tilde Y^0_a} &= \sup_{\substack{x\in M \\ 0<r^2<T}}||f||_{\tilde Y^0_a(B(x,r))} \text{ and }
||f||_{\tilde Y^1_a} = \sup_{\substack{x\in M\\ 0 < r^2 <T}}||f||_{\tilde Y^1_a(B(x,r))}.
\end{align*}
\end{definition} 

\begin{remark}
By definition, any element of $X$ is continuous (for positive times), because it is a locally uniform limit of smooth tensor fields.
\end{remark}

In order to state our main result, we must know that there exists a solution to the Ricci-DeTurck flow equation starting from $C^0$ initial data:
 \begin{lemma}\label{lemma:RDTexistence}
For any smooth Ricci flow $\bar g(t)$ on a closed manifold $M^n$, defined for $t\in [0,T]$, if $T$ is sufficiently small so that (\ref{eq:almosteucl}) and (\ref{eq:boundedcoeffs}) hold, there exist constants $\varepsilon(n,T\sup_{[0,T]}|\Rm|(\bar g))$, $C(n,T\sup_{[0,T]}|\Rm|(\bar g))>0$ such that the following is true:

 For every metric $g_0\in C^0(M)$ such that $||g_0 - \bar g(0)||_{L^\infty(M)} < \varepsilon$, there exists a solution $g_t \in \bar g(t) + X_T$ of the integral equation (\ref{eq:integraleq}) with $||g_t - \bar g(t)||_{X_T} \leq C||g_0 - \bar g(0)||_{L^\infty(M)}$. Furthermore, the solution is unique in $\{g_t: ||g_t - \bar g(t)||_{X_T} \leq C\varepsilon < 1 \}$.

\medskip
\noindent Moreover, there exist constants $\varepsilon = \varepsilon(n), C = C(n)$ such that the following is true:

For every metric $g_0\in C^0(M)$ and every smooth background metric $\bar g_0$ on $M$, if $||g_0 - \bar g(0)||_{L^\infty(M)} < \varepsilon$ and $\bar g(t)$ is the Ricci flow starting from $\bar g_0$, then there exists $T= T(\bar g(t))$ sufficiently small so that (\ref{eq:almosteucl}) and (\ref{eq:boundedcoeffs}) hold, and such that there is a solution $g_t \in \bar g(t) + X_T$ to the integral equation (\ref{eq:integraleq}) with $||g_t - \bar g(t)||_{X_T} \leq C||g_0 - \bar g(0)||_{L^\infty(M)}$. Furthermore, the solution is unique in $\{g_t : ||g_t - \bar g(t)||_{X_T} \leq C\varepsilon < 1\}$.
 \end{lemma}
 
 The main result of this section is:
 \begin{corollary}\label{cor:C0RDTderivbounds}
Let $\varepsilon$ be as in the second statement in Lemma \ref{lemma:RDTexistence}. There exists a positive constant $\varepsilon' = \varepsilon'(n) \leq \varepsilon$ such that the following is true:

Let $g_0\in C^0(M)$ and $\bar g(t)$ be a smooth Ricci flow defined on some positive time interval. Let $T$ be as in the second statement in Lemma \ref{lemma:RDTexistence}. Suppose $||g_0 - \bar g(0)||_{L^\infty(M)} < \varepsilon'$ and let $g_t$ be the solution to (\ref{eq:integraleq}) whose existence is assured by Lemma \ref{lemma:RDTexistence}. Then there exists $T' = T'(\bar g(t))< T$ and $c_k = c(m, n, \sup_{t\in [0,T], \ell \leq k}|\nabla^{\ell}\Rm|(\bar g(t)))$ such that $g_t$ is smooth on $M\times (0,T']$, continuous on $M\times [0,T']$, and satisfies
\begin{equation}\label{eq:C0RDTderivbounds}
|\nabla^k(g_t - \bar g_t)| \leq\frac{c_k}{t^{k/2}}||g_0 - \bar g_0||_{L^\infty(M)},
\end{equation}
for all $t\in (0,T']$, where $\nabla$ denotes the covariant derivative with respect to $\bar g(t)$.

In particular, if $g^i$ is a sequence of $C^0$ metrics on $M$ such that $g^i\xrightarrow[i\to\infty]{C^0} g$, and $g^i(t)$, $g(t)$ are the Ricci-DeTurck flows with respect to $\bar g(t)$ starting from $g^i$ and $g$ respectively, then $g^i(t)$ converges locally smoothly to $g(t)$ on $M\times (0,T']$.
\end{corollary}

\begin{remark}
In \cite{KL1}, Koch and Lamm proved the estimates (\ref{eq:C0RDTderivbounds}) by applying the analytic implicit function theorem and showing that, for a solution $h$ to (\ref{eq:integraleq}) on $\R^n$ with the standard Euclidean metric as the background metric, the tensor $h_{a,\tau}(x,\tau) := h(x-at, \tau t)$ is analytic in $a$ and $\tau$ in a neighborhood of $(a,\tau) = (0,0)$, and hence $h$ is analytic in $x$ and in $t$. Because we do not work on $\R^n$, this technique is not available to us. We will provide an alternate proof later in this section, by iteratively applying parabolic interior estimates and approximating $h$ by smooth solutions to (\ref{eq:integraleq}). 
\end{remark}

\begin{remark}
In \cite{Sim}, Simon showed that for a fixed complete background metric there is a complete solution $(g(t))_{t\in(0,T]}$ to the Ricci-DeTurck perturbation equation that converges locally uniformly to the initial data and satisfies
\begin{equation}
\sup_{x\in M}|\nabla^i g(t)|(x) \leq \frac{c_i}{t^{i/2}}
\end{equation}
for all $t\in (0,T]$, where $c_i$ depends on the dimension and the first $i$ covariant derivatives of the curvature of the background metric; see \cite[Theorem $1.1$]{Sim}. Because we will at times study sequences of Ricci-DeTurck flows for which the initial data converges uniformly to some limit, we require slightly different estimates, as in Corollary \ref{cor:C0RDTderivbounds}.
\end{remark}

 \begin{proof}[Proof of Lemma \ref{lemma:RDTexistence}]
 The proof of the first statement follows from the Banach fixed point theorem, and is extremely similar to much of the proof of \cite[Theorem $4.3$]{KL1}, with the modification that the quadratic term is of the form $Q^0 + \nabla^*Q^1$, where now $|Q^0|\leq c(n,\gamma)(|\nabla h|^2 + |\Rm(\bar g)||h|^2)$, by (\ref{eq:ptwiseR0}). Nevertheless, this does not require a significant change to their proof, as the $|h|^2$-term may be absorbed into the $L^\infty$ part of $||h||_X^2$ in the proof of \cite[Lemma $4.1$]{KL1}. Lemma \ref{lemma:easy} is a weighted analog of the first part of \cite[Lemma $4.1$]{KL1}, and we deal with the addition of the $|h|^2$-term more explicitly in the proof of that lemma. The fixed point argument in \cite{KL1} also requires a version of this result for a difference of two solutions, as stated in the second part of \cite[Lemma $4.1$]{KL1}. We prove the requisite result for our setting in Appendix \ref{appendix:iterationscheme}; see (\ref{eq:easylemunweighted}). Note that we do not make any claims about the regularity of such a solution in the statement of this lemma, so we do not need to perform an analog of Koch and Lamm's application of the analytic implicit function theorem.
 
 To prove the second statement, let $\varepsilon, C$ be the constants given by the first statement in the case that $T\sup_{[0,T]}|\Rm|(\bar g) = 1$. Then note that for any smooth Ricci flow defined on a time interval $[0,A]$, it is possible to find some $T = T(\sup_{[0,A]}|\Rm|(\bar g(t)))$ such that (\ref{eq:almosteucl}) and (\ref{eq:boundedcoeffs}) hold, and such that $T\sup_{[0,A]}|\Rm|(\bar g) \leq 1$.
 \end{proof}

Out next objective is to prove Corollary \ref{cor:C0RDTderivbounds}. We first prove a result concerning convergence of the initial data.
\begin{corollary}\label{cor:convergenceto0}
Let $\varepsilon$ be as in the second statement in Lemma \ref{lemma:RDTexistence}. There exists $\varepsilon ''(n) \leq \varepsilon$ such that the following is true:

Let $g_0$ be a $C^0$ metric on $M$ and $\bar g(t)$ be a background Ricci flow with $||g_0 - \bar g_0||_{C^0(M)}\leq \varepsilon''$. Let $g_t$ denote the Ricci-DeTurck flow starting from $g_0$ with respect to a background Ricci flow $\bar g(t)$, in the sense of Lemma \ref{lemma:RDTexistence}. Then $g(t)$ converges uniformly to $g_0$ as $t\to 0$. Thus, $g(t)$ is continuous in time.
\end{corollary}

\begin{proof}
Reduce $\varepsilon$ as determined in Lemma \ref{lemma:iterationcontraction}. Set $\varepsilon'' = \varepsilon/2$. Let $g_0^i$ be a sequence of smooth metrics on $M$ that converge to $g_0$ in $C^0$. Since $||g_0 - \bar g_0||_{C^0(M)} < \varepsilon''$, we have, for sufficiently large $i$, $||g_0^i - \bar g_0||_{C^0(M)} < \varepsilon = 2\varepsilon''$. Let $g^i(t)$ denote the Ricci-DeTurck flow starting from $g_0^i$ with respect to $\bar g(t)$. Since the $g_0^i$ are smooth, $g_t^i \to g_0^i$ uniformly as $t\to 0$. Thus, Lemma \ref{lemma:iterationcontraction} implies
\begin{align*}
||g_t - g_0||_{C^0(M)} &\leq ||g_t - g_t^i||_{C^0(M)} + ||g_t^i - g_0^i||_{C^0(M)} + ||g_0^i - g_0||_{C^0(M)}
\\& \leq c(n)||g_0 - g_0^i||_{C^0(M)} + ||g_t^i - g_0^i||_{C^0(M)} + ||g_0^i - g_0||_{C^0(M)}.
\end{align*}
 We find
\begin{align*}
\limsup_{t\to 0}||g_t - g_0||_{C^0(M)} &\leq (c(n) + 1)||g_0 - g_0^i||_{C^0(M)} + \limsup_{t\to 0}||g_t^i - g_0^i||_{C^0(M)} 
\\& = (c(n) + 1)||g_0 - g_0^i||_{C^0(M)}
\end{align*}
for all $i$. Letting $i\to \infty$, we find
\begin{equation}
\lim_{t\to 0}||g_t - g_0||_{C^0(M)} = 0.
\end{equation}
\end{proof}

\begin{proof}[Proof of Corollary \ref{cor:C0RDTderivbounds}]
First set $\varepsilon' \leq \varepsilon''$, so that Corollary \ref{cor:convergenceto0} implies that $g_t$ is continuous on $M\times [0,T]$.
Now fix $\alpha \in (0, \tfrac{1}{2})$, and measure all balls with respect to $\bar g_0$. We first show the derivative estimates for a smooth $(0,2)$-tensor $u$ satisfying the evolution equation
\begin{equation}
(\partial_t + L) u = (\bar g + u)^{-1} \star (\bar g + u)^{-1} \star \nabla u \star \nabla u +  [(\bar g + u)^{-1} - \bar g^{-1}]\star\Rm(\bar g)\star u + [(\bar g + u)^{-1}-\bar g^{-1}]\star u\star \nabla u,
\end{equation}
where $L$ is the linear operator $L = \Delta_{\bar g_t} + 2\Rm_{\bar g_t}$ as in (\ref{eq:hevolution}). Assume that we know 
\begin{equation}\label{eq:parabolicC0bound}
||u||_{C^0(M\times[0,T])} \leq C(n)||u_0||_{C^0(M)} \leq \frac{1}{2};
\end{equation} 
in our case this assumption will be satisfied due to the bound on the $X$-norm from Lemma \ref{lemma:RDTexistence}, after reducing $\varepsilon'$ if necessary. Then we have
\begin{equation}\label{eq:coeffsok}
|(\bar g + u)^{-1}| \leq c(n), |(\bar g + u)^{-1}-\bar g^{-1}| \leq c(n)
\end{equation}
by (\ref{eq:inverseexpansion1}) and (\ref{eq:inverseexpansion2}).

Define a sequence of domains  $Q_1, Q_2, \ldots$, by $Q_n = B_{\bar g_0}\left(x, 2 - \sum_{i=1}^n\tfrac{1}{2^i}\right)\times[-(1 - \sum_{i=1}^{n}\tfrac{1}{2^{i+1}})^2,0]$, so that $B_{\bar g_0}(x,1)\times[-\tfrac{1}{4}, 0]$ is contained in all of these domains. By arguing as in the proof of Lemma \ref{lemma:interiorests}, it is sufficient to show the estimate 
\begin{equation}\label{eq:C0alpha}
||\nabla^m u||_{C^{0,2\alpha; 0, \alpha}(Q_m)} \leq c_m||u_0||_{L^\infty(M)}
\end{equation} 
on $Q_m$, and to assume that the $Q_m$ are contained in an exponential coordinate chart on which (\ref{eq:boundedcoeffs}) is true. We will show that (\ref{eq:C0alpha}) is true for all $m$ inductively.

By \cite[Proposition $2.5$]{Bam2} and (\ref{eq:boundedcoeffs}) there exist $\varepsilon_1$ and $c_1$ depending only on $\alpha$ and the dimension, such that if $||u||_{L^\infty(Q_1)}< \varepsilon_1$, then $||u||_{C^{2,2\alpha;1,\alpha}(Q_2)} \leq c_1||u||_{L^\infty(Q_1)} \leq c_1 C||u_0||_{L^\infty(M)}$, where $C$ is as in (\ref{eq:parabolicC0bound}). Now reduce $\varepsilon'$ further, so that $\varepsilon' \leq \varepsilon_1/C$ and (\ref{eq:parabolicC0bound}) implies that $||u||_{C^0(M)}\leq \varepsilon_1$. Then (\ref{eq:C0alpha}) is true for $\nabla u$ and $\nabla^2 u$ on $Q_2$.

We estimate higher derivatives inductively. First observe that, because $\nabla$ is the Levi-Civita connection for $\bar g$, we have
\begin{equation}\label{eq:prodderiv}
\nabla(\bar g + u)^{-1} = -(\bar g + u)^{-1}\star (\nabla (\bar g + u))\star (\bar g + u)^{-1} = -(\bar g + u)^{-1}\star \nabla u \star (\bar g + u)^{-1}.
\end{equation}
Now assume that (\ref{eq:C0alpha}) holds on $Q_k$ for all $\nabla^{k}u$ with $k \leq m-1$. Observe that, after commuting the derivatives (see \cite[$(2.1.6)$ and $(2.3.3)$]{Top}), appealing to (\ref{eq:coeffsok}) and (\ref{eq:prodderiv}), and omitting constants that depend only on $n$ and $m$,
\begin{align*}
(\partial_t + L)(\nabla^m u) &= \nabla^m[(\partial_t + L)u] + \sum_{i=0}^{m-1}\nabla^iu \star \nabla \Ric - \nabla \Rm \star \nabla^i u
\\& = \sum_{i=0}^m\sum_{j=0}^{m-i}\sum_{k=0}^{m-i-j}[\nabla^i(\bar g + u)^{-1}]\star[\nabla^j(\bar g + u)^{-1}]\star\nabla^{k+1}u *\nabla^{m-i-j-k+1}u
\\& + \sum_{i=0}^{m}\sum_{j=0}^{m-i}\sum_{k=0}^{m-i-j}[\nabla^i(\bar g + u)^{-1}]\star\Rm(\bar g)\star \nabla^k u
\\& + \sum_{i=0}^{m}\sum_{j=0}^{m-i}[\nabla^i((\bar g + u)^{-1} - \bar g^{-1})]\star\nabla^ju \star \nabla^{m-i-j+1}u
\\& + \sum_{i=0}^{m-1}\nabla^iu \star \nabla \Ric - \nabla \Rm \star \nabla^i u
\\& = \nabla^{m+1}u\star\nabla u + \text{lower order terms},
\end{align*}
where ``lower order terms'' refers to terms involving derivatives of the curvature of the background metric of order at most $m$ and derivatives of $u$ of order strictly less than $m+1$.

By \cite[Theorem $8.11.1$]{Kry2}, there exists $N$ depending on $\alpha$, $n$, and $||\Rm(\bar g)||_{C^{0,2\alpha;0,\alpha}(Q_{m-1})}$ such that
\begin{equation}
||\nabla^{m-2}u||_{C^{2,2\alpha; 1,\alpha}(Q_{m})} \leq N(||\partial_t u + L u||_{C^{0,2\alpha,0,\alpha}(Q_{m-1})} + ||u||_{C^0(Q_{m-1})}).
\end{equation}
Then, inserting our analysis from above, we find
\begin{equation}
\begin{split}
||\nabla^{m-2}u||_{C^{2,2\alpha; 1,\alpha}(Q_{m})} &\leq c ||\nabla^{m-1}u||_{C^{0,2\alpha,0,\alpha}(Q_{m-1})}||\nabla u||_{C^{0,2\alpha,0,\alpha}(Q_{m-1})} + \text{ lower order terms}
\\& \leq c||u_0||_{L^\infty(M)},
\end{split}
\end{equation}
where $c$ depends on $n$, $m$, $\alpha$, and $||\nabla^k\Rm(\bar g)||_{C^{0,2\alpha; 0, \alpha}(Q_{m-1})}$ for $k\leq m$, whence follows (\ref{eq:C0alpha}) for $\nabla^m u$. After arguing as in the proof of Lemma \ref{lemma:interiorests} we find
\begin{equation}\label{eq:smoothderivbds}
||\nabla^k u||_{C^0(B_{\bar g_0}(x,\sqrt{t}/2)\times[\tfrac{7t}{8},t])} \leq \frac{c}{t^{k/2}}||u_0||_{C^0(B_{\bar g_0}(x,\sqrt{t})\times[t/2,t])},
\end{equation}
for the same constant $c$.

Now let $g_0^i$ be a sequence of $C^\infty$ metrics on $M$ that converge to $g_0$ in $C^0$, and let $g_t^i$ be the Ricci-DeTurck flows starting from $g_0^i$ with respect to $\bar g (t)$. 
Fix $t>0$ and $k$. Then, since $g_0^i$ and hence $g_t^i$ is smooth, (\ref{eq:smoothderivbds}) implies that
\begin{equation*}
||\nabla^k(g^i - \bar g)||_{C^0(B_{\bar g_0}(x, \sqrt{t}/2)\times[\tfrac{7t}{8}, t])} \leq \frac{c}{t^{k/2}}||g_0^i - \bar g(0)||_{C^0(M)},
\end{equation*}
so that, on a subsequence, $\nabla^k(g^i - \bar g)$ converges locally uniformly on $[\tfrac{7t}{8}, t]$ to some continuous limit. In particular, $g_t$ is smooth for fixed $t>0$, and $g_t^i \to g_t$ in $C^\infty_{\loc}(M \times (0,T])$; this proves the second statement of Corollary \ref{cor:C0RDTderivbounds}. Moreover, after taking limits, we have
\begin{equation}\label{eq:derivbdsunscaled}
|\nabla^k(g_t - \bar g(t))| \leq\frac{c}{t^{k/2}}||g_0 - \bar g(0)||_{L^\infty(M)}.
\end{equation}
\end{proof}
 
\section{Stability for Ricci-DeTurck flows with initial metrics that agree to greater than second order}\label{sec:fixedptconstruction}

The purpose of this section is to use the weighted norms introduced in Definition \ref{def:norms} to bound the growth under the Ricci-DeTurck flow of the difference of two $C^0$ metrics that agree to greater than second order about a point at the initial time.

The main result of the section is the following theorem:
\begin{theorem}\label{thm:fixedpointexistence}
Suppose $g_0'$ and $g_0''$ are two $C^0$ metrics on closed $n$-manifolds $M'$ and $M''$ respectively. Suppose also that there exists a locally defined diffeomorphism $\phi: U\to V$, where $U\subset M'$ is some neighborhood of a point $x_0' \in M''$ and $V\subset M''$ is a neighborhood of $x_0'' = \phi(x_0')$, and  that $\phi^*g''$ agrees to greater than second order with $g'$ around $x_0'$, in the sense of Definition \ref{def:norms}.

Then there exist Ricci-DeTurck flows $(g_t')_{t\in(0,T']}$ and $(g_t'')_{t\in (0,T'']}$ starting from $g_0'$ and $g_0''$ in the sense of Lemma \ref{lemma:RDTexistence} with respect to background Ricci flows $\bar g'(t)$ and $\bar g''(t)$ respectively, which satisfy
\begin{equation}
\sup_{0<t<T}||g_t' - \phi^*g_t''||_{L^\infty(B(x_0', R))} \leq (R + \sqrt{t})^{2+\eta}\big[c(n)||w_0(0)(h_0' - h_0'')||_{L^\infty(B(x_0', 4R))}  + c(n, R)\big],
\end{equation}
for all $R< \sqrt{T}/4$, where $T = T(n, \sup_{[0,T']}|\Rm(\bar g'_t)|, \sup_{[0,T'']}|\Rm(\bar g''_t)|, \bar g' - \phi^*\bar g'') \leq \min\{T', T''\}$ is small enough so that (\ref{eq:almosteucl}) and (\ref{eq:boundedcoeffs}) hold, $c$ is some finite positive number, and $w_0(0)$ is as in (\ref{eq:weight}).
\end{theorem}

From Theorem \ref{thm:fixedpointexistence} follows this key result about the scalar curvatures near the initial time of such solutions.
\begin{theorem}\label{thm:weakagreement}
Suppose $g_0'$ and $g_0''$ are two $C^0$ metrics on closed $n$-manifolds $M'$ and $M''$ respectively. Suppose also that there exists a locally defined diffeomorphism $\phi: U\to V$, where $U\subset M'$ is some neighborhood of a point $x_0' \in M'$ and $V\subset M''$ is a neighborhood of $x_0'' = \phi(x_0')$, and  that $\phi^*g''$ agrees to greater than second order with $g'$ around $x_0'$, in the sense of Definition \ref{def:norms}.

Let $g'_t$ and $g''_t$ be solutions to the Ricci-DeTurck flow equation on $M'$ and $M''$ respectively, starting from $g'$ and $g''$, with respect to background Ricci flows $\bar g'(t)$ and $\bar g''(t)$, $t\in [0,T]$, as described in Theorem \ref{thm:fixedpointexistence}. Then
\begin{equation}
\sup_{C>0}\left(\limsup_{t\to 0}\left(\sup_{B_{\bar g_0'}(x_0', Ct^\beta)} |R^{g'}|(t)\right)\right) = \sup_{C>0}\left(\limsup_{t\to 0}\left(\sup_{B_{\bar g_0''}(x_0'', Ct^\beta)} |R^{g''}|(t)\right)\right),
\end{equation}
for $\beta \in (1/(2+\eta), 1/2)$, where $R^{g'}$ and $R^{g''}$ denote the scalar curvatures with respect to $g'$ and $g''$ respectively.
\end{theorem}

We record here some observations concerning the weight $w_a(x,t)$:
Suppose that $y\in B(x, 2\sqrt{t})$. Then 
\begin{equation}\label{eq:wtimeslicecomparison}
\begin{split}
\frac{1}{(d(x,x_0) + \sqrt{t} + a)^{2+\eta}} &\geq \frac{1}{(d(y,x_0) + d(x,y) + \sqrt{t} + a)^{2+\eta}} \geq \frac{1}{(d(y,x_0) + 3\sqrt{t} + a)^{2+\eta}} 
\\& \geq \frac{3^{-2-\eta}}{(d(y,x_0) + \sqrt{t} + a)^{2+\eta}} 
\end{split}
\end{equation}
so $w_a(x,t) \geq 3^{-2-\eta}w_a(y,t)$. Moreover, $y\in B(x, 2\sqrt{t})$ implies that $x\in B(y, 2\sqrt{t})$, so exchanging $x$ and $y$ in the previous analysis implies $w_a(y,t)\geq 3^{-2-\eta}w_a(x,t)$. For $s\leq t$,
\begin{equation}\label{eq:wspaceslicecomparison}
w_a(x,s) \geq w_a(x,t).
\end{equation}
Finally, note that for all $(x,t)\in M\times(0,T)$ and $(y,s) \in B(x, \sqrt{t})\times(\tfrac{t}{2}, t)$,
$w_a(x,t)$ is comparable to $w_a(y,s)$, i.e. there is a universal constant $C$ depending only on $\eta$ such that
\begin{equation}\label{eq:wparabolicballcomparison}
\frac{1}{C}w_a(x,t)\leq w_a(y,s) \leq Cw_a(x,t).
\end{equation}
This holds because, by (\ref{eq:wtimeslicecomparison}) and (\ref{eq:wspaceslicecomparison}), we have
\begin{equation*}
3^{-2-\eta}w_a(x,t) \leq w_a(y,t) \leq w_a(y,s)
\end{equation*}
and
\begin{equation*}
w_a(y,s) \leq w_a(y,t/2) \leq \sqrt{2}^{2+\eta}w_a(y,t) \leq (3\sqrt{2})^{2+\eta}w_a(x,t).
\end{equation*}

We now prove some pointwise bounds for solutions to the homogeneous problem.
\begin{lemma}\label{lemma:homogeneoussup}
Fix a smooth background Ricci flow $\bar g(t)$ defined for $t\in [0,T]$, where $T$ is small enough so that (\ref{eq:almosteucl}) and (\ref{eq:boundedcoeffs}) hold, and let $\bar K_t$ be the corresponding heat kernel for the linear part of (\ref{eq:hevolution}). Suppose that $u_0$ is a $C^0$ $(0,2)$-tensor on $M$, and let 
\begin{equation*}
u(x,t) = \int_{M}\bar K(x,t;y,0)u_0(y)d_{\bar g_0}(y).
\end{equation*}
We have
\begin{equation}\label{eq:homogeneoussup}
\sup_{0<t<T}||w_a(t)u(t)||_{L^{\infty}(M)} \leq c||w_a(0)u_0||_{L^\infty(M)}
\end{equation}
and
\begin{equation}\label{eq:gradhomogeneoussup}
\sup_{0<t<T}||w_a(t)\nabla u(t)||_{L^{\infty}(M)} \leq \frac{c}{\sqrt{t}}||w_a(0)u_0||_{L^\infty(M)},
\end{equation}
\end{lemma}
where $c = c(n,T\sup_{[0,T]}|\Rm|(\bar g(t)))$.
\begin{proof}
Fix $x\in M$ and $t>0$. If $w_a(x,t) = 1$, then using Corollary \ref{cor:exponentialboundnonscalar} and Bishop-Gromov, we have
\begin{align*}
|w(x,t)u(x,t)| &= |u(x,t)| \leq c\int_{M}t^{-n/2}\exp\left(-\frac{d^2_{\bar g_0}(x,y)}{Dt}\right)|u_0(y)|d_{\bar g_0}(y) 
\\& \leq c||u_0||_{L^\infty(M)}\int_M\exp\left(-\frac{d^2_{\hat{\bar g}_0}(x,y)}{D}\right)d_{\hat{\bar g}_0}(y)
\\& \leq c||u_0||_{L^\infty(M)}\int_0^\infty\vol(B_s^{-t\sup_{[0,T]}|\Rm|(\bar g)})\exp\left(-\frac{s^2}{D}\right)ds
\\& \leq c||u_0||_{L^\infty(M)} \leq c||w_a(0)u_0||_{L^\infty(M)},
\end{align*}
where $\hat{\bar g}_0:= (1/t)\bar g_0$.
Similarly, we find
\begin{align*}
|w(x,t)\nabla u(x,t)| &= |\nabla u(x,t)| \leq c\int_{M}t^{-n/2}\frac{\exp\left(\frac{-d^2_{\bar g_0}(x,y)}{Dt}\right)}{\sqrt{t}}|u_0(y)|d_{\bar g_0}(y) 
\\& \leq \frac{c}{\sqrt{t}}||u_0||_{L^\infty(M)} \leq \frac{c}{\sqrt{t}}||w_a(0)u_0||_{L^\infty(M)}.
\end{align*}
Now suppose $w_a(x,t) = (d_{\bar g_0}(x,x_0) + \sqrt{t} + a)^{-2-\eta}$. First observe that $1 \leq (d_{\bar g_0}(y,x_0)+ \sqrt{t} + a)^{2+\eta}w_a(y,t)$ for all $y\in M$, $t>0$. We appeal to (\ref{eq:ptwiseheatkernel}), (\ref{eq:heatkernelawayfrom0}), and Jensen's inequality to find
\begin{align*}
&w_a(x,t)|u(x,t)| \leq (d_{\bar g_0}(x,x_0) + \sqrt{t} + a)^{-2 -\eta}\bigg[\int_{B_{\bar g_0}(x,1)} |\bar K_{t}(x,y)||u_0(y)|d_{\bar g_0}(y) 
\\& \qquad \qquad \qquad \qquad \qquad + \int_{M\setminus B_{\bar g_0}(x,1)}|\bar K_t(x,y)||u_0(y)|d_{\bar g_0}(y)\bigg]
\\& \leq C(d(x,x_0) + \sqrt{t} + a)^{-2 -\eta}\left(\int_{B_{\bar g_0}(x,1)}t^{-n/2}\exp\left(\frac{-d^2_{\bar g_0}(x,y)}{Dt}\right)|u_0(y)|d_{\bar g_0}(y)\right) 
\\& \qquad \qquad \qquad \qquad + C\sqrt{t}^{-2-\eta}e^{-\frac{1}{2Dt}}||u_0||_{L^\infty(M)}
\\& \leq C\left(\int_{B_{\bar g_0}(x,1)}t^{-n/2}\exp\left(\frac{-d^2_{\bar g_0}(x,y)}{Dt}\right)\frac{(d_{\bar g_0}(y,x_0) + \sqrt{t} + a)^{2+\eta}w_a(y)}{(d_{\bar g_0}(x,x_0) + \sqrt{t} + a)^{2+\eta}}|u_0(y)|d_{\bar g_0}(y)\right) 
\\& \qquad \qquad \qquad \qquad  + C\sqrt{t}^{-2-\eta}e^{-\frac{1}{2Dt}}||u_0||_{L^\infty(M)} 
\\& \leq C||w_a(0)u_0||_{L^\infty(M)}\left(\int_{B_{\bar g_0}(x,1)}t^{-n/2}\exp\left(\frac{-d_{\bar g_0}^2(x,y)}{Dt}\right)\frac{(d_{\bar g_0}(x,y) + d_{\bar g_0}(x,x_0) + \sqrt{t} + a)^{2+\eta}}{(d_{\bar g_0}(x,x_0) + \sqrt{t} + a)^{2+\eta}}d_{\bar g_0}(y)\right) 
\\& \qquad \qquad \qquad \qquad  + C\sqrt{t}^{-2-\eta}e^{-\frac{1}{2Dt}}||u_0||_{L^\infty(M)}
\\& \leq C||w_a(0)u_0||_{L^\infty(M)}\bigg(\int_{B_{\bar g_0}(x,1)}t^{-n/2}\exp\left(\frac{-d_{\bar g_0}^2(x,y)}{Dt}\right)\left(1 + \frac{d_{\bar g_0}(x,y)^{2+\eta}}{\sqrt{t}^{2+\eta}}\right)d_{\bar g_0}(y)\bigg) 
\\&\qquad \qquad \qquad \qquad  + C\sqrt{t}^{-2-\eta}e^{-\frac{1}{2Dt}}||u_0||_{L^\infty(M)}
\\& \leq C||w_a(0)u_0||_{L^\infty(M)}\int_M\exp\left(-\frac{d^2_{\bar g_0'}(x,y)}{D}\right)\left(1 + d_{\bar g_0'}^{2+\eta}(x,y)\right)d_{\bar g_0'}(y) + C\sqrt{t}^{-2-\eta}e^{-\frac{1}{2Dt}}||u_0||_{L^\infty(M)}
\\&\leq C||w_a(0)u_0||_{L^\infty(M)}\int_0^\infty\vol(B_s^{-t\sup_{[0,T]}|\Rm|(\bar g)})\exp\left(-\frac{s^2}{D}\right)(1 + s^{2+\eta})ds + C\sqrt{t}^{-2-\eta}e^{-\frac{1}{2Dt}}||u_0||_{L^\infty(M)}
\\& \leq C||w_a(0)u_0||_{L^\infty(M)}
\end{align*}
The proof to show (\ref{eq:gradhomogeneoussup}) is similar.
\end{proof}

We now bound the $\tilde Y_a$-norm of the quadratic term of (\ref{eq:hevolution}), cf. \cite[Lemma $4.1$]{KL1}.
\begin{lemma}\label{lemma:easy}
Fix a smooth background Ricci flow $\bar g(t)$ defined for $t\in [0,T]$. If $T$ is sufficiently small so that (\ref{eq:almosteucl}) holds, then for every $0<\gamma <1$ and every $h\in X^\gamma$ we have the estimate
\begin{equation}
||Q^0[h]+ \nabla^*Q^1[h]||_{\tilde Y_a(B_{\bar g_0}(x, r))} \leq c(n,\gamma, T\sup_{[0,T]}|\Rm|(\bar g(t)))||h||_{X(B(x,r))}||h||_{\tilde X_a(B_{\bar g_0}(x,r))}.
\end{equation}
\end{lemma}

\begin{proof} 
Let $h\in X^\gamma(B(x,r))$. Then, adopting the constants from (\ref{eq:ptwiseR0}), and using (\ref{eq:wtimeslicecomparison}) and (\ref{eq:wspaceslicecomparison}), we find
\begin{align*}
&||R^0[h]||_{\tilde Y^0_a(B(x,r))} \leq c(n,\gamma) w_{a}(x,r^2)\bigg(r^{-n}|||\nabla h|^2||_{L^1(B(x,r)\times(0,r^2))} + r^{\tfrac{4}{n+4}}|||\nabla h|^2||_{L^{\tfrac{n+4}{2}}(B(x,r)\times(\tfrac{r^2}{2}, r^2))} 
\\& \qquad \qquad + r^{-n}\sup_{[0,T]}|\Rm|(\bar g)||h||^2_{L^\infty(B(x,r)\times(0,r^2))}r^{n + 2} + r^{\tfrac{4}{n+4}}\sup_{[0,T]}|\Rm|(\bar g)||h||^2_{L^\infty(B(x,r)\times(\tfrac{r^2}{2},r^2))}(r^{n+2})^{\tfrac{2}{n+4}}\bigg)
\\& \leq c(n,\gamma) w_{a}(x,r^2)\bigg(r^{-n/2}||\nabla h||_{L^2(B(x,r)\times(0,r^2))} + r^{\tfrac{2}{n+4}}||\nabla h||_{L^{n+4}(B(x,r)\times(\tfrac{r^2}{2}, r^2))}
 \\& \qquad+ T\sup_{t\in[0,T]}|\Rm|(\bar g(t))\sup_{0<t<T}||w_a(t)h_t||_{L^\infty(B(x,r))}\bigg) \times \bigg(r^{-n/2}||\nabla h||_{L^2(B(x,r)\times(0,r^2))} 
 \\& \qquad \qquad \qquad + r^{\tfrac{2}{n+4}}||\nabla h||_{L^{n+4}(B(x,r)\times(\tfrac{r^2}{2}, r^2))} + ||h_t||_{L^\infty(B(x,r))}\bigg)  
 \\& \qquad \qquad \leq c(n,\gamma, T\sup_{[0,T]}|\Rm|(\bar g(t)))||h||_{X(B(x,r))}||h||_{\tilde X_a(B(x,r))}, 
\end{align*}
and, by (\ref{eq:ptwiseR1}),
\begin{align*}
||R^1[h]||_{\tilde Y^1_a(B(x,r))} &\leq c(n,\gamma)w_{a}(x,r^2)\bigg(r^{-\tfrac{n}{2}}|||h||\nabla h|||_{L^2(B(x,r)\times(0,r^2))} 
\\& \qquad \qquad + r^{\tfrac{2}{n+4}}|||h||\nabla h|||_{L^{n+4}(B(x,r)\times(\tfrac{r^2}{2}, r^2))}\bigg)
\\& \leq c \sup_{0<t<T}||h_t||_{L^\infty(B(x,r))}||\nabla h||_{\tilde Y^1_a(B(x,r))} \leq c||h||_{X(B(x,r))}||h||_{\tilde X_a(B(x,r))}.
\end{align*}
\end{proof}

We now bound the $\tilde X_a$-norm of the homogeneous part of a solution to (\ref{eq:integraleq}); cf. \cite[Lemma $2.2$]{KL1}
\begin{lemma}\label{lemma:homogeneous}
Fix a smooth background Ricci flow $\bar g(t)$ defined for $t\in [0,T]$, and let $\bar K_t$ be the corresponding heat kernel for the linear part of (\ref{eq:hevolution}). Let $h_0$ be a $C^0$ $(0,2)$-tensor on $M$ and $h$ be the time-dependent family of $(0,2)$-tensors on $M$ given by
\begin{equation*}
h(x,t) = \int_M \bar K(x,t;y,0)h_0(y)d_{\bar g_0}(y).
\end{equation*}
Then, for $T$ sufficiently small so that (\ref{eq:almosteucl}) and (\ref{eq:boundedcoeffs}) hold, we have 
\begin{equation}
||h||_{\tilde X_a} \leq c||w_a(0)h_0||_{L^\infty(M)},
\end{equation}
where $c = c(n, T\sup_{[0,T]}|\Rm|(\bar g(t)))$.
\end{lemma}
\begin{proof}
The bound on the $L^\infty$-term is handled by Lemma \ref{lemma:homogeneoussup}. To bound the $L^2$-term, we multiply by a cutoff function and integrate by parts, as follows. Let $\eta:[0,\infty) \to [0,1]$ be a smooth cutoff function such that $\eta \equiv 1$ on $[0,1]$ and $\eta \equiv 0$ on $[2,\infty)$, with gradient bounded by $2$, say. For all $x\in M$ and $t>0$ let $\eta_{x,t} = \eta(d_{\bar g_0}(x,\cdot)/\sqrt{t})$, so that $|\nabla \eta_{x,t}|_{\bar g_t} \leq \frac{c(T\sup_{[0,T]}|\Rm|(\bar g(t)))}{\sqrt{t}}$. We pair $\eta_{x,r^2}$ with the evolution equation for $h$ and integrate by parts over $M\times[0,r^2]$ to find:
\begin{align*}
\frac{1}{2}\int_M \eta_{x,r^2}^2&|h_{r^2}|_{\bar g_0}^2 - \frac{1}{2}\int_M \eta_{x,r^2}^2|h_0|_{\bar g_0}^2 = \int_0^{r^2}\int_M\eta^2_{x, r^2} \langle\partial_s h, h\rangle_{\bar g_0} = \int_0^{r^2}\int_M \eta^2_{x,r^2}\langle Lh, h \rangle_{\bar g_0} 
\\& = \int_0^{r^2}\int_M \eta^2_{x,r^2}\langle\Delta^{\bar g_s} + \Rm^{\bar g_s}(h), h\rangle_{\bar g_0}
\\& \leq e^{T\sup_{[0,T]}|\Rm|(\bar g(t))}\int_0^{r^2}\int_M\eta_{x,r^2}^2\langle\Delta^{\bar g_s}h, h\rangle_{\bar g(s)} + e^{T\sup_{[0,T]}|\Rm|(\bar g(t))}\int_0^{r^2}\int_M\eta_{x,r^2}^2\langle\Rm^{\bar g_s}(h), h\rangle_{\bar g_s}
\\& \leq -c\int_0^{r^2}\int_M\eta_{x,r^2}|\nabla h|_{\bar g_s}^2 + 2c\int_0^{r^2}\int_M\eta_{x,r^2}|\nabla \eta_{x,r^2}|_{\bar g_s}|\nabla h|_{\bar g_s} |h|_{\bar g_s} + c\int_0^{r^2}\int_M \eta^2_{x,r^2}|\Rm^{\bar g_s}|_{\bar g_s}|h|_{\bar g_s}^2
\\& \leq -c\int_0^{r^2}\int_M\eta_{x,r^2}|\nabla h|_{\bar g_0}^2 + 2c\int_0^{r^2}\int_M\eta_{x,r^2}|\nabla \eta_{x,r^2}|_{\bar g_0}|\nabla h|_{\bar g_0} |h|_{\bar g_0} + c\int_0^{r^2}\int_M \eta^2_{x,r^2}|\Rm^{\bar g_s}|_{\bar g_0}|h|_{\bar g_0}^2
\end{align*}
so, using Young's inequality, we find
\begin{align*}
0 &\leq \frac{1}{2}\int_M\eta_{x,r^2}^2|h_0|^2 - c\int_0^{r^2}\int_M\eta_{x,r^2}^2|\nabla h|^2 + 2c\int_0^{r^2}\int_M\eta_{x,r^2}|\nabla\eta_{x,r^2}||\nabla h||h| + c\int_0^{r^2}\int_M \eta^2_{x,r^2}|\Rm||h|^2
\\& \leq  \frac{1}{2}\int_M\eta_{x,r^2}^2|h_0|^2 - c\int_0^{r^2}\int_M\eta_{x,r^2}^2|\nabla h|^2 + \frac{c}{2}\int_0^{r^2}\int_M \eta_{x,r^2}^2|\nabla h|^2 
\\& \qquad \qquad + 8c\int_0^{r^2}\int_M |\nabla \eta_{x,r^2}|^2|h|^2 + c\int_0^{r^2}\int_M\eta^2_{x,r^2}|\Rm||h|^2
\end{align*}

Thus we find
\begin{align*}
\frac{c}{2}\int_0^{r^2}\int_M \eta_{x,r^2}^2|\nabla h| & \leq \frac{1}{2}\int_M\eta_{x,r^2}^2|h_0|^2  + 8c\int_0^{r^2}\int_M |\nabla \eta_{x,r^2}|^2|h|^2 + c\int_0^{r^2}\int_M\eta^2_{x,r^2}|\Rm||h|^2,
\end{align*}
so, by (\ref{eq:wtimeslicecomparison}) and (\ref{eq:wspaceslicecomparison}), we have
\begin{align*}
&w_{a}(x,r^2)^2r^{-n}\int_0^{r^2}\int_{B(x,r)}|\nabla h|^2 \leq cw_{a}(x,r^2)^2r^{-n}\int_{B(x,2r)}|h_0|^2 + \frac{cw_{a}(x,r^2)^2r^{-n}}{r^2}\int_0^{r^2}\int_{B(x, 2r)}|h|^2
\\& \qquad \qquad \qquad + cw_a(x,r^2)^2r^{-n}\int_0^{r^2}\int_{B(x,2r)}\sup_{[0,T]}|\Rm|(\bar g)|h|^2
\\& \leq cr^{-n}\int_{B(x,2r)}|w_{a}(r^2)h_0|^2 + \frac{cr^{-n}}{r^2}\int_0^{r^2}\int_{B(x,2r)}|w_{a}(r^2)h|^2
 + cr^{-n}\int_0^{r^2}\int_{B(x,2r)}\sup_{[0,T]}|\Rm|(\bar g)|w_a(r^2)h|^2
\\& \leq cr^{-n}\int_{B(x,2r)}|w_{a}(0)h_0|^2 + \frac{cr^{-n}}{r^2}\int_0^{r^2}\int_{B(x,2r)}|w_ah|^2
  + cr^{-n}\int_0^{r^2}\int_{B(x,2r)}\sup_{[0,T]}|\Rm|(\bar g)|w_ah|^2
\\& \leq cr^{-n}(2r)^n||w_a(0)h_0||^2_{L^\infty(M)} + cr^{-n}r^{-2}r^2(2r)^n||w_ah||^2_{L^\infty(M\times[0,r^2])} \\& \qquad \qquad \qquad + cr^{-n}(r^2\sup_{[0,T]}|\Rm|(\bar g(t)))(2r)^n||w_ah||_{L^\infty}^2
\\& \leq c||w_a(0)h_0||^2_{L^\infty(M)},
\end{align*}
where the last inequality follows from the pointwise bound (\ref{eq:homogeneoussup}) and (\ref{eq:almosteucl}). The $L^{n+4}$-term follows from (\ref{eq:gradhomogeneoussup}) and (\ref{eq:wparabolicballcomparison}) by integration, as
\begin{align*}
w_a(x,r^2)r^{\tfrac{2}{n+4}}&\left(\int_{r^2/2}^{r^2}\int_{B(x,r)}|\nabla h|^{n+4}\right)^{\tfrac{1}{n+4}} \leq cr^{\tfrac{2}{n+4}}\left(\int_{r^2/2}^{r^2}\int_{B(x,r)}|w_a(y,s)\nabla h(y,s)|^{n+4}\right)^{\tfrac{1}{n+4}}
\\&\leq c||w_a(0)h_0||_{L^\infty(M)}r^{\tfrac{2}{n+4}}\left(\int_{r^2/2}^{r^2}\int_{B(x,r)}(\sqrt{s})^{-n-4}\right)^{\tfrac{1}{n+4}} 
\\& \leq c||w_a(0)h_0||_{L^\infty(M)}r^{\tfrac{2}{n+4}}\left(r^{-n-4}r^{n+2}\right)^{\tfrac{1}{n+4}} \leq c||w_a(0)u_0||_{L^\infty(M)},
\end{align*}
also by (\ref{eq:almosteucl}).
\end{proof}

We now show (cf. \cite[Lemma $4.2$]{KL1}):
\begin{lemma}\label{lemma:hard}
Fix a smooth background Ricci flow $\bar g(t)$ defined for $t\in [0,T]$, and let $\bar K_t$ be the corresponding heat kernel for the linear part of (\ref{eq:hevolution}). Let $h$ be the time-dependent family of $(0,2)$-tensors on $M$ given by
\begin{equation*}
h(x,t) = \int_0^t\int_M\bar K(x,t;y,s)Q(y,s)d_{\bar g_s}(y)ds,
\end{equation*}
where $Q= Q^0 + \nabla^*Q^1 \in \tilde Y_a$.
Then, if $T$ is sufficiently small so that (\ref{eq:almosteucl}) and (\ref{eq:boundedcoeffs}) hold, 
\begin{equation}
||h||_{\tilde X_a(B(x,r))} \leq c||Q||_{\tilde Y_a}.
\end{equation}
for all $x\in M$, $0<r^2 < T$, where $c= c(n, T\sup_{[0,T]}|\Rm|(\bar g(t)))$.
\end{lemma}

Note here that the $\tilde X_a$-norm is localized to $B(x,r)$, but the $\tilde Y_a$-norm is not local.

\begin{proof}
We estimate each term of the norm separately, making use of the representation formula (\ref{eq:representationformula}). Let $(x',t)\in B_{\bar g_0}(x,r)\times(0,T)$, and let $\Omega(x',t)= B_{\bar g_0}(x', \sqrt{t})\times(\tfrac{t}{2}, t)$. Then
\begin{align*}
|w_a(x',t)h_t(x')| &\leq w_a(x',t)\left|\int_{\Omega(x',t)}\bar K(x',t;y,s)Q(y,s)d_{\bar g_s}(y_ds\right| 
\\& \qquad + w_a(x',t)\left|\int_{M\times[0,t]\setminus \Omega(x',t)}\bar K(x',t; y,s)Q(y,s)d_{\bar g_s}(y)ds\right| =: I + II.
\end{align*}

Using H\"older's inequality, we find that
\begin{align*}
I &\leq w_a(x',t)||\bar K(x',t;\cdot,\cdot)||_{L^{\tfrac{n+4}{n+2}}(\Omega(x',t))}||Q^0||_{L^{\tfrac{n+4}{2}}(\Omega(x',t))} 
\\& \qquad\qquad + w_a(x',t)||\nabla \bar K(x', t; \cdot,\cdot)||_{L^{\tfrac{n+4}{n+3}}(\Omega(x',t))}||Q^1||_{L^{n+4}(\Omega(x',t))}
\\& \qquad \qquad \qquad \leq c||Q||_{\tilde Y_a(B(x,r))},
\end{align*}
where the last inequality holds by (\ref{eq:ptwiseheatkernelderivbounds}) as follows:
\begin{align*}
\int_{t/2}^{t}\int_{B(x',\sqrt{t})}|\bar K(x',t; y,s)|^{\frac{n+4}{n+2}}d\bar g_s(y)ds & \leq c\int_0^{t/2}\int_{B(x',\sqrt{t})}(d(x',y) + \sqrt{s})^{-n\frac{n+4}{n+2}}d\bar g_0(y)ds
\\& \leq ct^{\frac{n}{2}}t^{-\frac{n(n+4)}{2(n+2)} + 1} = ct^{\frac{2}{n+2}},
\end{align*}
and similarly for $\nabla \bar K$.

We now estimate $II$ by way of (\ref{eq:ptwiseheatkernelawayfromball}). Let $\{z_i\}$ be a maximal collection of points in $M$ such that the balls $B_{\bar g_0}(z_i, \sqrt{t}/2)$ are pairwise disjoint. It follows that $\{B_{\bar g_0}(z_i, \sqrt{t})\}$ is a cover of $M$: otherwise there would exist some $x\in M$ such that $d(x,z_i)\geq \sqrt{t}$ for all $i$, and if $y \in B_{\bar g_0}(x, \sqrt{t}/2)$ then $d_{\bar g_0}(y,z_i) \geq d_{\bar g_0}(x, z_i) - d_{\bar g_0}(x,y) > \sqrt{t} - \sqrt{t/2} = \sqrt{t}/2$, so $B_{\bar g_0}(x, \sqrt{t}/2) \cap B_{\bar g_0}(z_i, \sqrt{t}/2) = \varnothing$.

Therefore, we have
\begin{align*}
II &\leq C\sum_{i=1}^{\infty}w_a(x',t)\int_0^t\int_{B(z_i,\sqrt{t})}t^{-n/2}\exp\left(-\frac{d^2(x',y)}{2Dt}\right)(|Q^0_s|(y) + t^{-1/2}|Q^1_s|(y))d_{\bar g_s}(y)ds
\\& \leq Cw_a(x',t)\sum_{i=1}^{\infty}\exp\left(-\frac{d^2(x',z_i)}{2Dt}\right)\int_0^t\int_{B(z_i,\sqrt{t})}t^{-n/2}(|Q^0_s|(y) + t^{-1/2}|Q^1_s|(y))d_{\bar g_s}(y)ds.\stepcounter{equation}\tag{\theequation}\label{eq:IIsummation}
\end{align*}

We now compare $w_a(x',t)$ with $w_a(z_i,t)$, for a fixed $i$.  If $w_a(x',t) = 1$, then certainly $w_a(x',t)\leq w_a(z_i,t)$.  If $z_i\in B(x', \sqrt{t})$, then, by (\ref{eq:wtimeslicecomparison}) we have $w_a(x',t)\leq c w_a(z_i,t)$. Otherwise, $d(x', z_i)> \sqrt{t}$. Then
\begin{align*}
\frac{1}{(d(x',x_0) + \sqrt{t} + a)^{2+\eta}}&\exp\left(-\frac{d^2(x',z_i)}{Dt}\right) 
\\& \leq \frac{1}{(d(x',x_0) + \sqrt{t} + a)^{2+\eta}}\exp\left(-\frac{d^2(x',z_i)}{2Dt}\right)c(2+\eta)\left(\frac{\sqrt{t}}{d(x', z_i) + \sqrt{t}}\right)^{2+\eta}
\\& =c\exp\left(-\frac{d^2(x',z_i)}{2Dt}\right)\left(\frac{\sqrt{2Dt}}{(d(x',x_0) + \sqrt{t} + a)(d(x', z_i) + \sqrt{t})}\right)^{2+\eta}
\\& \leq c\exp\left(-\frac{d^2(x',z_i)}{2Dt}\right)\left(\frac{c}{d(x',x_0) + d(x', z_i) + \sqrt{t} + a}\right)^{2+\eta}
\\& = c\exp\left(-\frac{d^2(x',z_i)}{2Dt}\right)\frac{1}{(d(z_i,x_0) + \sqrt{t} + a)^{2+\eta}}.
\end{align*}
Therefore,
\begin{align*}
w_a(x', t)\exp\left(-\frac{d^2(x',z_i)}{Dt}\right) &= \max\{(d(x',x_0) + \sqrt{t} + a)^{-2-\eta}, 1\}\exp\left(-\frac{d^2(x',z_i)}{Dt}\right) 
\\& = \max\left\{\exp\left(-\frac{d^2(x',z_i)}{Dt}\right)(d(x',x_0) + \sqrt{t} + a)^{-2-\eta}, \exp\left(-\frac{d^2(x',z_i)}{Dt}\right)\right\}
\\& \leq \max\left\{c\exp\left(-\frac{d^2(x',z_i)}{2Dt}\right)(d(z_i,x_0) + \sqrt{t} + a)^{-2-\eta}, c\exp\left(-\frac{d^2(x',z_i)}{2Dt}\right)\right\}
\\& =c\exp\left(-\frac{d^2(x',z_i)}{2Dt}\right)\max\left\{(d(z_i,x_0) + \sqrt{t} + a)^{-2-\eta}, 1\right\}
\\& = c\exp\left(-\frac{d^2(x',z_i)}{2Dt}\right) w_a(z_i,t). \stepcounter{equation}\tag{\theequation}\label{eq:expweight}
\end{align*}

\noindent \textbf{Claim:} 
\begin{equation}\label{eq:expclaim}
\sum_{z_i}\exp\left(-\frac{d^2_{\bar g_0}(x', z_i)}{2Dt}\right) \leq c(n, T\sup_{[0,T]}|\Rm|(\bar g(t))),
\end{equation}
where $c$ is some finite positive constant.

\begin{proof}[Proof of Claim]
We now observe that if $\hat{\bar g}_0 := (1/t)\bar g_0$, and if, for all $k\in \N$, $\mathcal{A}_k := \{z_i: k\sqrt{t} \leq d_{\bar g_0}(x', z_i) < (k+1)\sqrt{t}\}$, then
\begin{equation}\label{eq:pointcount}
|\mathcal{A}_k| \leq C(n)\vol_{\hat{\bar g}_0}(B_{\hat{\bar g}_0}(x', (k+3/2))).
\end{equation}
where $|\mathcal{A}_k|$ denotes the number of points $z_i$ in $\mathcal{A}_k$. To see why (\ref{eq:pointcount}) is true, note that, because the balls $B(z_i, \sqrt{t}/2)$ are pairwise disjoint,  (\ref{eq:almosteucl}) implies
\begin{align*}
|A_k|\frac{1}{10}\omega_n\left(\frac{\sqrt{t}}{2}\right)^n &\leq \bigcup_{z_i \in A_k}\vol_{\bar g_0}(B_{\bar g_0}(z_i, \sqrt{t}/2)) \leq \vol_{\bar g_0}(B_{\bar g_0}(x', (k+3/2)\sqrt{t}))
\\& = \sqrt{t}^n\vol_{\hat{\bar g}_0}(B_{\hat{\bar g}_0}(x', (k+3/2))).
\end{align*}
By (\ref{eq:pointcount}) we have
\begin{align*}
\sum_{z_i} \exp\left(-\frac{d^2_{\bar g_0}(x', z_i)}{2Dt}\right) & = \sum_{z_i} \exp\left(-\frac{d^2_{\hat{\bar g}_0}(x', z_i)}{2D}\right) \leq \sum_{k=0}^{\infty}\sum_{z_i \in \mathcal{A}_k}\exp\left(-\frac{k^2}{2D}\right)
\\& \leq \sum_{k=0}^{\infty}|\mathcal{A}_k|\exp\left(-\frac{k^2}{2D}\right) 
\\& \leq \sum_{k=0}^{\infty}C(n)\vol_{\hat{\bar g}_0}(B_{\hat{\bar g}_0}(x', (k+3/2)))\exp\left(-\frac{k^2}{2D}\right)
\\& \leq \sum_{k=0}^{\infty}C(n)\vol(B_{k + 3/2}^{-T\sup_{[0,T]}|\Rm|(\bar g(t))})\exp\left(-\frac{k^2}{2D}\right)
\\& \leq \sum_{k=0}^{\infty}C\exp(-C'k^2) \leq c(n, T\sup_{[0,T]}|\Rm|(\bar g(t))),
\end{align*}
where $C=C(n, T\sup_{[0,T]}|\Rm|(\bar g(t)))$ and $C' = C'(n, T\sup_{[0,T]}|\Rm|(\bar g(t)))$.
\end{proof}

Applying (\ref{eq:expweight}) and (\ref{eq:expclaim}) to (\ref{eq:IIsummation}), we find
\begin{align*}
II &\leq c\sum_{i=1}^\infty\exp\left(-\frac{d^2(x',z_i)}{2Dt}\right)w_a(z_i,t)\left(\sqrt{t}^{-n}\int_0^t\int_{B(z_i,\sqrt{t})}|Q^0_s|(y)d_{\bar g_s}(y)ds +t^{-\tfrac{n+1}{2}}\int_0^t\int_{B(z_i,\sqrt{t})}|Q_s^1|\right)
\\& \leq c\sum_{i=1}^\infty\exp\left(-\frac{d^2(x',z_i)}{2Dt}\right)w_a(z_i,t)\left(\sqrt{t}^{-n}||Q^0||_{L^1(B(z_i,\sqrt{t})\times(0,t))} + \sqrt{t}^{-n-1}\sqrt{t}^{n/2 + 1}||Q^1||_{L^2(B(z_i,\sqrt{t})\times(0,t))}\right)
\\& \leq c||Q||_{\tilde Y_a}\sum_{i=1}^\infty\exp\left(-\frac{d^2(x',z_i)}{2Dt}\right) \leq c||Q||_{\tilde Y_a}.
\end{align*}
This yields the pointwise estimate $|w_t(x')h_t(x')| \leq c||Q||_{\tilde Y_a}$.

We now estimate the $L^2$-term. Let $\eta_{x,r^2}$ be as in the proof of Lemma \ref{lemma:homogeneous}. We multiply by $\eta^2_{x,r^2}$, integrate, and apply Young's inequality as before to find
\begin{align*}
\frac{1}{2}\int_M\eta_{x,r^2}^2|h_{r^2}|^2_{\bar g_0} &- \frac{1}{2}\int_M\eta_{x,r^2}^2|h_{0}|^2_{\bar g_0} = \int_0^{r^2}\int_M \eta_{x,r^2}^2\langle\Delta^{\bar g_s} + \Rm^{\bar g_s}(h) + Q^0 + \nabla Q^1, h\rangle_{\bar g_0}
\\& \leq -c\int_0^{r^2}\int_M \eta_{x,r^2}|\nabla h|^2_{\bar g_0} + \frac{c}{4}\int_{0}^{r^2}\int_M \eta_{x,r^2}^2|\nabla h|^2 + 16c\int_0^{r^2}\int_{M} |\nabla \eta_{x,r^2}|^2|h|^2 
\\& \qquad + c\int_0^{r^2}\int_M \eta_{x,r^2}^2|\Rm||h|^2 + \int_{0}^{r^2}\int_M \eta_{x,r^2}^2\langle Q^0, h \rangle + \int_{0}^{r^2}\int_M \eta_{x,r^2}^2\langle \nabla Q^1, h \rangle
\\& \leq -c\int_0^{r^2}\int_M \eta_{x,r^2}|\nabla h|^2_{\bar g_0} + \frac{c}{4}\int_{0}^{r^2}\int_M \eta_{x,r^2}^2|\nabla h|^2 + 16c\int_0^{r^2}\int_{M} |\nabla \eta_{x,r^2}|^2|h|^2 
\\& \qquad + c\int_0^{r^2}\int_M \eta_{x,r^2}^2|\Rm||h|^2 + \int_0^{r^2}\int_{M} \eta_{x,r^2}^2|Q^0||h| 
\\& \qquad \qquad + 16 c\int_{0}^{r^2}\int_{M} \eta_{x,r^2}^2|Q^1|^2 + \frac{c}{4}\int_{0}^{r^2}\int_{M} \eta_{x,r^2}^2|\nabla h|^2.
\end{align*}

Therefore we find
\begin{align*}
\int_0^{r^2}\int_{B(x,r)}|\nabla h|^2 \leq \frac{c}{r^2}\int_0^{r^2}\int_{B(x,2r)}|h|^2 + c\int_0^{r^2}\int_{B(x,2r)}\eta^2_{x,r^2}|Q^0||h|+ c\int_0^{r^2}\int_{B(x,2r)}|Q^1|^2.
\end{align*}
In particular,
\begin{align*}
w_{a}(x,r^2)^2r^{-n}&\int_0^{r^2}\int_{B(x,r)}|\nabla h|^2 \leq r^{-n}\frac{c}{r^2}\int_0^{r^2}\int_{B(x,2r)}|w_ah|^2 
\\& \qquad \qquad + cr^{-n}\int_0^{r^2}\int_{B(x,2r)}\eta^2_{x,r^2}|Q^0||w_ah|+ cw_{a}(x,r^2)r^{-n}\int_0^{r^2}\int_{B(x,2r)}|Q^1|^2
\\& \qquad \qquad \qquad \qquad \leq c||Q||_{\tilde Y_a(B(x,2r))}^2.
\end{align*}
by (\ref{eq:wtimeslicecomparison}), where the last inequality follows from the pointwise bound applied to $w_a(y,s)|h(y,s)|$. In particular,
\begin{equation}
w_a(x,r^2)r^{-\tfrac{n}{2}}||\nabla h||_{L^2(B(x,r)\times(0,r^2))} \leq c||Q||_{\tilde Y_a(B(x,2r))}.
\end{equation}

It remains to estimate $w_a(x,r^2)r^{\tfrac{2}{n+4}}||\nabla h||_{L^{n+4}(B(x,r)\times(\tfrac{r^2}{2}, r^2))}$. If we argue by summation as in the estimate of $II$, and appeal to (\ref{eq:ptwiseheatkernelawayfromball}) we find that
\begin{equation}
w_a(x,t)\left|\int_{M\times(0,t)\setminus \Omega(x,r^2)}\nabla \tilde K(x', t'; y,s)Q_s^0(y) + \nabla^2\tilde K(x',t'; y,s)Q_s^1(y)dyds\right| \leq c||Q||_{\tilde Y_a}
\end{equation}
for all $(x', t') \in B(x,r)\times(\tfrac{r^2}{2}, r^2),$ so without loss of generality we may assume that
\begin{equation}\label{eq:vanishingQ}
\supp(Q^0), \, \supp(Q^1) \subset B(x,r)\times(\tfrac{r^2}{2}, r^2) = \Omega(x,r^2).
\end{equation}

We have
\begin{align*}
w_a(x,r^2)&r^{\tfrac{2}{n+4}}||\nabla h||_{L^{n+4}(\Omega(x,r^2))} \leq w_a(x,r^2)r^{\tfrac{2}{n+4}}\left|\left|\int_0^t\int_M\nabla \bar K_{t-s}(z,y)Q^0(y,s)d_{\bar g_s}(y)ds\right|\right|_{L^{n+4}((z,t)\in\Omega(x,r^2))} 
\\& \qquad \qquad + w_a(x,r^2)r^{\tfrac{2}{n+4}}\left|\left|\int_0^t\int_M\nabla^2 \bar K_{t-s}(z,y) Q^1(y,s)d_{\bar g_s}(y)ds\right|\right|_{L^{n+4}((z,t)\in\Omega(x,r^2))},
\end{align*}
To estimate the first part of this sum, we use Lemma \ref{lemma:Youngsfork} and (\ref{eq:vanishingQ}). We obtain
\begin{equation*}
\begin{split}
&\left|\left|\int_0^t\int_M \nabla \bar K_{t-s}(z,y)Q^0(y,s)\right|\right|_{L^{n+4}((z,t)\in\Omega(x,r^2))}
\\& \qquad \qquad \qquad \leq \sup_{(z,t)\in \Omega(x,r^2)}||\nabla K(z,t;\cdot,\cdot)||_{L^{\tfrac{n+4}{n+3}}(M\times(0,T))}||Q^0||_{L^{\tfrac{n+4}{2}}(\Omega(x,r^2))}
\\& \qquad \qquad \qquad \leq c||Q^0||_{L^{\tfrac{n+4}{2}}(\Omega(x,r^2))}
\end{split}
\end{equation*}
by (\ref{eq:ptwiseheatkernelderivbounds}).

We may not apply this technique to the remaining term because the bounds on $|\nabla^2\bar K|$ are not strong enough. Instead, let $u(z,t) = \int_0^t\int_M\nabla \bar K_{t-s}(z,y)Q^1(y,s)d_{\bar g_s}(y)ds$, so that $\partial_t u - \Delta^{\bar g_t}u + \Rm^{\bar g_t}(u) = \nabla Q^1$. Then
\begin{equation*}
\langle u, \partial_t u \rangle_{\bar g_0} = \langle u, \Delta^{\bar g_t}u - \Rm^{\bar g_t}(u) + \nabla Q^1\rangle_{\bar g_0}
\end{equation*}
and we have, by Young's inequality,
\begin{align*}
\frac{1}{2}\int_{B(x,r)}|u(r^2)|_{\bar g_0}^2 &- \frac{1}{2}\int_{B(x,r)}|u(\tfrac{r^2}{2})|_{\bar g_0}^2  = \int_{r^2/2}^{r^2}\int_{B(x,r)}\langle u, \partial_t u\rangle_{\bar g_0} 
\\& = \int_{r^2/2}^{r^2}\int_{B(x,r)}\langle u, \Delta^{\bar g_t}u - \Rm^{\bar g_t}(u) + \nabla Q^1\rangle_{\bar g_0}
\\& = \int_{r^2/2}^{r^2}\int_{B(x,r)}|u|^2_{\bar g_0} - c\int_{r^2/2}^{r^2}\int_{B(x,r)}|\nabla u|^2_{\bar g_0} + \sup_{[0,T]}|\Rm|(\bar g(t))\int_{r^2/2}^{r^2}\int_{B(x,r)}|u|^2_{\bar g_0} 
\\& \qquad \qquad \qquad + c\int_{r^2/2}^{r^2}\int_{B(x,r)}\langle Q^1, \nabla u\rangle_{\bar g_0}
\\& \leq 2c\int_{r^2/2}^{r^2}\int_{B(x,r)}|Q^1|^2_{\bar g_0} + \left(1 + \sup_{[0,T]}|\Rm|(\bar g(t))\right)\int_{r^2/2}^{r^2}\int_{B(x,r)}|u|^2_{\bar g_0} 
\\& \qquad \qquad \qquad -c(1 + \tfrac{1}{2})\int_{r^2/2}^{r^2}\int_{B(x,r)}|\nabla u|^2_{\bar g_0}
\end{align*}
where we have used (\ref{eq:vanishingQ}).
Observing that, by (\ref{eq:vanishingQ}), $u(r^2/2)=0$, we find
\begin{align*}
\int_{r^2/2}^{r^2}\int_{B(x,r)}|\nabla u|^2 &\leq c\int_{r^2/2}^{r^2}\int_{B(x,r)}|Q^1|_{\bar g_0}^2 + c\left(1 + \sup_{[0,T]}|\Rm(\bar g(t))|\right)\int_{r^2/2}^{r^2}\int_{B(x,r)}|u|^2_{\bar g_0}.
\end{align*}

Recalling the definition of $u$, we apply Lemma \ref{lemma:Youngsfork}, (\ref{eq:vanishingQ}), and Corollary \ref{cor:exponentialboundnonscalar} to find:
\begin{align*}
&\int_{r^2/2}^{r^2}\int_{B(x,r)} |\nabla u|^2 \leq c||Q^1||^2_{L^2(B(x,r)\times(\tfrac{r^2}{2}, r^2))} 
\\& \qquad \qquad + c\left(1 + \sup_{[0,T]}|\Rm|(\bar g(t))\right)\left|\left| \int_0^t\int_M \nabla \bar K_{t-s}(z,y)Q^1(y,s)d_{\bar g_s}(y)ds\right|\right|^2_{L^2((z,t)\in B(x,r)\times(\tfrac{r^2}{2}, r^2))}
\\& \leq c||Q^1||^2_{L^2(B(x,r)\times(\tfrac{r^2}{2}, r^2))} 
\\& \qquad \qquad + c\left(1 + \sup_{[0,T]}|\Rm|(\bar g(t))\right)\sup_{(z,t)\in B(x,r)\times(\tfrac{r^2}{2}, r^2)}||\nabla K(z,t; \cdot, \cdot)||^2_{L^1(M\times (0,t))}||Q^1||^2_{L^2(B(x,r)\times(\tfrac{r^2}{2}, r^2))}
\\& \leq c\left(1 + \sup_{[0,T]}|\Rm|(\bar g(t))\right)\sup_{(z,t)\in B(x,r)\times(\tfrac{r^2}{2}, r^2)}
\left|\int_0^t\int_M(t-s)^{-\tfrac{(n+1)}{2}}\exp\left(-\frac{d_{\bar g_0}(z,y)}{d(t-s)}\right)\right|^2 ||Q^1||^2_{L^2(\Omega(x,r^2))}
\\& \leq c\left(1 + \sup_{[0,T]}|\Rm|(\bar g(t))\right)\sup_{(z,t)\in B(x,r)\times(\tfrac{r^2}{2}, r^2)}
\left|\int_0^t(t-s)^{-\tfrac{1}{2}}\right|^2 ||Q^1||^2_{L^2(\Omega(x,r^2)}
\\& \leq c\left(1 + \sup_{[0,T]}|\Rm|(\bar g(t))\right)r^2||Q^1||^2_{L^2(B(x,r)\times(\tfrac{r^2}{2}, r^2))}
\leq cT\sup_{[0,T]}|\Rm|(\bar g(t))||Q^1||^2_{L^2(B(x,r)\times(\tfrac{r^2}{2}, r^2))}
\end{align*}
where $c= c(n, T\sup_{[0,T]}|\Rm|(\bar g(t)))$, and is adjusted as necessary.

Since $\nabla u(z,t) = \int_0^t\int_M\nabla^2 K_{t-s}(z,y)Q^1(y,s)d_{\bar g_s}(y)ds$, we have shown
\begin{equation}
\left|\left|\int_0^t\int_M\nabla^2 \bar K_{t-s}(z,y) Q^1(y,s)d_{\bar g_s}(y)ds\right|\right|_{L^2((z,t)\in \Omega(x,r^2))} \leq c||Q^1||_{L^2(\Omega(x,r^2))} 
\end{equation}
so the integral operator given by $\nabla^2\bar K$ is a bounded operator on $L^2(\Omega(x,r^2))$. By the Calder\'on-Zygmund Theorem and the Marcinkiewicz Interpolation Theorem we may extend the result to $1 < p < \infty$. In particular, 
\begin{equation}\label{eq:hardboundpostCZ}
\begin{split}
w_{a}(x,r^2)r^{\tfrac{2}{n+4}}&\left|\left|\int_0^t\int_M\nabla^2 \bar K_{t-s}(z,y) Q^1(y,s)d_{\bar g_s}(y)ds\right|\right|_{L^{n+4}((z,t)\in \Omega(x,r^2))} 
\\& \leq cw_{a}(x,r^2)r^{\tfrac{2}{n+4}}||Q^1||_{L^{n+4}(\Omega(x,r^2))} \leq c||Q||_{\tilde Y_a(B(x,r))}.
\end{split}
\end{equation}

\end{proof}

We are now ready to prove the main theorem of this section.

\begin{proof}[Proof of Theorem \ref{thm:fixedpointexistence}] 
Let $\varepsilon$ be as in Lemma \ref{lemma:RDTexistence}. We select smooth background metrics $\bar g'$ and $\bar g''$ as follows: First choose a smooth metric $\bar g$ on $M'$ such that $||g_0' - \bar g||_{C^0(B_{\bar g}(x_0', 4R), \bar g)}, ||\phi^*g_0'' - \bar g||_{C^0(B_{\bar g}(x_0', 4R))} < \varepsilon/2$, where the notation $||\cdot||_{C^0(B_{\bar g}(x_0', 4R), \bar g)}$ means that both the norm and the domain are measured using the metric $\bar g$, for sufficiently small $R>0$. This is possible because $g_0'(x_0') = \phi^*g_0''(x_0')$. Note that by choosing smaller $R$ and $T$, we can reduce $\varepsilon$ as needed.

Extend $\bar g$ to a smooth metric $\bar g'_0$ on $M'$ that is $\varepsilon$-close to $g_0'$ everywhere in $M'$, and extend $\phi_*\bar g$ to a smooth metric $\bar g''_0$ on $M''$ that is $\varepsilon$-close to $g_0''$ everywhere in $M''$. Solve the Ricci flow equation starting from $\bar g'_0$ on $M'$ and $\bar g''_0$ on $M''$ to find smooth Ricci flows $\bar g'(t)$ on $M'$ and $\bar g''(t)$ on $M''$, both defined on some time interval $[0,T]$.

By Lemma \ref{lemma:RDTexistence} there exist solutions $h_t' = g_t' - \bar g_t'$ and $h_t'' = g_t'' - \bar g_t''$ to (\ref{eq:representationformula}), defined for $t\in (0,T]$, starting from $g_0'$ and $g_0''$ in $X^{\bar g'}$ and $X^{\bar g''}$ respectively, and satisfying $||h_t'||_{X^{\bar g'}}, ||h_t''||_{X^{\bar g''}} \leq C\varepsilon < 1$, where $C$ is as in Lemma \ref{lemma:RDTexistence}. For the sake of simplicity, we assume that $C\varepsilon < 1/2$ in what follows; this can be achieved by reducing $\varepsilon$ and $T$ if necessary.

Now fix $x\in M, 0<r\leq 4R \leq \sqrt{T}$. Let $\varphi$ be a (time-independent) smooth cutoff function equal to $1$ on $B(x_0', R)$ that vanishes outside $B(x_0', 4R)$, such that $|\nabla^{\bar g'}\varphi|, |\Delta^{\bar g'}\varphi| \leq b$ for all $t\in [0,T]$, where $b$ is some finite, positive number (we may take $b$ to be bounded by some function of $R$). Now consider the evolution of $\varphi(h_t'-\phi^*h_t'')$:
\begin{align*}
(\partial_t + L^{\bar g'})\varphi(h_t' - \phi^*h_t'') &= \varphi\partial_t(h_t' - \phi^*h_t'') - \Delta^{\bar g'}(\varphi(h_t' - \phi^*h_t'')) + \varphi\Rm^{\bar g'}(h_t' - \phi^*h_t'')
\\&= \varphi\left[\partial_t(h_t' - \phi^*h_t'') - \Delta^{\bar g'}(h_t' - \phi^*h_t'') + \Rm^{\bar g'}(h_t' - \phi^*h_t'')\right]
\\& \qquad - (\Delta^{\bar g'}\varphi)(h_t' - \phi^*h_t'') - 2\nabla^{\bar g'}_{\nabla\varphi} (h_t' - \phi^*h_t'')
\\& = \varphi\big[(\partial_t h_t' - \Delta^{\bar g'}h_t' + \Rm^{\bar g'}(h_t')) - (\partial_t \phi^*h_t'' - \Delta^{\phi^*\bar g''}\phi^*h_t'' + \Rm^{\phi^*\bar g''}(\phi^*h_t''))\big] 
\\& - \left(\Delta^{\phi^*\bar g''}\phi^*h_t'' - \Delta^{\bar g'}\phi^*h_t''\right) + \Rm^{\phi^*\bar g''}(\phi^*h_t'') - \Rm^{\bar g'}(\phi^*h_t'')\big] 
\\& -(\Delta^{\bar g'}\varphi)(h_t' - \phi^*h_t'')- 2\nabla^{\bar g'}_{\nabla\varphi}(h_t' - \phi^*h_t'')
\\& = \varphi\big[Q^{\bar g'}[h_t'] - Q^{\phi^*\bar g''}[\phi^*h_t''] - \left(\Delta^{\phi^*\bar g''}\phi^*h_t'' - \Delta^{\bar g'}\phi^*h_t''\right) 
\\& \qquad+ \Rm^{\phi^*\bar g''}(\phi^*h_t'') - \Rm^{\bar g'}(\phi^*h_t'')\big] -(\Delta^{\bar g'}\varphi)(h_t' - \phi^*h_t'')
\\& \qquad \qquad -2\nabla^{\bar g'}_{\nabla\varphi} (h_t' - \phi^*h_t''),
\end{align*}
where defined. Therefore, if $\bar K$ is the heat kernel corresponding to $\bar g'$, observing that the integrand vanishes outside of a set on which $\phi^*g_t''$ and $\phi^*\bar g_t''$ are defined, we find
\begin{align*}
\varphi(h_t' - \phi^*h_t'')(x) &= \int_0^t\int_M \bar K_{t-s}(x,y)\varphi(Q^{\bar g'}[h'] - Q^{\phi^*\bar g''}[\phi^*h''])(y,s)dyds
\\& -\int_0^t\int_M\bar K_{t-s}(x,y)\varphi(\Delta^{\phi^*\bar g''}\phi^*h'' - \Delta^{\bar g'}\phi^*h'')(y,s)dyds 
\\& + \int_0^t\int_M\bar K_{t-s}(x,y)\varphi(\Rm^{\phi^*\bar g''}(\phi^*h'')- \Rm^{\bar g'}(\phi^*h''))(y,s)dyds 
\\& - \int_0^t\int_M\bar K_{t-s}(x,y)(\Delta^{\bar g'}\varphi)(h' - \phi^*h'')(y,s)dyds
\\& + 2\int_0^t\int_M\bar K_{t-s}(x,y)\nabla_{\nabla \varphi}^{\bar g'}(h' - \phi^*h'')\rangle(y,s)dyds 
\\& + \int_M \bar K_{t}(x,y)\varphi(h'_0-\phi^*h''_0)(y)dy
\end{align*}

Inheriting the constants from Lemmata \ref{lemma:homogeneous} and \ref{lemma:hard}, we have
\begin{align*}
||\varphi(h'-\phi^*h'')||&_{\tilde X_a(B(x,r))} \leq c||\varphi(Q_0^{\bar g'}[h'] - Q_0^{\phi^*\bar g''}[\phi^*h''])||_{\tilde Y^0_a} + c||\varphi(Q_1^{\bar g'}[h'] - Q_1^{\phi^*\bar g''}[\phi^*h''])||_{\tilde Y^1_a} 
\\& + \bigg|\bigg| \int_0^t\int_M \bar K_{t-s}(x,y)\varphi(\Delta^{\bar g'}\phi^*h'' - \Delta^{\phi^*\bar g''}\phi^*h'')(y,s)dyds\bigg|\bigg|_{\tilde X_a(B(x,r))}  
\\& + c||\varphi(\Rm^{\bar g'}(\phi^*h'') - \Rm^{\phi^*\bar g''}(\phi^*h''))||_{\tilde Y^0_a} 
\\& + \left|\left|\int_0^t\int_M \bar K_{t-s}(x,y)(\Delta^{\bar g'}\varphi)(y)(h_t' - \phi^*h_t'')(y,s)d_{\bar g'_s}(y)ds\right|\right|_{\tilde X_a} 
\\& + \left|\left|2\int_0^t\int_M \bar K_{t-s}(x,y)\nabla^{\bar g'}_{\nabla\varphi}(h_t' - \phi^*h_t'')(y,s) d_{\bar g'_s}(y)ds\right|\right|_{\tilde X_a}  
\\& + c||\varphi w_a(0)(h_0' - \phi^*h_0'')||_{L^\infty(M)}
\\& =: I + II+ III+ IV + V + VI +c||\varphi w_a(0)(h_0' - \phi^*h_0'')||_{L^\infty(M)},
\end{align*}
where here all norms are computed with respect to $\bar g'$.
Note first that, after reducing $T$ sufficiently depending on $||\nabla^{\ell + m}(\bar g' - \phi^*\bar g'')||_{C^0(B_{\bar g_0'}(x_0', 4R), \bar g_0')}$, we may assume that, for a tensor field $h$, we have
\begin{equation}\label{eq:eatsweight}
|(\nabla^{\bar g'})^\ell h - (\nabla^{\phi^*\bar g''})^\ell h| \leq 20Rt^m |h|
\end{equation} 
for all $\ell \leq 3$, say, where $m\geq 2+ \eta$, as follows: Reduce $T$ (and hence $R$) so that $|(\nabla^{\bar g'})^{\ell + m} h(x_0',t) - (\nabla^{\phi^*\bar g''})^{\ell + m} h(x_0',t)| < 20$ for all $0\leq t \leq T$; this is possible because $\bar g'(0) = \phi^*\bar g''(0)$ in $B(x_0',4R)$. Then, by integrating in time, we find $|(\nabla^{\bar g'})^{\ell} h(x_0',t) - (\nabla^{\phi^*\bar g''})^{\ell} h(x_0',t)| < 20t^m$. To obtain the bound elsewhere in the ball, integrate along a geodesic. In particular, by requiring that $m \geq 2+\eta$, we have $w_a(t)|\nabla^{\bar g'}h-\nabla^{\bar g''}h| \leq c(n)|h|$. Moreover, (\ref{eq:eatsweight}) together with the fact that $\bar g_0' = \phi^*\bar g_0''$ on $B(x_0', 4R)$ and the definition of the $X$-norm implies that
\begin{equation}\label{eq:eatsweightcomp}
||\phi^*h''||_{X^{\bar g'}(B_{\bar g_0'}(x, r))} \leq c(n)||h''||_{X^{\bar g''}(B_{\bar g_0''}(\phi(x), r))},
\end{equation}
for $B_{\bar g_0'}(x,r)\subset B_{\bar g_0'}(x_0', 4R)$, so we often replace instances of $||\phi^*h''||_{X^{\bar g'}(\phi^{-1}(B(x, r)))}$ with $||h''||_X^{\bar g''}$ in what follows.

\emph{Term $I$:}
First, note that by (\ref{eq:Q0is}) and (\ref{eq:inverseexpansion1}) we find
\begin{align*}
&|(\bar g' + h')^{-1}\star(\bar g' + h')^{-1}\star\nabla^{\bar g'}h'\star\nabla^{\bar g'}h' - (\phi^*\bar g'' + \phi^*h'')^{-1}\star(\phi^*\bar g'' + \phi^*h'')^{-1}\star\nabla^{\phi^*\bar g''}\phi^*h''\star\nabla^{\phi^*\bar g''}\phi^*h''|
\\& \leq |(\bar g' + h')^{-1}\star(\bar g' + h')^{-1}\star\nabla^{\bar g'}h'\star\nabla^{\bar g'}h' - (\bar g' + h')^{-1}\star(\bar g' + h')^{-1}\star\nabla^{\phi^*\bar g''}\phi^*h''\star\nabla^{\phi^*\bar g''}\phi^*h'' |
\\& \qquad + |(\bar g' + h')^{-1}\star(\bar g' + h')^{-1}\star\nabla^{\phi^*\bar g''}\phi^*h''\star\nabla^{\phi^*\bar g''}\phi^*h'' 
\\& \qquad\qquad\qquad - (\phi^*\bar g'' + \phi^*h'')^{-1}\star(\phi^*\bar g'' + \phi^*h'')^{-1}\star\nabla^{\phi^*\bar g''}\phi^*h''\star\nabla^{\phi^*\bar g''}\phi^*h''|
\\& \leq c(n)|\nabla^{\bar g'}h'\star\nabla^{\bar g'}h' - \nabla^{\phi^*\bar g''}\phi^*h'' \star \nabla^{\phi^*\bar g''}\phi^*h''| + c(n,\gamma)|\nabla^{\phi^*\bar g''}\phi^*h''\star\nabla^{\phi^*\bar g''}\phi^*h''|
\\& \leq c(n)|\nabla^{\bar g'}h'\star\nabla^{\bar g'}h' - \nabla^{\bar g'}h'\star\nabla^{\phi^*\bar g''}\phi^*h'' | 
\\& \qquad + c(n)|\nabla^{\bar g'}h'\star\nabla^{\phi^*\bar g''}\phi^*h'' - \nabla^{\phi^*\bar g''}\phi^*h''\star\nabla^{\phi^*\bar g''}\phi^*h''| 
\\& \qquad \qquad + c(n)|\nabla^{\phi^*\bar g''}\phi^*h''\star\nabla^{\phi^*\bar g''}\phi^*h''|
\\& \leq c(n)|\nabla^{\bar g'}h'||\nabla^{\bar g'}h' - \nabla^{\phi^*\bar g''}\phi^*h''| + c(n, \gamma)|\nabla^{\phi^*\bar g''}\phi^*h''||\nabla^{\bar g'}h' - \nabla^{\phi^*\bar g''}\phi^*h''| 
\\& \qquad + c(n)|\nabla^{\phi^*\bar g''}\phi^*h''\star\nabla^{\phi^*\bar g''}\phi^*h''| 
\\& \leq c(n)|\nabla^{\bar g'}h'||\nabla^{\bar g'}h' - \nabla^{\bar g'}\phi^*h''| + c(\gamma)|\nabla^{\bar g'}h'||\nabla^{\bar g'}\phi^*h'' - \nabla^{\phi^*\bar g''}\phi^*h''|
\\& \qquad + c(n)|\nabla^{\phi^*\bar g''}\phi^*h''||\nabla^{\bar g'}h' - \nabla^{\bar g'}\phi^*h''| + c(\gamma)|\nabla^{\phi^*\bar g''}\phi^*h''||\nabla^{\bar g'}\phi^*h'' - \nabla^{\phi^*\bar g''}\phi^*h''| 
\\& \qquad \qquad + c(n)|\nabla^{\phi^*\bar g''}\phi^*h''\star\nabla^{\phi^*\bar g''}\phi^*h''|.
\end{align*}

Similarly, by (\ref{eq:Q0is}) and (\ref{eq:inverseexpansion2}) we have
\begin{align*}
& |[(\bar g' + h')^{-1} - (\bar g')^{-1}]\star\Rm^{\bar g'}\star h' - [(\phi^*\bar g'' + \phi^*h'')^{-1} - (\phi^*\bar g'')^{-1}]\star\Rm^{\phi^*\bar g''}\star \phi^*h'' |
\\& \leq |[(\bar g' + h')^{-1} - (\bar g')^{-1}]\star\Rm^{\bar g'}\star h' - [(\bar g' + h')^{-1} - (\bar g')^{-1}]\star\Rm^{\phi^*\bar g''}\star \phi^*h''| 
\\& \qquad \qquad + |[(\bar g' + h')^{-1} - (\bar g')^{-1}]\star\Rm^{\phi^*\bar g''}\star \phi^*h'' - [(\phi^*\bar g'' + \phi^*h'')^{-1} - (\phi^*\bar g'')^{-1}]\star\Rm^{\phi^*\bar g''}\star \phi^*h'' |
\\& \leq  c(n)|h'||\Rm^{\bar g'}\star h' - \Rm^{\phi^*\bar g''}\star \phi^*h''|  + c(n)|\Rm^{\phi^*\bar g''}||\phi^*h''|
\\& \leq c(n)|h'||\Rm^{\bar g'}\star h' - \Rm^{\bar g'}\star \phi^*h''| + c(n)|h'||\Rm^{\bar g'}\star \phi^*h'' - \Rm^{\phi^*\bar g''}\star \phi^*h''| + c(n)|\Rm^{\phi^*\bar g''}||\phi^*h''|
\\& \leq c(n)|h'||\Rm^{\bar g'}||h' - \phi^*h''| + c(n)|h'||\phi^*h''||\Rm^{\bar g'} - \Rm^{\phi^*\bar g''}| + c(n)|\Rm^{\phi^*\bar g''}||\phi^*h''|
\end{align*}

Thus we have
\begin{align*}
|\varphi(Q_0^{\bar g'}[h'] & - Q_0^{\phi^*\bar g''}[\phi^*h''])| \leq c(n)|\nabla^{\bar g'}h'||\varphi(\nabla^{\bar g'}h' - \nabla^{\bar g'}\phi^*h'')| + c(n)|\nabla^{\bar g'}h'||\varphi(\nabla^{\bar g'}\phi^*h'' - \nabla^{\phi^*\bar g''}\phi^*h'')|
\\& + c(n)|\nabla^{\phi^*\bar g''}\phi^*h''||\varphi(\nabla^{\bar g'}h' - \nabla^{\bar g'}\phi^*h'')| + c(n)|\nabla^{\phi^*\bar g''}\phi^*h''||\varphi(\nabla^{\bar g'}\phi^*h'' - \nabla^{\phi^*\bar g''}\phi^*h'')| 
\\& + c(n)|\varphi(\nabla^{\phi^*\bar g''}\phi^*h''\star\nabla^{\phi^*\bar g''}\phi^*h'')| + c(n)|h'||\Rm^{\bar g'}||\varphi(h' - \phi^*h'')| 
\\& + c(n)|h'||\phi^*h''||\varphi(\Rm^{\bar g'} - \Rm^{\phi^*\bar g''})| + c(n)|\Rm^{\phi^*\bar g''}||\varphi\phi^*h''|
\end{align*}

Arguing as in the proof of Lemma \ref{lemma:easy} and applying (\ref{eq:eatsweight}), we find
\begin{equation}
\begin{split}
I &\leq c(n)(||h'||_{X^{\bar g'}}+ ||h''||_{X^{\bar g''}})||\varphi(h'-\phi^*h'')||_{\tilde X_a^{\bar g'}} 
\\& \qquad + c(n)||h'||_{X^{\bar g'}}||h''||_{X^{\bar g''}} + c(n, \gamma)||h''||_{X^{\bar g''}} + c(n)||h''||^2_{X^{\bar g''}} 
\\& = c(n)(||h'||_{X^{\bar g'}} + ||h''||_{X^{\bar g''}})||\varphi(h'-\phi^*h'')||_{\tilde X_a^{\bar g'}} + C(n),
\end{split}
\end{equation}
where $C(n)$ is a finite positive number.

\emph{Term $II$:}
Applying (\ref{eq:Q1is}), (\ref{eq:inverseexpansion1}), (\ref{eq:inverseexpansion2}), and (\ref{eq:inverseexpansion3}) we have
\begin{align*}
|Q_1^{\bar g'}[h'] - &Q_1^{\phi^*\bar g''}[\phi^*h'']| \leq c(n)|[(\bar g' + h')^{-1} - (\bar g')^{-1}]\star\nabla^{\bar g'}h' - [(\phi^*\bar g'' + \phi^*h'')^{-1}- (\phi^*\bar g'')^{-1}]\star\nabla^{\phi^*\bar g''}\phi^*h''|
\\& \leq c(n)||[(\bar g' + h')^{-1} - (\bar g')^{-1}]\star\nabla^{\bar g'}h' - [(\bar g' + h')^{-1} - (\bar g')^{-1}]\star\nabla^{\bar g'}\phi^*h''| 
\\& + c(n)|[(\bar g' + h')^{-1} - (\bar g')^{-1}]\star\nabla^{\bar g'}\phi^*h'' - [(\bar g' + h')^{-1} - (\bar g')^{-1}]\star\nabla^{\phi^*\bar g''}\phi^*h''|
\\& + c(n)|[(\bar g' + h')^{-1} - (\bar g')^{-1}]\star\nabla^{\phi^*\bar g''}\phi^*h'' - [(\phi^*\bar g'' + \phi^*h'')^{-1} - (\phi^*\bar g'')^{-1}]\star\nabla^{\phi^*\bar g''}\phi^*h''|
\\& \leq c(n)|h'||\nabla^{\bar g'}h' - \nabla^{\bar g'}\phi^*h''| + c(n)|h'||\nabla^{\bar g'}\phi^*h'' - \nabla^{\phi^*\bar g''}\phi^*h''|
\\& + c(n)|(\bar g')^{-1} - (\phi^*\bar g'')^{-1}||\nabla^{\phi^*\bar g''}\phi^*h''| + c(n)|(\bar g' + h')^{-1} - (\bar g' + \phi^*h'')^{-1}||\nabla^{\phi^*\bar g''}\phi^*h''| 
\\& + c(n)|(\bar g' + \phi^*h'')^{-1} - (\bar g' + (-\bar g' + \phi^*\bar g'' + \phi^*h''))^{-1}||\nabla^{\phi^*\bar g''}\phi^*h''|
\\& \leq c(n)|h'||\nabla^{\bar g'}h' - \nabla^{\bar g'}\phi^*h''| + c(n)|h'||\nabla^{\bar g'}\phi^*h'' - \nabla^{\phi^*\bar g''}\phi^*h''|
\\& + c(n)|\bar g' - \phi^*\bar g''||\nabla^{\phi^*\bar g''}\phi^*h''| + |h' - \phi^*h''||\nabla^{\phi^*\bar g''}\phi^*h''|.
\end{align*}
Then, again arguing as in the analysis of Term $I$ and the proof of Lemma \ref{lemma:easy}, we find
\begin{equation}
\begin{split}
II & \leq c(n)(||h'||_{X^{\bar g'}} + ||h''||_{X^{\bar g''}})||\varphi(h'-\phi^*h'')||_{\tilde X_a^{\bar g'}} + c(n)||h''||_{X}
\\& \leq c(n)(||h'||_{X^{\bar g'}} + ||h''||_{X^{\bar g''}})||\varphi(h'-\phi^*h'')||_{\tilde X_a^{\bar g'}} + C(n),
\end{split}
\end{equation}
where $C(n)$ is some finite positive number.

\emph{Term $III$:}
Working in coordinates, we find
\begin{align*}
&\bigg|\int_0^t\int_M K_{t-s}(x,y)\varphi(\Delta^{\bar g'}\phi^*h'' - \Delta^{\phi^*\bar g''}\phi^*h'')dyds\bigg| \\& = \bigg|\int_0^t\int_M K_{t-s}(x,y)\varphi((\bar g')^{ij}\nabla^{\bar g'}_i\nabla^{\bar g'}_j\big((\phi^*h'')_{ab}dx^a\otimes dx^b\big)  -(\phi^*\bar g'')^{ij}\nabla^{\phi^*\bar g''}_i\nabla^{\phi^*\bar g''}_j\big((\phi^*h'')_{ab}dx^adx^b\big)dyds\bigg|
\\& \leq \bigg|\int_0^t\int_M K_{t-s}(x,y)\varphi((\bar g')^{ij}\nabla^{\bar g'}_i\nabla^{\bar g'}_j\big((\phi^*h'')_{ab}dx^a\otimes dx^b\big)  -(\phi^*\bar g'')^{ij}\nabla^{\bar g'}_i\nabla^{\bar g'}_j\big((\phi^*h'')_{ab}dx^adx^b\big)dyds\bigg|
\\& + \bigg|\int_0^t\int_M K_{t-s}(x,y)\varphi((\phi^*\bar g'')^{ij}\nabla^{\bar g'}_i\nabla^{\bar g'}_j\big((\phi^*h'')_{ab}dx^a\otimes dx^b\big)  -(\phi^*\bar g'')^{ij}\nabla^{\phi^*\bar g''}_i\nabla^{\phi^*\bar g''}_j\big((\phi^*h'')_{ab}dx^adx^b\big)dyds\bigg| 
\\& \leq \bigg|\int_0^t\int_M \big(\nabla^{\bar g'}_iK_{t-s}(x,y)\big)\varphi((\bar g')^{ij}\nabla^{\bar g'}_j\big((\phi^*h'')_{ab}dx^a\otimes dx^b\big)  -(\phi^*\bar g'')^{ij}\nabla^{\bar g'}_j\big(\phi^*h''_{ab}dx^adx^b\big)dyds\bigg|
\\& + \bigg|\int_0^t\int_M K_{t-s}(x,y)(\nabla^{\bar g'}_i\varphi)((\bar g')^{ij}\nabla^{\bar g'}_j\big((\phi^*h'')_{ab}dx^a\otimes dx^b\big)  -(\phi^*\bar g'')^{ij}\nabla^{\bar g'}_j\big((\phi^*h'')_{ab}dx^adx^b\big)dyds\bigg|
\\& + \bigg|\int_0^t\int_M K_{t-s}(x,y)\varphi[(\nabla^{\bar g'}_i((\bar g')^{ij} - (\phi^*\bar g'')^{ij}))\nabla^{\bar g'}_j\big((\phi^*h'')_{ab}dx^a\otimes dx^b\big)]dyds\bigg|
\\& + \bigg|\int_0^t\int_M K_{t-s}(x,y)\varphi[(\phi^*\bar g'')^{ij}\partial_i(\phi^*h'')_{ab}(\nabla^{\bar g'}_j - \nabla^{\phi^*\bar g''}_j)dx^a\otimes dx^b] dyds\bigg|
\\& +\bigg|\int_0^t\int_M K_{t-s}(x,y)\varphi[(\phi^*\bar g'')^{ij}\partial_j(\phi^*h'')_{ab}(\nabla^{\bar g'}_i - \nabla^{\phi^*\bar g''}_i)dx^a\otimes dx^b] dyds\bigg|
\\& + \bigg| \int_0^t\int_M K_{t-s}(x,y)\varphi[(\phi^*h'')_{ab}(\nabla^{\bar g'}_i\nabla^{\bar g'}_j - \nabla^{\phi^*\bar g''}_i\nabla^{\phi^*\bar g''}_j)dx^a\otimes dx^b]dyds \bigg|
\end{align*}

Now we observe the following consequences of H\"older's inequality and (\ref{eq:almosteucl}):
\begin{equation}\label{eq:L1L2}
r^{-n}||f||_{L^1(B(x,r)\times(0,r^2))} \leq r^{-n}||1||_{L^2(B(x,r)\times(0,r^2))}||f||_{L^2(B(x,r)\times(0,r^2))} \leq cr^{-n/2 + 1/2}||f||_{L^2(B(x,r)\times (0,r^2))}
\end{equation}
\begin{equation}\label{eq:Ln+4Ln+4/2}
\begin{split}
r^{\tfrac{4}{n+4}}||f||_{L^{\tfrac{n+4}{2}}(B(x,r)\times(\tfrac{r^2}{2}, r^2))} &= r^{\tfrac{4}{n+4}} || |f|^{\tfrac{n+4}{2}}||_{L^{1}(B(x,r)\times (\tfrac{r^2}{2}, r^2))}^{\tfrac{2}{n+4}}
\\& \leq r^{\tfrac{4}{n+4}}\left(|| |f|^{\tfrac{n+4}{2}}||_{L^2(B(x,r)\times (\tfrac{r^2}{2}, r^2))}||1||_{L^2(B(x,r)\times (\tfrac{r^2}{2}, r^2))}\right)^{\tfrac{2}{n+4}}
\\& \leq cr^{\tfrac{2}{n+4}}||f||_{L^{n+4}(B(x,r)\times (\tfrac{r^2}{2}, r^2))}
\end{split}
\end{equation}

Then, using Lemma \ref{lemma:hard} and (\ref{eq:eatsweight}), we conclude:
\begin{align*}
\bigg|\bigg|& \int_0^t\int_M K_{t-s}(x,y)\varphi(\Delta^{\bar g'}\phi^*h'' - \Delta^{\bar g''}\phi^*h'') \bigg|\bigg|_{\tilde X_a(B(x,r))}
\\& \leq ||\varphi(((\bar g')^{ij} - (\phi^*\bar g'')^{ij})\nabla^{\bar g'}_j\phi^*h'')||_{\tilde Y_a^1} + ||(\nabla^{\bar g'}_i\varphi)(((\bar g')^{ij} - (\phi^*\bar g'')^{ij})\nabla^{\bar g'}_j\phi^*h'')||_{\tilde Y_a^0}
\\& + ||\varphi(\nabla^{\bar g'}_i((\bar g')^{ij} - (\phi^*\bar g'')^{ij})\nabla^{\bar g'}_j\phi^*h'')||_{\tilde Y_a^0}
+ ||\varphi((\phi^*\bar g'')^{ij}\partial_i(\phi^*h'')_{ab}(\nabla^{\bar g'}_j - \nabla^{\phi^*\bar g''}_j)dx^a\otimes dx^b)||_{\tilde Y_a^0}
\\& + ||\varphi((\phi^*\bar g'')^{ij}\partial_j(\phi^*h'')_{ab}(\nabla^{\bar g'}_i - \nabla^{\phi^*\bar g''}_i)dx^a\otimes dx^b)||_{\tilde Y_a^0} 
\\& + ||\varphi((\phi^*h'')_{ab}(\nabla^{\bar g'}_i\nabla^{\bar g'}_j - \nabla^{\phi^*\bar g''}_i\nabla^{\phi^*\bar g''}_j)dx^a\otimes dx^b)||_{\tilde Y_a^0} + \text{lower order terms}
\\& \leq c(n, b)||\varphi\phi^*h''||_{X^{\bar g''}},
\end{align*}
where the estimate on the first term follows from the definition of the norms as in the proof of Lemma \ref{lemma:easy}, the estimate on the last term follows from bounding the $\tilde Y^0_a$-norm by the $L^\infty$ norm as in the proof of Lemma \ref{lemma:easy}, the bound on the lower order terms follows from Lemmata \ref{lemma:easy} and \ref{lemma:hard}, and the bounds on the remaining terms are a consequence of taking $f = \nabla^{\bar g'}\phi^*h''$ in (\ref{eq:L1L2}) and (\ref{eq:Ln+4Ln+4/2}).

\emph{Term $IV$:}
We have 
\begin{equation*}
|\varphi(\Rm^{\bar g'}(\phi^*h'') - \Rm^{\phi^*\bar g''}(\phi^*h''))| \leq |\Rm^{\bar g'} - \Rm^{\phi^*\bar g''}||\varphi(\phi^*h'')| \leq c(n)Rt^m|\varphi(\phi^*h'')|
\end{equation*}
so, comparing the $Y^0$-norm to the $L^\infty$ norm as in the proof of Lemma \ref{lemma:easy}, we find
\begin{equation}
IV \leq c(n)||h''||_{X^{\bar g''}}.
\end{equation}

\emph{Term $V$:}
By Lemma \ref{lemma:hard} we have that 
\begin{equation}\label{eq:Vvanishingbound}
\begin{split}
V &\leq c(n)||(\Delta^{\bar g'}\varphi)(h' - \phi^*h'')||_{\tilde Y^0_a} \leq c(n)\sup_{x\in M, 0<r^2 < T}w_a(x,r^2)\bigg(r^{-n}||(\Delta^{\bar g'}\varphi)(h' - \phi^*h'')||_{L^1(B_{\bar g_0'}(x,r)\times (0,r^2))} 
\\& \qquad \qquad + r^{\frac{4}{n+4}}||(\Delta^{\bar g'}\varphi)(h' - \phi^*h'')||_{L^{\frac{n+4}{2}}(B_{\bar g_0'}(x,r)\times (\tfrac{r^2}{2},r^2))}\bigg)
\\&= c(n)\sup_{x\in M, 0<r^2 < T}w_a(x,r^2)\bigg(r^{-n}||(\Delta^{\bar g'}\varphi)(h' - \phi^*h'')||_{L^1([B_{\bar g_0'}(x,r)\cap A_{\bar g_0'}(x_0';R, 4R)]\times (0,r^2))} 
\\& \qquad \qquad + r^{\frac{4}{n+4}}||(\Delta^{\bar g'}\varphi)(h' - \phi^*h'')||_{L^{\frac{n+4}{2}}([B_{\bar g_0'}(x,r)\cap A_{\bar g_0'}(x_0';R, 4R)]\times (\tfrac{r^2}{2},r^2))}\bigg),
\end{split}
\end{equation}
where $A_{\bar g_0'}(x_0'; R, 4R):= B_{\bar g_0'}(x_0', 4R)\setminus B_{\bar g_0'}(x_0', R)$, since $\varphi \equiv 1$ on $B_{\bar g_0'}(x_0', R)$

For fixed $(x,r)\in M\times (0,\sqrt{T})$, if $B_{\bar g_0'}(x,r) \subset B_{\bar g_0'}(x_0', R)$, then $\Delta^{\bar g'}\varphi \equiv 0$ on $B_{\bar g_0'}(x,r)$, so 
\begin{equation}\label{eq:vanishingpart}
w(x,r^2)r^\theta||(\Delta^{\bar g'}\varphi)(h' - \phi^*h'')||_{L^p([B_{\bar g_0'}(x,r) \cap A_{\bar g_0'}(x_0'; R, 4R)]\times (0, r^2))} = 0,
\end{equation}
for any exponent $\theta$ and positive integer $p$.

On the other hand, if there exists $y\in B_{\bar g_0'}(x,r)$ such that $y\notin B_{\bar g_0'}(x_0', R)$, then $d_{\bar g_0'}(x, x_0') \geq d_{\bar g_0'}(x_0', y) - d_{\bar g_0'}(x, y) \geq R- r$ so
\begin{equation}\label{eq:boundedpart}
\begin{split}
w_a(x,r^2)r^\theta&||(\Delta^{\bar g'}\varphi)(h' - \phi^*h'')||_{L^p([B_{\bar g_0'}(x,r) \cap A_{\bar g_0'}(x_0'; R, 4R)]\times (0, r^2))} 
\\& \leq  \max\left\{\frac{1}{(d_{\bar g_0'}(x, x_0') + r + a)^{2+\eta}}, 1\right\}||\Delta^{\bar g'}\varphi||_{L^\infty(M)}r^\theta||h' - \phi^*h''||_{L^p(B_{\bar g_0'}(x,r)\times (0, r^2))}
\\& \leq \frac{||\Delta^{\bar g'}\varphi||_{L^\infty(M)}}{R^{2+\eta}}r^\theta||h' - \phi^*h''||_{L^p(B_{\bar g_0'}(x,r)\times (0, r^2))}
\\&  \leq \frac{||\Delta^{\bar g'}\varphi||_{L^\infty(M)}}{R^{2+\eta}}r^\theta\left(||h'||_{L^p(B_{\bar g_0'}(x,r)\times (0, r^2))} + ||\phi^*h''||_{L^p(B_{\bar g_0'}(x,r)\times (0, r^2))}\right)
\end{split}
\end{equation}

Taking the supremum over all pairs $(x,r)$ and applying (\ref{eq:vanishingpart}) and (\ref{eq:boundedpart}) to (\ref{eq:Vvanishingbound})
\begin{equation}
V \leq \frac{c(n)||\Delta^{\bar g'}||_{L^\infty(M)}}{R^{2+\eta}}\left(||h'||_{Y^0} + ||h''||_{Y^0}\right) \leq c(n, R, b),
\end{equation}
where $c(n, R, b)$ is some finite positive constant, and the last inequality follows from comparing the $Y^0$-norm to the $L^\infty$-norm as in the proof of Lemma \ref{lemma:easy}, so that $||h'||_{Y^0}+ ||h''||_{Y^0} \leq ||h'||_{X} + ||h''||_X \leq c$.

\emph{Term $VI$:}
Using a similar argument to that of the analysis of Term $V$, we find
\begin{equation}
\begin{split}
\left|\left|2\int_0^t\int_M \bar K_{t-s}(x,y)\nabla^{\bar g'}_{\nabla \varphi}(h' - \phi^*h'')(y,s)d_{\bar g'_s}(y)ds\right|\right|_{\tilde X_a} &\leq \frac{c(n)||\nabla \varphi||_{L^\infty(M)}}{R^{2+\eta}}\left(||\nabla h'||_{Y^1} + ||\nabla h''||_{Y^1}\right)
\\& \leq \frac{c(n)||\nabla \varphi||_{L^\infty(M)}}{R^{2+\eta}}\left(||h'||_{X} + ||h''||_{X}\right)
\\& \leq c(n, R, b).
\end{split}
\end{equation}

Having estimated terms $I - VI$, we conclude:
\begin{equation}\label{eq:inequalitystring}
\begin{split}
||\varphi(h'-\phi^*h'')||_{\tilde X_a^{\bar g'}(B(x,r))} & \leq c(n)(||h'||_{X^{\bar g'}} + ||h''||_{X^{\bar g''}})||\varphi(h'-\phi^*h'')||_{\tilde X_a^{\bar g'}}
\\& + c(n)||w_a(0)\varphi(h_0' - \phi^*h_0'')||_{L^\infty(M)} + c(n, R, b),
\end{split}
\end{equation}
where $c(n, R, b)$ is some finite positive constant. 

Observe that $||\varphi(h'-\phi^*h'')||_{\tilde X_a^{\bar g'}} \leq \tfrac{1}{a}||\varphi(h'-\phi^*h'')||_{X} \leq \tfrac{1}{a}(||h'||_{X^{\bar g'}} + ||h''||_{X^{\bar g''}})$. Thus, $||\varphi(h'-\phi^*h'')||_{\tilde X_a^{\bar g'})}$ is finite for all $a>0$. We show that, because $||\varphi(h'-\phi^*h'')||_{\tilde X_a^{\bar g'}}$ is finite, it is in fact bounded by some constant that does not depend on $a$. 

First, recall that, by Lemma \ref{lemma:RDTexistence}, we have
\begin{equation}
||h'||_{\bar g'} + ||h''||_{\bar g''} \leq C(n)(||g_0' - \bar g_0'||_{L^\infty(M)} + ||g_0'' - \bar g_0''||_{L^\infty(M)}) \leq 2C(n)\varepsilon.
\end{equation}
Reduce $\varepsilon$ as necessary so that 
\begin{equation}\label{eq:smallcoeff}
2c(n)C(n)\varepsilon < 1/2,
\end{equation}
where $c(n)$ is the constant from (\ref{eq:inequalitystring}). Now observe that 
\begin{align*}
 ||\varphi(h' - h'')||_{\tilde X_a}  &= \sup_{0<t<T}||w_a(t)(h_t' - \phi^*h_t'')||_{L^\infty(B(x_0', R))} 
\\& \qquad + \sup_{\substack{(x,r)\in M\times (0,\sqrt{T})\\B(x,r)\subset B(x_0',R)}}w_a(x,r^2)\bigg(r^{-\frac{n}{2}}||\nabla(h' - \phi^*h'')||_{L^2(B(x,r)\times (0,r^2))} 
\\& \qquad \qquad \qquad + r^{\frac{2}{n+4}}||\nabla(h' - \phi^*h'')||_{L^{n+4}(B(x,r)\times (\tfrac{r^2}{2}, r^2))}\bigg)
\\& + \sup_{0<t<T}||w_a(t)(h_t' - \phi^*h_t'')||_{L^\infty(A(x_0'; R,4R))} 
\\& \qquad + \sup_{\substack{(x,r)\in M\times (0,\sqrt{T})\\B(x,r)\not \subset B(x_0',R)}}w_a(x,r^2)\bigg(r^{-\frac{n}{2}}||\nabla[\varphi(h' - \phi^*h'')]||_{L^2(B(x,r)\times (0,r^2))} 
\\& \qquad \qquad \qquad + r^{\frac{2}{n+4}}||\nabla[\varphi(h' - \phi^*h'')]||_{L^{n+4}(B(x,r)\times (\tfrac{r^2}{2}, r^2))}\bigg)
\\& \leq \sup_{\substack{0<t<T\\ (x,r)\in M\times (0,\sqrt{T}) \\ B(x,r)\subset B(x_0', R)}}||w_a(t)(h_t' - \phi^*h_t'')||_{L^\infty(B(x,r))} 
\\& \qquad + \sup_{\substack{(x,r)\in M\times (0,\sqrt{T})\\B(x,r)\subset B(x_0',R)}}w_a(x,r^2)\bigg(r^{-\frac{n}{2}}||\nabla(h' - \phi^*h'')||_{L^2(B(x,r)\times (0,r^2))} 
\\& \qquad \qquad \qquad + r^{\frac{2}{n+4}}||\nabla(h' - \phi^*h'')||_{L^{n+4}(B(x,r)\times (\tfrac{r^2}{2}, r^2))}\bigg)
\\& + \sup_{0<t<T}||w_a(t)(h_t' - \phi^*h_t'')||_{L^\infty(A(x_0'; R,4R))} 
\\& \qquad  + \sup_{\substack{(x,r)\in M\times (0,\sqrt{T})\\B(x,r)\not \subset B(x_0',R)}}w_a(x,r^2)\bigg(r^{-\frac{n}{2}}||(\nabla\varphi)\otimes(h' - \phi^*h'')||_{L^2(B(x,r)\times (0,r^2))} 
\\& \qquad \qquad \qquad + r^{\frac{2}{n+4}}||(\nabla\varphi)\otimes(h' - \phi^*h'')||_{L^{n+4}(B(x,r)\times (\tfrac{r^2}{2}, r^2))}\bigg)
 \\& + \sup_{\substack{(x,r)\in M\times (0,\sqrt{T})\\B(x,r)\not \subset B(x_0',R)}}w_a(x,r^2)\bigg(r^{-\frac{n}{2}}||\varphi\nabla(h' - \phi^*h'')||_{L^2(B(x,r)\times (0,r^2))} 
  \\& \qquad \qquad \qquad + r^{\frac{2}{n+4}}||\varphi\nabla(h' - \phi^*h'')||_{L^{n+4}(B(x,r)\times (\tfrac{r^2}{2}, r^2))}\bigg)
 \\& \leq \sup_{\substack{(x,r)\in M\times (0,T)\\ B(x,r)\subset B(x_0', R)}}||h' - \phi^*h''||_{\tilde X_a(B(x,r))} + \frac{(1 + ||\nabla \varphi||_{L^\infty(M)})}{R^{2+\eta}}(||h'||_{X^{\bar g'}} + ||h''||_{X^{\bar g''}})
 \\& = \sup_{\substack{(x,r)\in M\times (0,T)\\ B(x,r)\subset B(x_0', R)}}||h' - \phi^*h''||_{\tilde X_a(B(x,r))} + c(n, R, b).
\end{align*}
where the last inequality follows from estimating $w_a(x,r^2)$ as in the analysis of terms $V$ and $VI$.

Then, taking the supremum over all $(x,r)$ with $B(x,r)\subset B(x_0', R)$ so that $\varphi \equiv 1$ on $B(x,r)$ and applying (\ref{eq:smallcoeff}) to (\ref{eq:inequalitystring}) we find
\begin{equation}\label{eq:validinequalitystring}
\begin{split}
\sup_{\substack{(x,r)\in M\times (0,\sqrt{T})\\ B(x,r)\subset B(x_0', R)}}||h' - \phi^*h''||_{\tilde X_a(B(x,r))} &\leq \frac{1}{2} \left(\sup_{\substack{(x,r)\in M\times (0,\sqrt{T})\\ B(x,r)\subset B(x_0', R)}}||h' - \phi^*h''||_{\tilde X_a(B(x,r))}\right) 
\\& \qquad \qquad \qquad + c(n)||w_a(0)\varphi(h_0' - \phi^*h_0'')||_{L^\infty(M)} + c(n, R, b)
\end{split}
\end{equation}
Now, because $||\varphi(h'-\phi^*h'')||_{\tilde X_a^{\bar g'}}$ is finite, we may rearrange (\ref{eq:validinequalitystring}) to find
\begin{equation}
||h' - \phi^*h''||_{\tilde X_a(B(x_0', R))} \leq c(n)||w_a(0)\varphi(h_0' - \phi^*h_0'')||_{L^\infty(M)} + c(n, R, b).
\end{equation}
In particular, for all $a>0$, $x\in B(x_0', R)$ and $0<t<T$ we have
\begin{equation}
\begin{split}
|g_t'(x)- \phi^*g_t''(x)| &\leq  (d(x,x_0') + \sqrt{t} + a)^{2+\eta}[c(n)||w_a(0)\varphi(h_0' - \phi^*h_0'')||_{L^\infty(M)} 
\\& + ||w_a(t)(\bar g'(t)) - \phi^*\bar g''(t)||_{L^\infty(B(x_0', R))} + c(n, R, b)].
\end{split}
\end{equation}
Letting $a\searrow 0$, we find
\begin{equation}
\begin{split}
\sup_{0<t<T}||g_t' - \phi^*g_t''||_{L^\infty(B(x_0', R))} &\leq (R + \sqrt{t})^{2+\eta}[c(n)||w_0(0)\varphi(h_0' - h_0'')||_{L^\infty(M)} 
\\& +\sup_{0<t< T}||w_0(t)(\bar g'(t)) - \phi^*\bar g''(t)||_{L^\infty(B(x_0', R))} + c(n, R, b)],
\end{split}
\end{equation}
whence follows the result.
\end{proof}

We now are ready to prove Theorem \ref{thm:weakagreement}.
\begin{proof}[Proof of Theorem \ref{thm:weakagreement}]
Let $R$ be as in Theorem \ref{thm:fixedpointexistence} and reduce $R$ is necessary so that $g'$ and $\phi^*g''$ are $\varepsilon'/4$-close, where $\varepsilon'$ is the constant from Corollary \ref{cor:C0RDTderivbounds}.
We work within a time slice. Select $t\in (0,T)$, and let $\nabla$ denote the connection associated with $\bar g'(t)$. Observe that, if $\ell$ is some positive integer, then by Corollary \ref{cor:C0RDTderivbounds} we have
\begin{align*}
||\nabla^2(g_t' - \phi^*g_t'')||_{L^\infty(B_{\bar g_0'}(x_0', R))} &\leq ||\nabla^\ell g_t' - (\nabla^{\phi^*\bar g''})^\ell \phi^*g_t''||_{L^\infty(B_{\bar g_0'}(x_0', R))} 
\\& \qquad+ ||(\nabla^{\phi^*\bar g''})^{\ell}\phi^*g_t'' - \nabla^{\ell}\phi^*g_t''||_{L^\infty(B_{\bar g_0'}(x_0',R))}
\\& \leq c||\nabla^\ell g_t'||_{L^\infty(B_{\bar g_0'}(x_0',R))} + c||(\nabla^{\phi^*\bar g''})^\ell g_t''||_{L^\infty(B_{\bar g_0'}(x_0', R))} 
\\& \qquad + c||(\nabla^{\phi^*\bar g''})^{\ell}\phi^*g_t'' - \nabla^{\ell}\phi^*g_t''||_{L^\infty(B_{\bar g_0'}(x_0',R))}
\\& \leq 2c\varepsilon t^{-\ell/2} + c(\ell)t^{-\ell/2} \leq ct^{-\ell/2},
\end{align*}
due to the bound
\begin{align*}
||(\nabla^{\phi^*\bar g''})^{\ell}\phi^*g_t'' - \nabla^{\ell}\phi^*g_t''||_{L^\infty(B_{\bar g_0'}(x_0',R))} & \leq ||(\partial^{\ell - 1}\phi^*\bar g''_{ij})(\nabla^{\phi^*\bar g''} - \nabla)dx^i\otimes dx^j||_{L^\infty(B_{\bar g_0'}(x_0',R))} 
\\& + ||(\partial^{\ell - 1}\phi^*\bar g''_{ij})((\nabla^{\phi^*\bar g''})^\ell - \nabla^\ell)(dx^i\otimes dx^j)||_{L^\infty(B_{\bar g_0'}(x_0',R))} 
\\& + \cdots + ||\phi^*g''_{ij}((\nabla^{\phi^*\bar g''})^{\ell} - \nabla^{\ell})(dx^i\otimes dx^j)||_{L^\infty(B_{\bar g_0'}(x_0',R))}
\\& \leq c(\ell)t^{-\ell/2},\stepcounter{equation}\tag{\theequation}\label{eq:derivscloseweak}
\end{align*}
where $\partial$ denotes a partial derivative of the coefficient function, and the inequality follows from the fact that $\bar g'_0= \phi^*\bar g''_0$ on $B_{\bar g_0'}(x_0',R)$, as remarked in (\ref{eq:eatsweight}), and the application of Corollary \ref{cor:C0RDTderivbounds} to $g_t''$.

Now note that for all $\beta < \tfrac{1}{2}$ and all $D>0$ such that $Dt^\beta \leq R$, we have, by Theorem \ref{thm:fixedpointexistence},
\begin{equation}\label{eq:0derivsclose}
||g_t' - \phi^*g_t''||_{L^\infty(B_{\bar g_0'}(x_0', Dt^\beta))} \leq c(Dt^\beta + \sqrt{t})^{2+\eta} \leq Ct^{\beta(2+\eta)},
\end{equation}
where $C = C(n, R)$.
Moreover, if $a\leq Dt^\beta$, then we have, by Lemma \ref{lemma:derivativeinterpolation},
\begin{equation}
||\nabla^2(g_t' - \phi^*g_t'')||_{L^\infty(B_{\bar g_0'}(x_0', Dt^\beta))} \leq \frac{c}{a^2}||g_t' - \phi^*g_t''||_{L^\infty(B_{\bar g_0'}(x_0', Dt^\beta + a))} + ca^{\ell - 2}||\nabla^{\ell}(g_t' - \phi^*g_t'')||_{L^\infty(B_{\bar g_0'}(x_0', Dt^\beta + a))},
\end{equation}
where $c= c(n, \sup_{t\in [0,T], 0\leq k \leq \ell}|(\nabla^{\bar g'})^k\Rm|(\bar g'(t)), \inf_{[0,T]}\inj(\bar g'(t)))$.
We now specify some parameters: Fix $\beta\in (1/(2+\eta), 1/2)$, so that $(2+\eta)\beta >1$. Fix $\delta>0$ sufficiently small so that $-1 - 2\delta + (2+\eta)\beta > 0$. Choose $\ell$ large so that $\delta \ell - 1 - 2\delta > 0$. Let $a = t^{\tfrac{1}{2} + \delta}$. Observe that $a< t^\beta$, since $t<1$ and $1/2 + \delta > 1/2 > \beta$. By assuming that $t$ is sufficiently small (depending on $D$, $\beta$, and $\delta$), we may assume that $a < Dt^\beta$. Then we find
\begin{equation}\label{eq:secondderivsclose}
||\nabla^2(g_t' - \phi^*g_t'')||_{L^\infty(B_{\bar g_0'}(x_0', Dt^\beta))} \leq \frac{c}{t^{2(\tfrac{1}{2} + \delta)}}t^{(2+\eta)\beta} + ct^{\tfrac{\ell}{2} + \delta\ell - 1-2\delta}t^{-\ell / 2} \leq ct^\gamma,
\end{equation}
where $\gamma$ is some positive number, and $c$ does not depend on $t$. 

Moreover, using (\ref{eq:derivscloseweak}) we find
\begin{equation}\label{eq:firstderivsclose}
(||\nabla g'||_{L^\infty(B_{\bar g_0'}(x_0',Dt^\beta))} + ||\phi^*g''||_{L^\infty(B_{\bar g_0'}(x_0',Dt^\beta))})||\nabla(g' - \phi^*g'')||_{L^\infty(B_{\bar g_0'}(x_0',Dt^\beta))} \leq ct^{\gamma'}t^{-1/2}
\end{equation}
where $\gamma' > 1/2$, as follows: arguing as in the proof of (\ref{eq:secondderivsclose}), we apply Lemma \ref{lemma:derivativeinterpolation}, choosing our parameters as follows: let $a= t^{\alpha}$ where $\alpha < (2+\eta)\beta - \tfrac{1}{2}$ and choose $\ell$ sufficiently large so that $\alpha\ell - \alpha > 1$. Then
\begin{equation}
||\nabla(g' - \phi^*g'')||_{L^\infty(B_{\bar g_0'}(x_0',Dt^\beta))} \leq \frac{c}{t^{\alpha}}t^{(2+\eta)\beta} + ct^{\alpha\ell - \alpha}t^{-1/2} \leq ct^{\gamma'}
\end{equation}

We now estimate the difference in scalar curvatures. Observe:
\begin{equation}\label{eq:ptwisescalardifference}
\begin{split}
|R^{g'} - \phi^*R^{g''}| &\leq |(\nabla^2g')\star(g')^{-1} - (\nabla^2 \phi^*g'')\star(\phi^*g'')^{-1}| 
\\& \qquad + |(\nabla g')\star(\nabla g')\star(g')^{-1} - (\nabla \phi^*g'')\star(\nabla \phi^*g'') \star (\phi^*g'')^{-1}|
\\& \leq c(n)|\nabla^2(g' - \phi^*g'')||(g')^{-1}| + c(n)|\nabla^2\phi^*g''||(g')^{-1} - (\phi^*g'')^{-1}| 
\\& \qquad + c(n)(|\nabla g'| + |\nabla \phi^*g''|)|\nabla(g' - \phi^*g'')||(g')^{-1}| + |\nabla \phi^*g''|^2|(g')^{-1} - (\phi^*g'')^{-1}|
\\& \leq c|\nabla^2(g' - \phi^*g'')| + c|\nabla^2\phi^*g''||g' - \phi^*g''|
\\& \qquad + c(|\nabla g'| + |\nabla \phi^*g''|)|\nabla(g' - \phi^*g'')| + c |\nabla \phi^*g''|^2|g'- \phi^*g''|,
\end{split}
\end{equation}
where $c$ does not depend on $t$.
Moreover, observe that, by (\ref{eq:derivscloseweak}) and (\ref{eq:0derivsclose}) we have
\begin{equation}\label{eq:extraok}
|\nabla^2\phi^*g''||g' - \phi^*g''| + |\nabla \phi^*g''|^2|g'- \phi^*g''| \leq c\frac{t^{(2+\eta)\beta}}{t} = ct^{\gamma''},
\end{equation}
where $\gamma''>0$.

Combining (\ref{eq:secondderivsclose}), (\ref{eq:firstderivsclose}), (\ref{eq:extraok}), and (\ref{eq:ptwisescalardifference}), we find
\begin{equation}\label{eq:scalarsclose}
||R^{g'_t} - \phi^*R^{g''_t}||_{L^\infty(B_{\bar g_0'}(x_0',Dt^\beta))} \leq ct^\omega,
\end{equation} 
where $\omega$ is some small positive exponent, and $c$ does not depend on $t$.
Therefore, for all $C>0$, 
\begin{align*}
\limsup_{t\to 0}\left(\sup_{B_{\bar g_0'}(x_0',Ct^\beta)}|R^{g'}|(t)\right) & \leq \limsup_{t\to 0}\left(\sup_{B_{\bar g_0'}(x_0',Ct^\beta)}|R^{g'} - \phi^*R^{g''}|(t)\right) + \limsup_{t\to 0}\left(\sup_{B_{\bar g_0''}(x_0'',Ct^\beta)}|R^{g''}|(t)\right)
\\& \leq \limsup_{t\to 0}ct^\omega + \limsup_{t\to 0}\left(\sup_{B_{\bar g_0''}(x_0'',Ct^\beta)}|R^{g''}|(t)\right)
\\& = \limsup_{t\to 0}\left(\sup_{B_{\bar g_0''}(x_0'',Ct^\beta)}|R^{g''}|(t)\right),
\end{align*}
and vice-versa, so
\begin{equation}
\sup_{C>0}\left(\limsup_{t\to 0}\left(\sup_{B_{\bar g_0'}(x_0,Ct^\beta)}|R^{g'}|(t)\right)\right) = \sup_{C>0}\left(\limsup_{t\to 0}\left(\sup_{B_{\bar g_0''}(x_0,Ct^\beta)}|R^{g''}|(t)\right)\right).
\end{equation}
\end{proof}

\section{Regularizing Ricci flow}\label{sec:RRF}
In this section we pullback the Ricci-DeTurck flow from \S \ref{sec:initialderivatives} to obtain a regularizing Ricci flow.

\begin{definition}\label{def:RRF}
Let $g_0$ be a $C^0$ metric on a manifold $M$. We say that a pair $((\tilde g_t)_{t\in (0,T]}, \chi)$, where $(\tilde g_t)_{t\in (0,T]}$ is a time dependent family of smooth metrics defined on some positive time interval and $\chi$ is a continuous surjective map $M\to M$ , is a \emph{regularizing Ricci flow} for $g_0$ if the following are true:
\begin{enumerate}
\item The family $(\tilde g_t)_{t\in (0,T]}$ is a Ricci flow, i.e. for all $t\in (0,T]$ we have
\begin{equation*}
\frac{\partial \tilde g_t}{\partial t}= - 2\Ric(\tilde g_t), \text{ and }
\end{equation*}
\item there exists a smooth family of diffeomorphisms  of $M$, $(\chi_t)_{t\in (0,T]}$, such that
\begin{equation*}
\chi_t\xrightarrow[t\to 0]{C^0} \chi \qquad \text{ and } \qquad ||(\chi_t)_*\tilde g_t - g_0||_{C^0(M)}\xrightarrow[t\to 0]{} 0,
\end{equation*}
\end{enumerate}
where all norms are computed with respect to some stationary smooth background metric.
\end{definition}
We will at times suppress the continuous surjection $\chi$, and refer to the family $(\tilde g_t)_{t\in (0,T]}$ as the regularizing Ricci flow.

\begin{remark}
If $M$ is closed, then, because $T_pM$ is finite dimensional for all $p\in M$, it does not matter which stationary background metric we use to compute the $C^0$ norms in Definition \ref{def:RRF}, i.e. if we have $||(\chi_t)_*\tilde g_t - g_0||_{C^0(M)}\xrightarrow[t\to 0]{} 0$ with respect to one stationary background metric,  the same statement holds with respect to any other stationary background metric.
\end{remark}

We show the following:
\begin{theorem}[Existence of a regularizing Ricci flow]\label{thm:existence}
Let $M$ be a closed manifold. If $g_0$ is a $C^0$ metric on $M$ and $g_t$ is a Ricci-DeTurck flow on $M$ starting from $g_0$ in the sense of Corollary \ref{cor:C0RDTderivbounds}, then there exists a smooth family of diffeomorphisms $(\chi_t)_{t\in (0,T]}$ such that $\chi_t^*g_t$ is a Ricci flow defined for $t\in (0,T]$ and $\chi_t\xrightarrow[t\to 0]{C^0} \chi$, where $\chi$ is some continuous surjection of $M$. In particular, there exists a regularizing Ricci flow for any $C^0$ metric on $M$.
\end{theorem}

\begin{theorem}[Uniqueness of  regularizing Ricci flows]\label{thm:uniqueness}
Let $M$ be a closed manifold, and $g_0$ a $C^0$ metric on $M$. Suppose $((\tilde g_t^1)_{t\in(0,T^1]}, \chi^1)$ and $((\tilde g_t^2)_{t\in(0,T^2]}, \chi^2)$ are two regularizing Ricci flows for $g_0$. Then there is a stationary diffeomorphism $\alpha: M\to M$ such that $\alpha^*\tilde g_t^1 = \tilde g_t^2$ on $(0,T^1]\cap(0,T^2]$ and $\chi^1\circ\alpha = \chi^2$.
\end{theorem}

Moreover, regularizing Ricci flows are unique in a broader sense:
\begin{corollary}\label{cor:metricuniqueness}
Suppose that $(M_1, g^1)$ and $(M_2, g^2)$ are closed Riemannian manifolds and that there exists a $C^0$ metric space isometry $\varphi: M_2\to M_1$, i.e. $\varphi$ is a $C^0$ bijection with $d_{g_1}(\varphi(x),\varphi(y)) = d_{g_2}(x, y)$ for all $x, y\in M$. If $(\tilde g^1(t))_{t\in (0,T_1]}$ and $(\tilde g^2(t))_{t\in (0,T_2]}$ are regularizing Ricci flows for $g^1$ and $g^2$ respectively, then there is a diffeomorphism $\alpha: M_2 \to M_1$ such that $\alpha^*\tilde g^1(t) = \tilde g^2(t)$ for all $t\in (0,\min\{T_1, T_2\}]$.
\end{corollary}
In particular, Corollary \ref{cor:torusrigidity} is the optimal result.

\begin{remark}
Theorem \ref{thm:introthm0} is immediate from Theorem \ref{thm:existence} and Theorem \ref{thm:uniqueness}.
\end{remark}

\begin{remark}
Theorem \ref{cor:introcor1} follows from (\ref{eq:scalarsclose}) by taking $\tilde g'_t$ and $\tilde g''_t$ to be the regularizing Ricci flows to be the ones obtained from Ricci-DeTurck flows as in Theorem \ref{thm:existence}. 
\end{remark}

\begin{proof}[Proof of Theorem \ref{thm:existence}]
Choose a smooth background Ricci flow $\bar g_t$ with $||g_0 - \bar g_0||_{L^\infty(M)} < \varepsilon'$, where $\varepsilon'$ is as in Corollary \ref{cor:C0RDTderivbounds}. Find a solution $g_t$ to the Ricci-DeTurck flow starting from $g_0$ in the sense of Corollary \ref{cor:C0RDTderivbounds}, on a time interval $T$. Then $g_t\to g_0$ uniformly by Corollary \ref{cor:convergenceto0} and $g_t$ is smooth for positive times. We now show that $g_t$ pulls back to a regularizing Ricci flow. Let $\chi_t$ be a family of diffemorphisms $M\to M$ defined for $t\in (0,T]$, satisfying
\begin{equation}
\begin{cases}
(\frac{\partial}{\partial t}\chi_t) (\chi_t^{-1}(x)) &= X_{\bar g(t)}(g(t))\big|_{x} \text{ for all }x\in M,
\\ \chi_{t_0} &= \id,
\end{cases}
\end{equation}
where $X$ is as in (\ref{eq:Xoperator}) and $t_0>0$. Note that such a solution exists by standard ODE theory, since $X$ is nonsingular for $t>0$ by Lemma \ref{lemma:driftbound}.
Define $\tilde g(t)$ for $t>0$ by $\tilde g(t) = \chi_t^*g(t)$. As discussed in \S \ref{sec:preliminaries}, $\tilde g(t)$ satisfies the Ricci flow equation, and $(\chi_t)_*\tilde g(t) = g(t)\to g_0$. It remains to be shown that there exists some continuous surjection $\chi$ such that $\chi_t\xrightarrow[t\to 0]{C^0} \chi$.

First, pick a sequence of time slices, $t_i \searrow 0$. Then, by Lemma \ref{lemma:driftbound}, we find that
\begin{equation*}
d_{\bar g_0}(\chi_{t_i}(p), \chi_{t_j}(p)) \leq c(\sqrt{t_i} - \sqrt{t_j}),
\end{equation*}
for all $p\in M$ and $i\geq j$, where $c$ is as in Lemma \ref{lemma:driftbound}. In particular, $\chi_{t_i}$ is a Cauchy sequence in $C^0$, and hence converges uniformly to some continuous limit, $\chi$. If $t_i'\searrow 0$ is a different sequence of time slices, then $d_{\bar g_0}(\chi_{t_i}(p), \chi_{t_i'}(p)) \leq c|\sqrt{t_i} -\sqrt{t_i'}|$ so $\chi_{t_i'}\xrightarrow[i\to\infty]{C^0} \chi$ as well. Thus, $\chi_t\xrightarrow[t\to 0]{C^0}\chi$. 

We now show that $\chi$ is a surjection. Fix $x\in M$. By compactness, $\chi_{t_i}^{-1}(x)$ has a convergent subsequence in $M$. Pass to this subsequence, and let $y$ denote its limit. Then we have
\begin{align*}
d_{\bar g_0}(x, \chi(y)) &= d_{\bar g_0}(\chi_{t_i}(\chi_{t_i}^{-1}(x)), \chi(y))
\\& \leq d_{\bar g_0}(\chi_{t_i}(\chi_{t_i}^{-1}(x)), \chi(\chi_{t_i}^{-1}(x))) + d_{\bar g_0}(\chi(\chi_{t_i}^{-1}(x)), \chi(y))
\\& \leq \sup_{p\in M}d_{\bar g_0}(\chi_{t_i}(p), \chi(p)) + d_{\bar g_0}(\chi(\chi_{t_i}^{-1}(x)), \chi(y))
\end{align*}
so, letting $i\to \infty$, we find that $x = \chi(y)$, and hence $\chi$ is surjective.
\end{proof}

Before showing Theorem \ref{thm:uniqueness} and Corollary \ref{cor:metricuniqueness}, we will first prove:
\begin{lemma}\label{lemma:isometryofRRF}
Suppose that $M_1$ and $M_2$ are closed manifolds with $C^0$ metrics $g^1$ and $g^2$ respectively, with regularizing Ricci flows $((\tilde g^1_t)_{t\in (0,T^1]}, \chi^1)$ and $((\tilde g^2_t)_{t\in (0,T^2]}, \chi^2)$. Suppose also that there exist a sequence of numbers $\delta_i\searrow 0$, a sequence of times $t_i \searrow 0$, and a sequence of smooth maps $\varphi_i: M_2 \to M_1$ such that, for all $i$, $\varphi_i: (M_2, (\chi_{t_i}^2)_*\tilde g^2(t_i)) \to (M_1, (\chi_{t_i}^1)_*\tilde g^1(t_i))$ is a $(1+\delta_i)$-bilipschitz map, where $(\chi_t^1)_{t\in (0,T^1]}$ and $(\chi_t^2)_{t\in (0,T^2]}$ are the smooth families of diffeomorphisms given by Definition \ref{def:RRF}. Suppose also that the sequence $\varphi_i$ converges uniformly to some continuous function $\varphi:M_2\to M_1$ as $i\to \infty$. Then there exists a smooth stationary diffeomorphism $\alpha: M_2 \to M_1$ such that $\alpha^*\tilde g^1(t) = \tilde g^2(t)$ for all $t\in (0, \min\{T^1, T^2\}]$ and $\chi^1\circ \alpha = \varphi\circ\chi^2$.
\end{lemma}
\begin{proof}
Choose a smooth metric $\bar g_0$ on $M_2$ with $||g^2 - \bar g_0||_{C^0(M_2, \bar g_0)} < \varepsilon'/4$, where $\varepsilon'$ is as in Corollary $\ref{cor:C0RDTderivbounds}$. Making $i$ sufficiently large so that $||(\chi_{t_i}^2)_*\tilde g^2(t_i) - g^2||_{C^0(M_2, \bar g_0)} < \varepsilon'/4$, we find that
\begin{equation}\label{eq:initialbilipschitzcheck}
\begin{split}
||(\chi_{t_i}^2)_*\tilde g^2(t_i) - \bar g_0||_{C^0(M_2, \bar g_0)} & \leq ||(\chi_{t_i}^2)_*\tilde g^2(t_i) - g^2||_{C^0(M_2, \bar g_0)} + ||g^2 - \bar g_0||_{C^0(M_2, \bar g_0)} 
< \varepsilon'/2
\end{split}
\end{equation}
Moreover, since $\varphi_i: (M_2, (\chi_{t_i}^2)_*\tilde g^2(t_i)) \to (M_1, (\chi_{t_i}^1)_*\tilde g^1(t_i))$ is $(1+\delta_i)$-bilipschitz, (\ref{eq:initialbilipschitzcheck}) implies that
\begin{equation}
||\varphi_i^*(\chi_{t_i}^1)_*\tilde g^1(t_i) - (\chi_{t_i}^2)_*\tilde g^2(t_i)||_{C^0(M_2, \bar g_0)} \leq c(\varepsilon')||\varphi_i^*(\chi_{t_i}^1)_*\tilde g^1(t_i) - (\chi_{t_i}^2)_*\tilde g^2(t_i)||_{C^0(M_2, (\chi_{t_i}^2)_*\tilde g^2(t_i))} \leq c(\varepsilon')\delta_i.
\end{equation}
Therefore, making $i$ sufficiently large, we have
\begin{align*}
||\varphi_i^*(\chi_{t_i}^1)_*\tilde g^1(t_i) - \bar g_0||_{C^0(M_2,\bar g_0)} &\leq || \varphi_i^*(\chi_{t_i}^1)_*\tilde g^1(t_i) - (\chi_{t_i}^2)_*\tilde g^2(t_i)||_{C^0(M_2, \bar g_0)} + ||(\chi_{t_i}^2)_*\tilde g^2(t_i) - \bar g_0||_{C^0(M_2, \bar g_0)}
\\& \leq c(\varepsilon')\delta_i + \varepsilon'/2 < \varepsilon'.
\end{align*}
Thus, if $\bar g(t)$ is the Ricci flow starting from $\bar g_0$, Corollary \ref{cor:C0RDTderivbounds} implies that we may find smooth Ricci DeTurck-flows $g^{1,i}(t)$ and $g^{2,i}(t)$ with respect to $\bar g(t)$, starting from $\varphi_i^*(\chi_{t_i}^1)_*\tilde g^1(t_i)$ and $(\chi_{t_i}^2)_*\tilde g^2(t_i)$ respectively, defined for $t_i \leq t \leq T'(\bar g)$. Moreover, Lemma \ref{lemma:iterationcontraction} and (\ref{eq:distancedistortion}) imply that 
\begin{align*}
||g^{1,i}(t) - g^{2,i}(t)||_{C^0(M, \bar g_0)} &\leq c||g^{1,i}(t_i) - g^{2,i}(t_i)||_{C^0(M, \bar g_0)} 
\\& = c||\varphi_i^*(\chi_{t_i}^1)_*\tilde g^1(t_i) - (\chi_{t_i}^2)_*\tilde g^2(t_i)||_{C^0(M, \bar g_0)} \leq C\delta_i,
\end{align*} after adjusting the constant from Lemma \ref{lemma:iterationcontraction} according to (\ref{eq:distancedistortion}).

Since, for $j=1$ (resp. $j=2$), $g^{j,i}(t)$ is a smooth Ricci-DeTurck flow starting from $\varphi_i^*(\chi_{t_i}^1)_*\tilde g^1(t_i)$ (resp. $(\chi_{t_i}^2)_*\tilde g^2(t_i)$), it is isometric via a time-dependent family of diffeomorphisms to the (unique) Ricci flow starting from $\varphi_i^*(\chi_{t_i}^1)_*\tilde g^{1}(t_i)$ (resp. $(\chi_{t_i}^2)_*\tilde g^{2}(t_i)$), which is given by $\varphi_i^*(\chi_{t_i}^1)_*\tilde g^1(t)$ (resp. $(\chi_{t_i}^2)_*\tilde g^2(t)$), for $t_i \leq t \leq T'(\bar g)$.
 Thus, for $j=1,2$, there exists a smooth family of diffeomorphisms $\phi_t^{j,i}$ with $\phi_{t_i}^{j,i} = \id$ such that 
\begin{equation}
(\phi_t^{1,i})^*g^{1,i}(t) = \varphi_i^*(\chi_{t_i}^1)_*\tilde g^1(t) \text{ and }(\phi_t^{2,i})^*g^{2,i}(t) = (\chi_{t_i}^2)_*\tilde g^2(t)
\end{equation}
for $t_i \leq t \leq T'$.
Thus we find, for $j=1$ (resp. $j=2$), that $g^{1,i}(t) = (\phi_t^{1,i})_*\varphi_i^*(\chi_{t_i}^1)_*\tilde g^1(t)$ (resp. $g^{2,i}(t) = (\phi_t^{2,i})_*(\chi_{t_i}^2)_*\tilde g^2(t)$) and
\begin{equation}
||(\phi_t^{1,i})_*\varphi_i^*(\chi_{t_i}^1)_*\tilde g^1(t) -  (\phi_t^{2,i})_*(\chi_{t_i}^2)_*\tilde g^2(t)||_{C^0(M, \bar g_0)} = ||g^{1,i}(t) - g^{2,i}(t)||_{C^0(M, \bar g_0)} \leq C\delta_i
\end{equation}
for $t_i \leq t \leq T'$.
Thus, we may define $\psi_t^i := (\chi_{t_i}^1)^{-1} \circ \varphi_i \circ (\phi_{t}^{1,i})^{-1}\circ\phi_t^{2,i}\circ\chi_{t_i}^2$ and conclude that
\begin{equation}
||(\psi_t^i)^*\tilde g^1(t) - \tilde g^2(t)||_{C^0(M, (\phi_t^{2,i}\circ\chi_{t_i}^2)^*\bar g_0)} \leq C\delta_i,
\end{equation}
for $t_i \leq t \leq T'$. We would like to remove the dependence of this norm on $i$, so that we may let $i\to \infty$. To do this, observe that $||g^{2,i}(t) - \bar g(t)||_{C^0(M, \bar g_0)} \leq C\varepsilon'$, so $||\tilde g^2(t) - (\phi_t^{2,i}\circ\chi_{t_i}^2)^*\bar g(t)||_{C^0(M, (\phi_t^{2,i}\circ\chi_{t_i}^2)^*\bar g_0)} \leq C\varepsilon'$. Thus $\tilde g^2(t)$ is uniformly $(1+C\varepsilon')$-bilipschitz to $(\phi_t^{2,i}\circ\chi_{t_i}^2)^*\bar g(t)$, which is uniformly bilipschitz to $(\phi_t^{2,i}\circ\chi_{t_i}^2)^*\bar g_0$. In particular, we may measure the $C^0$-norm with $\tilde g^2(t)$ instead of $(\phi_t^{2,i}\circ\chi_{t_i}^2)^*\bar g_0$ to find
\begin{equation}\label{eq:C0metricswitch}
||(\psi_t^i)^*\tilde g^1(t) - \tilde g^2(t)||_{C^0(M, \tilde g^2(t))} \leq c(\varepsilon', \sup_{[0,T]}|\Rm|(\bar g(t)), C)\delta_i.
\end{equation}
Thus we have that $\psi_t^i: (M_2, \tilde g^2(t)) \to (M_1, \tilde g^1(t))$ is a $(1 + c\delta_i)$-bilipschitz map. In particular, Arzel\`a-Ascoli implies that, after passing to a subsequence, $\psi_t^i$ converges uniformly to some $1$-bilipschitz map $\alpha_t: (M_2,\tilde g^2(t)) \to (M_1, \tilde g^1(t))$. Since $\alpha_t$ is $1$-bilipschitz, it is an isometry, and thus it must be smooth, since $\tilde g^2(t)$ and $\tilde g^1(t)$ are smooth.

We now show that, in fact, there exists a stationary diffeomorphism $\alpha$ such that $\alpha^*\tilde g_t^1 = \tilde g_t^2$. First fix $i$. For $t \geq t_i$, $\tilde g_t^2 =\alpha_t^*\tilde g_t^1$ is a (smooth) Ricci flow starting from $\tilde g_{t_i}^2 = \alpha_{t_i}^*\tilde g_{t_i}^1$, and so is $\alpha_{t_i}^*\tilde g_t^1$. By uniqueness of smooth Ricci flows on closed manifolds, $\alpha_t^*\tilde g_t^1 = \alpha_{t_i}^*\tilde g_t^1$. In particular, $\alpha_t^*\tilde g_t^1=\alpha_{t_i}^*\tilde g_t^1$ for $t\geq t_i$, or $(\alpha_t^{-1})^*\alpha _{t_i}^*\tilde g_t^1 = \tilde g_t^1$. In particular, $(\alpha_{t_i}\circ\alpha_t^{-1})$ is an isometry of $\tilde g_t^1$, so, by compactness, there exists some limiting diffeomorphism $\beta_t: M_1\to M_1$ such that $\tilde g_t^1 = \beta_t^*\tilde g_t^1$ with $\alpha_{t_i}\circ\alpha_t^{-1}\xrightarrow[]{C^\infty}\beta_t$, after passing to a subsequence.

Define $\alpha = \beta_t\circ\alpha_t$. We show that $\alpha = \beta_t\circ\alpha_t$ is independent of choice of $t$, as follows:
\begin{equation}\label{eq:limisometry}
\beta_t \circ\alpha_t = (\lim_{i\to\infty}\alpha_{t_i}\circ\alpha_t^{-1})\circ\alpha_t = \lim_{i\to\infty}\alpha_{t_i},
\end{equation}
which is independent of $t$. Then we have
\begin{equation}
\alpha^*\tilde g_t^1 =(\alpha_t)^*(\beta_t^*\tilde g_t^1) =\alpha_t^*\circ\tilde g_t^1 = \tilde g_t^2.
\end{equation}
Thus we have shown that $\tilde g_t^1$ and $\tilde g_t^2$ are isometric by way of a stationary diffeomorphism. 

It remains to relate $\chi^1$ and $\chi^2$. Fix $x\in M$, and a time slice $t$. First, note that by Lemma \ref{lemma:driftbound} we have $d_{\bar g(t_i)}(\phi_t^{i,2}(p), \phi_{t_i}^{i,2}(p)) \leq c(\sqrt{t} - \sqrt{t_i})$ for $t\geq t_i$ and all $p\in M$, where $c$ is as in Lemma \ref{lemma:driftbound}. Similarly, 
\begin{equation*}
d_{\bar g(t_i)}((\phi_t^{i,1})^{-1}(p), p) = d_{\bar g(t_i)}(\phi_{t_i}^{i,1}\circ(\phi_t^{i,1})^{-1}(p), \phi_t^{i,1}\circ(\phi_t^{i,1})^{-1}(p)) \leq c(\sqrt{t} - \sqrt{t_i}).
\end{equation*}
Thus, using the fact that $\phi_{t_i}^{i,j} = \id$ for $j= 1,2$, we have
\begin{align*}
d_{\bar g(t_i)}((\phi_t^{i,1})^{-1}(\phi_t^{2,i}(\chi_{t_i}^2(x))), \chi_{t_i}^2(x)) & \leq d_{\bar g(t_i)}((\phi_t^{i,1})^{-1}(\phi_t^{2,i}(\chi_{t_i}^2(x))), \phi_t^{2,i}(\chi_{t_i}^2(x))) 
\\&  \qquad + d_{\bar g(t_i)}(\phi_t^{2,i}(\chi_{t_i}^2(x)), \chi_{t_i}^2(x))
\\& \leq 2c(\sqrt{t} - \sqrt{t_i})
\end{align*}
for $t\geq t_i$. In particular, by (\ref{eq:distancedistortion}) and adjusting $c$ we find
\begin{equation}
d_{\bar g(0)}((\phi_t^{i,1})^{-1}(\phi_t^{2,i}(\chi_{t_i}^2(x))), \chi_{t_i}^2(x)) \leq c(\sqrt{t} - \sqrt{t_i}),
\end{equation}
so
\begin{equation}\label{eq:unifto0}
\lim_{j\to\infty}\bigg[\lim_{i\to\infty}\left(d_{\bar g_0}(\varphi\circ(\phi_{t_j}^{1,i})^{-1}\circ\phi_{t_j}^{2,i}\circ\chi_{t_i}^2(x), \varphi\circ\chi^2_{t_i}(x))\right)\bigg] = 0.
\end{equation}
Thus, for all $j\leq i$ we have
\begin{align*}
&d_{\bar g_0}(\chi^1\circ\alpha(x), \varphi\circ\chi^2(x))
 \leq d_{\bar g_0}(\chi^1\circ\alpha(x), \chi^1\circ\alpha_{t_j}(x)) + d_{\bar g_0}(\chi^1\circ\alpha_{t_j}(x), \chi^1\circ \psi_{t_j}^i(x)) 
\\&  \qquad + d_{\bar g_0}(\chi^1\circ(\chi_{t_i}^1)^{-1}\circ\varphi_i\circ(\phi_{t_j}^{1,i})^{-1}\circ\phi_{t_j}^{2,i}\circ\chi_{t_i}^2(x), \chi_{t_i}^1\circ(\chi_{t_i}^1)^{-1}\circ\varphi_i\circ(\phi_{t_j}^{1,i})^{-1}\circ\phi_{t_j}^{2,i}\circ\chi_{t_i}^2(x))
\\&  \qquad + d_{\bar g_0}(\varphi_i\circ(\phi_{t_j}^{1,i})^{-1}\circ\phi_{t_j}^{2,i}\circ\chi_{t_i}^2(x), \varphi\circ(\phi_{t_j}^{1,i})^{-1}\circ\phi_{t_j}^{2,i}\circ\chi_{t_i}^2(x))
\\&  \qquad + d_{\bar g_0}(\varphi\circ(\phi_{t_j}^{1,i})^{-1}\circ\phi_{t_j}^{2,i}\circ\chi_{t_i}^2(x), \varphi\circ\chi_{t_i}^2(x))
\\&\qquad + d_{\bar g_0}(\varphi\circ\chi_{t_i}^2(x), \varphi\circ\chi^2(x))
\\& \leq d_{\bar g_0}(\chi^1\circ\alpha(x), \chi^1\circ\alpha_{t_j}(x)) + d_{\bar g_0}(\chi^1\circ\alpha_{t_j}(x), \chi^1\circ \psi_{t_j}^i(x))
\\& \qquad + \sup_{p\in M_1}d_{\bar g_0}(\chi^1(p), \chi_{t_i}^1(p))
\\& \qquad + \sup_{p\in M_2}d_{\bar g_0}(\varphi_i(p), \varphi(p))
\\&  \qquad + d_{\bar g_0}(\varphi\circ(\phi_{t_j}^{1,i})^{-1}\circ\phi_{t_j}^{2,i}\circ\chi_{t_i}^2(x), \varphi\circ\chi_{t_i}^2(x))
\\&\qquad + d_{\bar g_0}(\varphi\circ\chi_{t_i}^2(x), \varphi\circ\chi^2(x))
\end{align*}
In particular,
\begin{align*}
d_{\bar g_0}(\chi^1\circ\alpha(x), &\varphi\circ\chi^2(x))\leq \lim_{j\to \infty}\bigg[\lim_{i\to\infty}\bigg(d_{\bar g_0}(\chi^1\circ\alpha(x), \chi^1\circ\alpha_{t_j}(x)) + d_{\bar g_0}(\chi^1\circ\alpha_{t_j}(x), \chi^1\circ \psi_{t_j}^i(x))
\\& \qquad + \sup_{p\in M_1}d_{\bar g_0}(\chi^1(p), \chi_{t_i}^1(p)) + \sup_{p\in M_2}d_{\bar g_0}(\varphi_i(p), \varphi(p))
\\&  \qquad + d_{\bar g_0}(\varphi\circ(\phi_{t_j}^{1,i})^{-1}\circ\phi_{t_j}^{2,i}\circ\chi_{t_i}^2(x), \varphi\circ\chi_{t_i}^2(x)) + d_{\bar g_0}(\varphi\circ\chi_{t_i}^2(x), \varphi\circ\chi^2(x))\bigg)\bigg]
\\& \leq \lim_{j \to\infty} d_{\bar g_0}(\chi^1\circ\alpha(x), \chi^1\circ\alpha_{t_j}(x))
\\&\qquad + \lim_{j\to\infty}\bigg[\lim_{i\to\infty}\bigg(d_{\bar g_0}(\varphi\circ(\phi_{t_j}^{1,i})^{-1}\circ\phi_{t_j}^{2,i}\circ\chi_{t_i}^2(x), \varphi\circ\chi_{t_i}^2(x))\bigg) \bigg] = 0,
\end{align*}
by (\ref{eq:unifto0}), and hence $\chi^1\circ\alpha(x) = \varphi\circ\chi^2(x)$.
\end{proof}

\begin{proof}[Proof of Theorem \ref{thm:uniqueness}]
We apply Lemma \ref{lemma:isometryofRRF} with $M_1 = M_2 = M$ and $\varphi_i = \varphi = \id$. Let $\delta_i\searrow 0$. Since, for $j=1,2$, $(\chi_t^j)_*\tilde g^j_t\xrightarrow[t\to 0]{C^0} g_0$, there exists a sequence of times $t_i\searrow 0$ such that 
\begin{equation*}
||(\chi_{t_i}^1)_*\tilde g^1_{t_i} - (\chi_{t_i}^2)_*\tilde g^2_{t_i}||_{C^0(M, (\chi_{t_i}^2)_*\tilde g^2_{t_i})} \leq \delta_i.
\end{equation*}
Then Lemma \ref{lemma:isometryofRRF} implies that there exists a smooth stationary map $\alpha:M\to M$ such that $\alpha^*\tilde g^1(t) = \tilde g^2(t)$ for all $t\in (0,\min\{T^1, T^2\}]$ and $\chi^1\circ\alpha = \chi^2$.
\end{proof}

\begin{proof}[Proof of Corollary \ref{cor:metricuniqueness}]
For any sequence $\varepsilon_i\searrow 0$, there exists a sequence of times slices $t_i\searrow 0$ such that $\varphi: (M_2, (\chi_{t_i}^2)_*\tilde g^2_{t_i}) \to (M_1, (\chi_{t_i}^1)_*\tilde g^1_{t_i})$ is a $(1 + \varepsilon_i)$-bilipschitz map, since or $j= 1,2,$ we have $(\chi_{t}^j)_*\tilde g^j_t \xrightarrow[t\to 0]{C^0}g^j$. In particular, we may find (see, for instance, \cite[Theorem $4.4$]{Kar} and \cite[Lemma $D.1$]{BK}) smooth maps $\varphi_i: (M_2, (\chi_{t_i}^2)_*\tilde g^2_{t_i}) \to (M_1, (\chi_{t_i}^1)_*\tilde g^1_{t_i})$ that are $(1+\delta_i)$-bilipschitz, where $\delta_i$ is some sequence of positive numbers that decreases to $0$.

Moreover, the $\varphi_i$ converge to some uniform limit by Arzel\`a-Ascoli, since each $\varphi_i$ is a $(1+\delta_i + \gamma_i)$-bilipschitz map $(M_2, g^2)\to (M_1, g^1)$, where $\gamma_i$ also decreases to $0$. This is because
\begin{align*}
&||\varphi_i^*g^1 - g^2||_{L^\infty(M_2, g^2)}  \leq ||\varphi_i^*g^1 - \varphi_i^*(\chi_{t_i}^1)_*\tilde g^1_{t_i}||_{L^\infty(M_2, g^2)} + ||\varphi_i^*(\chi_{t_i}^1)_*\tilde g^1_{t_i} - (\chi_{t_i}^2)_*\tilde g_{t_i}^2||_{L^\infty(M_2, g^2)} 
\\& \qquad \qquad \qquad \qquad \qquad \qquad + ||(\chi_{t_i}^2)_*\tilde g_{t_i}^2 - g^2||_{L^\infty(M_2, g^2)}
\\& \leq ||\varphi_i^*g^1 - \varphi_i^*(\chi_{t_i}^1)_*\tilde g^1_{t_i}||_{L^\infty(M_2, g^2)} + ||\varphi_i^*(\chi_{t_i}^1)_*\tilde g^1_{t_i} - (\chi_{t_i}^2)_*\tilde g_{t_i}^2||_{L^\infty(M_2, (\chi_{t_i}^2)_*\tilde g_{t_i}^2)} 
\\& \qquad - \left( ||\varphi_i^*(\chi_{t_i}^1)_*\tilde g^1_{t_i} - (\chi_{t_i}^2)_*\tilde g_{t_i}^2||_{L^\infty(M_2, (\chi_{t_i}^2)_*\tilde g_{t_i}^2)} - ||\varphi_i^*(\chi_{t_i}^1)_*\tilde g^1_{t_i} - (\chi_{t_i}^2)_*\tilde g_{t_i}^2||_{L^\infty(M_2, g^2)} \right)
\\& \qquad \qquad + ||(\chi_{t_i}^2)_*\tilde g_{t_i}^2 - g^2||_{L^\infty(M_2, g^2)},
\end{align*}
and we set 
\begin{equation*}
\begin{split}
\gamma_i &:= ||\varphi_i^*g^2 - \varphi_i^*(\chi_{t_i}^2)_*\tilde g^2_{t_i}||_{L^\infty(M_1, g^1)} 
\\& \qquad - \left( ||\varphi_i^*(\chi_{t_i}^2)_*\tilde g^2_{t_i} - (\chi_{t_i}^1)_*\tilde g_{t_i}^1||_{L^\infty(M_1, (\chi_{t_i}^1)_*\tilde g_{t_i}^1)} - ||\varphi_i^*(\chi_{t_i}^2)_*\tilde g^2_{t_i} - (\chi_{t_i}^1)_*\tilde g_{t_i}^1||_{L^\infty(M_1, g^1)} \right) 
\\& \qquad \qquad + ||(\chi_{t_i}^1)_*\tilde g_{t_i}^1 - g^1||_{L^\infty(M_1, g^1)}.
\end{split}
\end{equation*}
Then Lemma \ref{lemma:isometryofRRF} implies that there exists a stationary smooth isometry $\alpha: (M_2, \tilde g_t^2)\to (M_1, \tilde g_t^1)$ for all $t\in (0,\min\{T^1, T^2\}]$.
\end{proof}

\section{Pointwise nonnegative scalar curvature for $C^0$ metrics}\label{sec:equivalentdefs}
In light of the previous section, we now study various properties of Definition \ref{def:RRFscalar}.

\begin{lemma}\label{lemma:flowindependence} Let $M^n$ be a manifold and $g_0$ a $C^0$ metric on $M$. Suppose that $((\tilde g_t^1)_{t\in(0,T^1]}, \chi^1)$ and $((\tilde g_t^2)_{t\in(0,T^2]}, \chi^2)$ are two regularizing Ricci flows for $g_0$. Suppose that (\ref{eqn:RRFscalar}) holds for $((\tilde g_t^2)_{t\in(0,T^2]}, \chi^2)$ at some point $y^2$ with $\chi^2(y^2) = x$. Then (\ref{eqn:RRFscalar}) also holds for $((\tilde g_t^1)_{t\in(0,T^1]}, \chi^1)$ at a point $y^1$ with $\chi^1(y^1) =x$. In particular, it is equivalent to require in Definition \ref{def:RRFscalar} that \emph{all} regularizing Ricci flows for $g_0$ satisfy (\ref{eqn:RRFscalar}) at some point in the corresponding preimage of $x$.
\end{lemma}
\begin{proof}
By Theorem \ref{thm:uniqueness} there exists an isometry $\alpha: M\to M$ with $\alpha^*\tilde g_t^1 = \tilde g_t^2$ and $\chi^1\circ\alpha = \chi^2$. Then, for fixed $C,t>0$ we have, for $x_0\in M$,
\begin{align*}
\inf_{B_{\tilde g_t^2}(y^2, Ct^\beta)}R^{\tilde g^2}(\cdot, t) & = \inf_{B_{\alpha^*\tilde g_t^1}(y^2, Ct^\beta)}R^{\alpha^*\tilde g^1}(\cdot, t)
\\&= \inf\{R^{\tilde g_1}(\alpha(y)): d_{\tilde g_t^1}(\alpha(y^2), \alpha(y)) < Ct^\beta\}
\\&= \inf\{R^{\tilde g_1}(y): d_{\tilde g_t^1}(\alpha(y^2), y) < Ct^\beta\}
\\& = \inf_{B_{\tilde g_t^1}(\alpha(y^2), Ct^\beta)}R^{\tilde g^1}(\cdot, t).
\end{align*}
By Theorem \ref{thm:uniqueness}, $\chi^1(\alpha(y^2)) = \chi^2(y^2) = x$, so we define $y^1 := \alpha(y^2)$ to find that (\ref{eqn:RRFscalar}) holds for $\tilde g_t'$ at $y^1$.
\end{proof}

\begin{lemma}
Definition \ref{def:RRFscalar} is independent of choice of $y$, i.e. if $\chi(y) = x = \chi(y')$ and (\ref{eqn:RRFscalar}) is true at $y$, then (\ref{eqn:RRFscalar}) is also true at $y'$.
\end{lemma}
\begin{proof}
First assume that the regularizing Ricci flow $((\tilde g_t)_{t\in (0,T]}, \chi)$ with corresponding family of diffeomorphisms $\chi_t$ has the property that $(\chi_t)_*\tilde g_t$ is a Ricci-DeTurck flow starting from $g_0$ in the sense of Corollary \ref{cor:C0RDTderivbounds}, as in Theorem \ref{thm:existence}.

Recall that $\chi$ is the uniform limit as $t\searrow$ of the $\chi_t$. We apply (\ref{eq:Xbound}) and the bounds on the time derivative from Corollary \ref{cor:C0RDTderivbounds}, and argue as in the proof of Lemma \ref{lemma:driftbound} to find
\begin{align*}
d_{\tilde g(t)}(y, y') &= d_{\tilde g(t)}(\chi_t^{-1}(\chi_t(y)), \chi_t^{-1}(\chi_t(y'))) = d_{(\chi_t)_*\tilde g_t}(\chi_t(y), \chi_t(y'))
\\& \leq \int_0^t\left|\frac{\partial}{\partial s}\left[d_{(\chi_s)_*\tilde g_s}(\chi_s(y), \chi_s(y'))\right]\right|ds + d_{g_0}(\chi(y), \chi(y')) \leq c\sqrt{t},
\end{align*}
where $\bar g$ is a smooth background Ricci flow as in Corollary \ref{cor:C0RDTderivbounds}, $c$ depends on $C$ and is adjusted according to the constants in Lemma \ref{lemma:driftbound} and Corollary \ref{cor:C0RDTderivbounds}, and $d_{g_0}(\chi(y), \chi(y')) = 0$ since $\chi(y) = \chi(y')$. Then for fixed $D>0$ there exists a constant $D' >0$ such that, for small $t$,
\begin{equation*}
\begin{split}
B_{\tilde g(t)}(y, D't^\beta) &\subset B_{\tilde g(t)}(y', Dt^\beta)\\
B_{\tilde g(t)}(y', D't^\beta) &\subset B_{\tilde g(t)}(y, Dt^\beta),
\end{split}
\end{equation*}
 so that (\ref{eqn:RRFscalar}) holds at $y$ if and only if it holds at $y'$. Thus Definition \ref{def:RRFscalar} is independent of choice of $y$, when $\tilde g_t$ is constructed from a Ricci-DeTurck flow, as in Theorem \ref{thm:existence}.

Now let $((\tilde g_t')_{(0,T']}, \chi')$ be an arbitrary regularizing Ricci flow for $g_0$, i.e. $\tilde g_t'$ does not necessarily come from a Ricci-DeTurck flow as in Theorem \ref{thm:existence}, and suppose that $y$ and $y'$ satisfy $\chi'(y) = x = \chi'(y')$. Then, by Theorem \ref{thm:uniqueness}, there exists a diffeomorphism $\alpha:M\to M$ such that $\tilde g_t' = \alpha^*\tilde g_t$ and $\chi' = \chi\circ\alpha$. Then, as in the proof of Lemma \ref{lemma:flowindependence}, we find that (\ref{eqn:RRFscalar}) holds for $\tilde g_t'$ at $y$ if and only if it holds for $\tilde g_t$ at $\alpha(y)$, and that it holds for $\tilde g_t'$ and $y'$ if and only if it holds for $\tilde g_t$ at $\alpha(y')$. Moreover, as discussed above, (\ref{eqn:RRFscalar}) holds for $\tilde g_t$ at $\alpha(y)$ if and only if it holds for $\tilde g_t$ at $\alpha(y')$, since $\chi(\alpha(y)) = \chi'(y) = x = \chi'(y') = \chi(\alpha(y'))$. Therefore, (\ref{eqn:RRFscalar}) holds for $\tilde g_t'$ at $y$ if and only if it holds for $\tilde g_t'$ at $y'$.
\end{proof}

\begin{remark}\label{rmk:scalinginvariance}
If $\kappa = 0$, Definition \ref{def:RRFscalar}  is invariant under rescaling: if $g_0$ has nonnegative scalar curvature at $x$ in the $\beta$-weak sense for some regularizing Ricci flow $\tilde g(t)$ for $g_0$, then, by parabolic rescaling, $g'(t):= \lambda \tilde g(t/\lambda)$ is a regularizing Ricci flow for $\lambda g_0$ ($\lambda > 0$), and we have 
\begin{align*}
\inf_{C > 0}\left(\liminf_{t\to 0}\left(\inf_{B_{\tilde g'(t)}(y, Ct^\beta)}R^{\tilde g'}(\cdot, t)\right)\right) &= \inf_{C > 0}\left(\liminf_{t\to 0}\left(\inf_{B_{\lambda \tilde g(t/\lambda)}(y, Ct^\beta)}\lambda^{-1}R^{\tilde g}(\cdot, t/\lambda)\right)\right)
\\&= \lambda^{-1}\inf_{C > 0}\left(\liminf_{t\to 0}\left(\inf_{B_{\tilde g(t/\lambda)}(x', C\lambda^{\beta-1/2}(t/\lambda)^\beta)}R^{\tilde g}(\cdot, t/\lambda)\right)\right) \geq 0.
\end{align*}
\end{remark}

We now present some equivalent formulations of Definition \ref{def:RRFscalar}. By allowing the balls to be measured by a stationary metric and reformulating Definition \ref{def:RRFscalar} in terms of Ricci-DeTurck flow, we may apply the results of \S \ref{sec:initialderivatives} and \S \ref{sec:fixedptconstruction} and use heat kernel estimates from \S \ref{sec:preliminaries}.

\begin{lemma}\label{lemma:RDTscalar}
Let $M^n$ be a manifold, $g_0$ a $C^0$ metric on $M$, and $x\in M$. The following are equivalent:
\begin{enumerate}
\item\label{item:RRFscalardef} The scalar curvature of $g_0$ is bounded below by $\kappa$ at $x$ in the $\beta$-weak sense, i.e. in the sense of Definition \ref{def:RRFscalar}.
\item\label{item:RDTdef}
If $g_t$ is a Ricci-DeTurck flow starting from $g_0$ in the sense of Corollary \ref{cor:C0RDTderivbounds} and if $\bar g_0$ is a stationary metric on $M$ that is uniformly bilipschitz to $(g_t)_{t\in (0,T]}$, then
\begin{equation*}
\inf_{C>0}\left(\liminf_{t\searrow 0}\left(\inf_{B_{\bar g_0}(x, Ct^\beta)}R^{g_t}(\cdot)\right)\right) \geq \kappa.
\end{equation*}
\end{enumerate}
\end{lemma}

\begin{proof}[Proof of Lemma \ref{lemma:RDTscalar}]

First suppose that (\ref{item:RRFscalardef}) holds, and let $g_t$ be a Ricci-DeTurck flow starting from $g_0$ in the sense of Corollary \ref{cor:C0RDTderivbounds}. Let $\tilde g_t:= \chi_t^*g_t$ as in Theorem \ref{thm:existence}. Then (\ref{item:RRFscalardef}) holds for $\tilde g_t$, with $\chi_t\to \chi$ as in Definition \ref{def:RRFscalar}, and $y\in M$ such that $\chi(y) = x$. Note first that Lemma \ref{lemma:driftbound} implies $d_{\bar g_0}(\chi_t(p), \chi(p)) \leq c\sqrt{t}$, so $d_{\bar g_0}(\chi_t(y), x) \leq c\sqrt{t}$. In particular, there is some $C'$ such that, for sufficiently small $t$,
\begin{equation*}
B_{\bar g_0}(x, C't^\beta) \subset B_{g_t}(\chi_t(y), Ct^\beta).
\end{equation*}
Now observe that
\begin{equation*}
\inf_{B_{\tilde g(t)}(y, Ct^\beta)}R^{\tilde g_t}(\cdot) = \inf_{B_{\chi_t^*g_t}(y, Ct^\beta)}R^{\chi_t^*\tilde g_t}(\cdot) = \inf_{B_{g_t}(\chi_t(y), Ct^\beta)}R^{g_t}(\cdot) \leq \inf_{B_{\bar g_0}(x, C't^\beta)}R^{g_t}(\cdot)
\end{equation*}
Thus (\ref{item:RRFscalardef}), implies that 
\begin{equation*}
\inf_{C>0}\left(\liminf_{t\searrow 0}\left(\inf_{B_{\bar g_0}(x, Ct^\beta)}R(\cdot, t)\right)\right) \geq \kappa.
\end{equation*}

Conversely, suppose that (\ref{item:RDTdef}) holds. By Lemma \ref{lemma:flowindependence} it is sufficient to show that there exists some regularizing Ricci flow for $g_0$, say $(\tilde g_t, \chi)$, for which (\ref{item:RRFscalardef}) holds. Let $g_t$ be a Ricci-DeTurck flow starting from $g_0$ in the sense of Corollary \ref{cor:C0RDTderivbounds}. By assumption, (\ref{item:RDTdef}) holds for $g_t$, and we may use $g_t$ to construct a regularizing Ricci flow $(\tilde g_t, \chi)$ as in Theorem \ref{thm:existence}. Arguing similarly as above, we find that, since $g_t$ satisfies (\ref{item:RDTdef}), (\ref{item:RRFscalardef}) is true for $\tilde g_t$.
\end{proof}

\begin{corollary}\label{cor:weakagreement}
Suppose that $g_0'$ and $g_0''$ are $C^0$ metrics on $M'$ and $M''$ respectively. Suppose also that there exists a locally defined diffeomorphism $\phi: U\to V$, where $U$ and $V$ are neighborhoods of $x_0'\in M'$ and $x_0'' \in M''$ respectively, such that $\phi(x_0') = x_0''$ and $\phi^*g_0''$ agrees to greater than second order (in the sense of Definition \ref{def:norms}) with $g_0'$ about $x_0'$. Then $g_0'$ has scalar curvature bounded below by $\kappa$ at $x_0'$ in the $\beta$-weak sense if and only if $g_0''$ does at $x_0''$, for all $\beta<1/2$ sufficiently close to $1/2$.

\end{corollary}
\begin{proof}
This follows from Theorem \ref{thm:weakagreement} and (\ref{item:RDTdef}) in Lemma \ref{lemma:RDTscalar}, since if $g'(t)$ and $g''(t)$ are Ricci-DeTurck flows starting from $g'$ and $g''$ with respect to background Ricci flows $\bar g'(t)$ and $\bar g''(t)$ respectively, as described in Theorem \ref{thm:weakagreement}, then we have
\begin{equation}
\sup_{C>0}\left(\limsup_{t\to 0}\left(\sup_{B_{\bar g_0'}(x_0', Ct^\beta)}|R^{g'}|\right)\right) = \sup_{C>0}\left(\limsup_{t\to 0}\left(\sup_{B_{\bar g_0''}(x_0'', Ct^\beta)}|R^{g''}|\right)\right).
\end{equation}
\end{proof}

\begin{proposition}
Let $\mathcal{G}_x$ denote the space of germs of $C^0$ metrics on $M$ at $x$. Definition \ref{def:RRFscalar} descends to $\mathcal{G}_x$. Moreover, if we define the equivalence relation $\sim$ on $\mathcal{G}_x$ by $[g]\sim [g']$ if $g$ and $g'$ agree to greater than second order at $x$ (in the sense of Definition \ref{def:norms}), where $g$ and $g'$ are metrics on $M$ and $[g]$ and $[g']$ are their respective germs at $x$, then Definition \ref{def:RRFscalar} descends to the quotient space $\mathcal{G}_x/\sim$.
\end{proposition}
\begin{proof}
If $g$ and $g'$ agree on a neighborhood of $x$, then they certainly agree to greater than second order about $x$, so, by Corollary \ref{cor:weakagreement}, $g$ has scalar curvature bounded below by $\kappa$ at $x$ in the $\beta$-weak sense if and only if $g'$ does. In particular, Definition \ref{def:RRFscalar} descends to $\mathcal{G}_x$.

Similarly, if $[g']\sim [g'']$, then, by Corollary \ref{cor:weakagreement}, $g'$ has scalar curvature bounded below by $\kappa$ at $x$ in the $\beta$-weak sense if and only if $g''$ does.
\end{proof}

\section{Behavior of $\beta$-weak scalar curvature under the Ricci flow}\label{sec:maxprinciple}
In this section we prove Theorems \ref{thm:introthm2} and and \ref{thm:kappaapproximation}, and Corollary \ref{cor:torusrigidity}. They follow quickly from:
\begin{proposition}\label{prop:RDTmaxprinciple}
Suppose that $g_0$ is a $C^0$ metric on a closed manifold $M$ and that there exists $\beta \in (0, 1/2)$ such that, for all $x\in M$, $g_0$ has scalar curvature bounded below by $\kappa$ at $x$ in the $\beta$-weak sense. Then, if $g(t)$ is a Ricci-DeTurck flow starting from $g_0$ in the sense of Corollary \ref{cor:C0RDTderivbounds}, and $R(\cdot, t)$ is the scalar curvature of $g(t)$, we have $R(x,t) \geq \kappa$ for all $x\in M$.
\end{proposition}

We first prove:
\begin{lemma}\label{lemma:localscalarbound}
For any smooth Ricci flow $\bar g(t)$ and all $\varepsilon>0$, $\beta \in (0,1/2)$, and $\kappa\in \R$ there exists $T= T(\varepsilon, \beta, \kappa, \bar g(t))$ such that the following is true:

Suppose $g_t$ is a Ricci-DeTurck flow starting from some $C^0$ metric $g_0$ in the sense of Corollary \ref{cor:C0RDTderivbounds}, with respect to $\bar g(t)$. For any $t\leq T$, if $g_0$ has scalar curvature bounded below by $\kappa$ in the $\beta$-weak sense everywhere in $B_{\bar g_0}\left(x,\left(\tfrac{2^\beta}{2^\beta - 1}\right)t^\beta\right)$, then $R^{g}(x, t) \geq \kappa - \varepsilon$.

\end{lemma}
\begin{remark}
If $g_0$ is smooth, then Lemma \ref{lemma:localscalarbound} says that if $R^g(z, 0) \geq \kappa$ for all $z\in B_{\bar g_0}\left(x,\left(\tfrac{2^\beta}{2^\beta - 1}\right)t^\beta\right)$, then $R^g(x, t) \geq \kappa - \varepsilon$.
\end{remark}
\begin{proof}[Proof of Lemma \ref{lemma:localscalarbound}]
Fix a smooth Ricci flow $\bar g(t)$, $\varepsilon>0$, $\beta\in (0,1/2)$, and $\kappa\in \R$. We show that following: that there exists $T$ such that if $g_t$ is a Ricci-DeTurck flow with respect to $\bar g(t)$ in the sense of Corollary \ref{cor:C0RDTderivbounds} and $R^g(x,t) < \kappa - \varepsilon$ for some $t \leq T$, then $g_0$ does not have scalar curvature bounded below by $\kappa$ in the $\beta$-weak sense everywhere in $B_{\bar g_0}\left(x,\left(\tfrac{2^\beta}{2^\beta - 1}\right)t^\beta\right)$.

Observe that, by Corollary \ref{cor:C0RDTderivbounds}, $g(t)$ satisfies the requisite derivative bounds $|\nabla^{\bar g(t)}g(t)| \leq c't^{-m/2}$ for $m= 1,2$, where $c'$ depends only on the constants from Corollary \ref{cor:C0RDTderivbounds}, and $g_t$ is uniformly $c''$-bilipschitz to  $\bar g_0$, where $c''$ depends only on $T\sup|\Rm|(\bar g(t))$. Thus the heat kernel estimate from Corollary \ref{cor:heatkernelboundforballs} holds, and the inherited constants do not depend on $g_t$.

We will show that exists a sequence of points $x^{(k)}\in B_{\bar g_0}(x^{(k-1)}, (t/2^{k-1})^\beta)$ such that
\begin{equation}\label{eq:subsequenceproperty}
R(x^{(k)}, t/2^k) \leq \kappa -\varepsilon + \sum_{i=1}^{k}\left[\frac{Cn}{(t/2^i)} + C\kappa\right]\exp\left(-\frac{(t/2^{i-1})^{2\beta}}{D(t/2^i)}\right) < \kappa - \frac{\varepsilon}{2},
\end{equation}
where $C= C(n,c,\sup|\Rm|(\bar g_0))$ and $D= D(c',c'')$ are the constants from Corollary \ref{cor:heatkernelboundforballs}.

\noindent \textbf{Claim.} There exists $T=T(\varepsilon, \beta, \kappa, C, D)$ such that, for $0<t\leq T$, we have
\begin{equation}\label{eq:expseriessmall}
\sum_{i=1}^{\infty}\left[\frac{Cn}{(t/2^i)} + C\kappa \right]\exp\left(-\frac{(t/2^{i-1})^{2\beta}}{D(t/2^i)}\right) < \frac{\varepsilon}{2}
\end{equation}
\begin{proof}[Proof of claim]
First observe that 
\begin{equation*}
\exp\left(-\frac{(t/2^{i-1})^{2\beta}}{D(t/2^i)}\right) \leq c(a)\left(\frac{(t/2^{i-1})^{2\beta}}{D(t/2^i)}\right)^{-a} = c(a)\left(\frac{2^{2\beta}}{D}\right)^{-a}\frac{t^{a(1-2\beta)}}{2^{ai(1-2\beta)}}.
\end{equation*}
In particular,
\begin{equation*}
\begin{split}
\left[\frac{Cn}{(t/2^i)} + C\kappa\right]\exp\left(-\frac{(t/2^{i-1})^{2\beta}}{D(t/2^i)}\right) &\leq Cnc(a)\left(\frac{2^{2\beta}}{D}\right)^{-a}t^{a(1-2\beta)-1}\left(\frac{1}{2^{a(1-2\beta)-1}}\right)^i 
\\& + C\kappa c(a)\left(\frac{2^{2\beta}}{D}\right)^{-a}t^{a(1-2\beta)}\left(\frac{1}{2^{a(1-2\beta)}}\right)^i.
\end{split}
\end{equation*}
Now pick $a = a(\beta) > 1/(1-2\beta)$ so that $1/2^{a(1-2\beta) - 1} < 1$, and pick $T = T(\varepsilon, a, \kappa)$ such that if $0<t\leq T$, we have
\begin{equation*}
\begin{split}
\sum_{i=1}^{\infty}\left[\frac{Cn}{(t/2^i)} + C\kappa\right]\exp\left(-\frac{(t/2^{i-1})^{2\beta}}{D(t/2^i)}\right) & \leq Cnc(a)\left(\frac{2^{2\beta}}{D}\right)^{-a}t^{a(1-2\beta)-1}\sum_{i=1}^{\infty}\left(\frac{1}{2^{a(1-2\beta)-1}}\right)^i 
\\& + C\kappa c(a)\left(\frac{2^{2\beta}}{D}\right)^{-a}t^{a(1-2\beta)}\sum_{i=1}^{\infty}\left(\frac{1}{2^{a(1-2\beta)}}\right)^i < \frac{\varepsilon}{2}.
\end{split}
\end{equation*}
\end{proof}
Now let $(x,t)$ be a pair with $R(x,t) \leq \kappa -\varepsilon$, such that $t\leq T$. Then (\ref{eq:expseriessmall}) implies that, for all $k$,
\begin{equation}\label{eq:experrorneg}
 \sum_{i=1}^{k}\left[\frac{Cn}{(t/2^i)} + C\kappa \right]\exp\left(-\frac{(t/2^{i-1})^{2\beta}}{D(t/2^i)}\right) <  \frac{\varepsilon}{2}.
\end{equation}
We construct a sequence $x^{(k)}$ iteratively as follows: set $(x^0, t) := (x,t)$. Then, given $x^{(k)} \in B_{\bar g_0}(x^{(k-1)}, (t/2^{k-1})^\beta)$ such that 
\begin{equation*}
R(x^{(k)}, t/2^k) \leq \kappa -\varepsilon + \sum_{i=1}^{k}\left[\frac{Cn}{(t/2^i)} + C\kappa\right]\exp\left(-\frac{(t/2^{i-1})^{2\beta}}{D(t/2^i)}\right)
\end{equation*}
 we show that there exists $x^{(k+1)} \in B_{\bar g_0}(x^{(k)}, (t/2^{k})^\beta)$ such that 
\begin{equation*}
R(x^{(k+1)}, t/2^{k+1}) \leq \kappa -\varepsilon + \sum_{i=1}^{k+1}\left[\frac{Cn}{(t/2^i)} + C\kappa\right]\exp\left(-\frac{(t/2^{i-1})^{2\beta}}{D(t/2^i)}\right).
\end{equation*}
Suppose that no such $x^{(k+1)}$ exists. Then, for all $y\in B_{\bar g_0}(x^{(k)}, (t/2^{k})^\beta)$ we have
\begin{equation*}
R(y,t/2^{k+1}) - \kappa >  -\varepsilon + \sum_{i=1}^{k+1}\left[\frac{Cn}{(t/2^i)} + C\kappa \right]\exp\left(-\frac{(t/2^{i-1})^{2\beta}}{D(t/2^i)}\right).
\end{equation*}
 By (\ref{eq:RDTRevolution}) we may use (\ref{eq:scalarlowerbound}), (\ref{eq:experrorneg}), and Corollary \ref{cor:heatkernelboundforballs} to find
\begin{align*}
 -\varepsilon &+ \sum_{i=1}^{k}\frac{Cn}{(t/2^i)}\exp\left(-\frac{(t/2^{i-1})^{2\beta}}{D(t/2^i)}\right) \geq R(x^{(k)}, t/2^k) - \kappa
\\& = \int_{B_{\bar g_0}(x^{(k)}, (t/2^{k})^\beta)}\Phi(x^{(k)}, t/2^k; y, t/2^{k+1})[R(y,t/2^{k+1}) - \kappa]d_{g(t/2^{k+1})}(y) 
\\& + \int_{M\setminus B_{\bar g_0}(x^{(k)}, (t/2^{k})^\beta)}\Phi(x^{(k)}, t/2^k; y, t/2^{k+1})[R(y,t/2^{k+1}) - \kappa ]d_{g(t/2^{k+1})}(y)
\\& > \left(-\varepsilon + \sum_{i=1}^{k+1}\left[\frac{Cn}{(t/2^i)} + C\kappa\right]\exp\left(-\frac{(t/2^{i-1})^{2\beta}}{D(t/2^i)}\right)\right)\int_{B_{\bar g_0}(x^{(k)}, (t/2^{k})^\beta)}\Phi(x^{(k)}, t/2^k; y, t/2^{k+1})d_{g(t/2^{k+1})}(y)
\\& - \left[\frac{n}{2(t/2^{k+1})} + \kappa \right] \int_{M\setminus B_{\bar g_0}(x^{(k)}, (t/2^{k})^\beta)}\Phi(x^{(k)}, t/2^k; y, t/2^{k+1})d_{g(t/2^{k+1})}(y)
\\& \geq  -\varepsilon + \sum_{i=1}^{k+1}\left[\frac{Cn}{(t/2^i)} + C\kappa\right]\exp\left(-\frac{(t/2^{i-1})^{2\beta}}{D(t/2^i)}\right) - \left[\frac{Cn}{2(t/2^{k+1})} + C\kappa\right]\exp\left(-\frac{(t/2^{k})^{2\beta}}{D(t/2^{k+1})}\right)
\\& \geq  - \varepsilon + \sum_{i=1}^k\left[\frac{Cn}{t/2^{i}} + C\kappa\right]\exp\left(-\frac{(t/2^{i-1})^{2\beta}}{D(t/2^i)}\right) + \left[\frac{Cn}{(t/2^k)} + C\kappa\right]\exp\left(-\frac{(t/2^{k})^{2\beta}}{D(t/2^{k+1})}\right).
\end{align*}
In particular, we may conclude that
\begin{equation*}
0> \left[\frac{Cn}{(t/2^k)} + C\kappa\right]\exp\left(-\frac{(t/2^{k})^{2\beta}}{D(t/2^{k+1})}\right),
\end{equation*}
which is a contradiction.

Thus a sequence $x^{(k)}$ satisfying (\ref{eq:subsequenceproperty}) exists. By (\ref{eq:expseriessmall}), the sequence $x^{(k)}$ satisfies
\begin{equation*}
R(x^{(k)}, t/2^k) < \kappa - \frac{\varepsilon}{2}.
\end{equation*}
Furthermore, for $k< j$, we have
\begin{equation*}
d_{\bar g_0}(x^{(k)}, x^{(j)}) \leq \sum_{i=k}^{j-1}d_{\tilde g}(x^{(i)}, x^{(i+1)}) < t^{\beta}\sum_{i=k}^{\infty}\frac{1}{2^{\beta i}} \xrightarrow[k\to \infty]{}0,
\end{equation*}
since the tails of a convergent series must always tend to $0$. Thus $\left(x^{(k)}\right)$ is Cauchy, and hence converges to some $x^\infty\in M$. Moreover, for all $k$,
\begin{equation}
d_{\bar g_0}(x^{(k)}, x^\infty) < t^\beta\sum_{i=k}^{\infty}\frac{1}{2^{\beta i}} = \left(\frac{2^\beta}{2^{\beta} - 1}\right)\left(\frac{t}{2^k}\right)^\beta,
\end{equation}
and
\begin{equation}
d_{\bar g_0}(x, x^\infty) = d_{\bar g_0}(x^0, x^\infty) < \left(\frac{2^\beta}{2^\beta - 1}\right)t^\beta.
\end{equation}
Let $D = 2^\beta/(2^{\beta} - 1)$, so that
\begin{equation*}
\liminf_{t\to 0}\left(\inf_{B_{\bar g_0}(x^{\infty}, Dt^\beta)}R(\cdot, t)\right) \leq \liminf_{k\to \infty}\left(\inf_{B_{\bar g_0}(x^{\infty}, D(t/2^k)^\beta)}R(\cdot, t/2^k)\right) \leq \lim_{k\to\infty}R(x^{(k)}, t/2^k) < \kappa -\varepsilon/2.
\end{equation*}
This  violates (\ref{item:RDTdef}) in Lemma \ref{lemma:RDTscalar} and $x_\infty$, so $g_0$ does not have scalar curvature bounded below by $\kappa$ in the $\beta$-weak sense at $x_\infty$.
\end{proof}

\begin{proof}[Proof of Proposition \ref{prop:RDTmaxprinciple}]
Let $\varepsilon_i\searrow 0$. Let $t_i\searrow 0$ be a sequence of times such that, for all $i$, $t_i \leq T(\varepsilon_i, \beta, \kappa)$, where $T(\varepsilon_i, \beta, \kappa)$ is given by Lemma \ref{lemma:localscalarbound}. Then $R(\cdot, t_i) \geq \kappa - \varepsilon_i$ for all $i$, so we may the usual maximum principle for smooth solutions to the Ricci-DeTurck flow starting from time $t_i$ (see (\ref{eq:Rpreservation})), to find
\begin{equation*}
R(\cdot, t) \geq \kappa - \varepsilon_i \text{ for } t\geq t_i.
\end{equation*}
In particular, for all $t>0$ we have
\begin{equation*}
R(\cdot, t) \geq \lim_{i\to\infty} \kappa - \varepsilon_i = \kappa.
\end{equation*}
\end{proof}

\begin{proof}[Proof of Theorem \ref{thm:introthm2}]
We may assume without loss of generality that $(\tilde g_t, \chi)$ is constructed from a Ricci-DeTurck flow starting from $g_0$ in the sense of Corollary \ref{cor:C0RDTderivbounds} as in Theorem \ref{thm:existence}, i.e. that there exists a smooth family of diffeomorphisms $\chi_t$ such that $(\chi_t)_*\tilde g_t =: g_t$ is such a Ricci-DeTurck flow, because Theorem \ref{thm:uniqueness} implies that $(\tilde g_t, \chi)$ is isometric to such a regularizing Ricci flow. By Proposition \ref{prop:RDTmaxprinciple}, $R^{g_t}(x) \geq \kappa$ for all $x\in M$. Therefore, $R^{\tilde g_t}(x) = R^{\chi_t^*g_t}(x) = R^{g_t}(\chi_t(x)) \geq \kappa$ for all $x\in M$.
\end{proof}

\begin{proof}[Proof of Theorem \ref{thm:kappaapproximation}]
Fix a smooth background metric $\bar g_0$ with $||g - \bar g_0||_{C^0(M, \bar g_0)} < \varepsilon'/2$, where $\varepsilon'$ is as in Corollary \ref{cor:C0RDTderivbounds}. Then, for sufficiently large $i$, there exists a Ricci-DeTurck flow $g_i(t)$ starting from $g_i$ in the sense of Corollary \ref{cor:C0RDTderivbounds}, which is smooth for $t>0$. By Proposition \ref{prop:RDTmaxprinciple}, for all $i$ we have $R(g_i(t)) \geq \kappa_i$, for $t>0$.

Let $g(t)$ denote the Ricci-DeTurck flow starting from $g$ in the sense of Corollary \ref{cor:C0RDTderivbounds}. By Corollary \ref{cor:C0RDTderivbounds}, $g_i(t)\xrightarrow[i\to\infty]{C^\infty_{\loc}(M\times (0,T')}g(t)$, so for all $x\in M$, $C>0$, and $t\in (0,T']$ we have
\begin{equation}
\inf_{B_{\bar g_0}(x, Ct^\beta)}R^{g(t)}(\cdot) \geq \lim_{i\to\infty}\inf\inf_{B_{\bar g_0}(x, Ct^\beta)}R^{g_i(t)}(\cdot) \geq \lim_{i\to\infty} \kappa_i = \kappa.
\end{equation}
In particular, $g$ has scalar curvature bounded below by $\kappa$ in the $\beta$-weak sense everywhere, so by Theorem \ref{thm:introthm2}, any regularizing Ricci flow $(\tilde g_t)_{t\in (0,T]}$ for $g$ satisfies $R(\tilde g_t) \geq \kappa$ for all $t\in (0,T]$.
\end{proof}

\begin{proof}[Proof of Corollary \ref{cor:torusrigidity}]
Let $((\tilde g_t)_{t\in (0,T]}, \chi)$ be a regularizing Ricci flow for $g$, and let $(\chi_t)_{t\in (0,T]}$ be the smooth family of diffeomorphisms given by Definition \ref{def:RRF}. By Theorem \ref{thm:introthm2}, $R(\tilde g(t)) \geq 0$ for all $t\in (0,T]$, so by the scalar torus rigidity theorem (see \cite[Corollary $2$]{SY} and \cite[Corollary A]{GL}), $\tilde g(t)$ is flat for all $t\in (0,T]$. By uniqueness of the Ricci flow starting from any (smooth) positive time slice, we must have that $\tilde g(t) = g_{\flat}$ for all $t\in (0,T]$, where $g_{\flat}$ is some fixed flat metric on $\mathbb{T}$. By Definition \ref{def:RRF}, we have $(\chi_t)_*g_{\flat}\xrightarrow[t\to 0]{C^0} g$, so the constant sequence $g_{\flat}$ converges in the Gromov-Hausdorff sense to $g$. In particular, $(\mathbb{T}, d_g)$ is isometric as a metric space to $(\mathbb{T}, d_{g_{\flat}})$.
\end{proof}

\begin{appendices}
\section{Iteration scheme for the Ricci-DeTurck flow}\label{appendix:iterationscheme}
The aim of this section is essentially to show that closeness of two metrics is stable under the Ricci-DeTurck flow, by appealing to the Banach fixed point theorem. We first record the unweighted versions of several results that we have proven in \S \ref{sec:fixedptconstruction}; see also \cite[Lemmata $2.2$, $4.1$, and $4.2$]{KL1}.
\begin{lemma}
We have, for every $0<\gamma < 1$ and every two $h', h''\in X^{\gamma}$,
\begin{equation}\label{eq:easylemunweighted}
||Q^0[h'] - Q^0[h''] + \nabla^*Q^1[h']- \nabla^*Q^1[h'']||_{Y} \leq c_1(||h'||_X + ||h''||_X)||h' - h''||_X,
\end{equation}
where $c_1 = c_1(n,\gamma, T\sup_{t\in [0,T]}|\Rm|(\bar g_t))$.
Moreover, if $(\partial_t + L)h\equiv 0$ with initial condition $h_0\in L^\infty(M)$, then
\begin{equation}\label{eq:homogeneouslemunweighted}
||h||_X \leq c_2||h_0||_{L^\infty(M)},
\end{equation}
where $c_2$ is the constant from Lemma \ref{lemma:homogeneous}.
Finally, if $(\partial_t + L)h = Q\in Y$ and $h_0 \equiv 0$, then
\begin{equation}\label{eq:hardlemunweighted}
||h||_X \leq c_3||Q||_Y,
\end{equation}
where $c_3$ is the constant from Lemma \ref{lemma:hard}.
\end{lemma}
\begin{proof}
The proofs of (\ref{eq:homogeneouslemunweighted}) and (\ref{eq:hardlemunweighted}) follow similarly to the proofs of Lemmata \ref{lemma:homogeneous} and \ref{lemma:hard}, by omitting the weights. The proof of (\ref{eq:easylemunweighted}) is similar to the analysis of Terms $I$ and $II$ in the proof of Theorem \ref{thm:fixedpointexistence} but we shall give more details here. The analysis is simplified by the fact that $h'$ and $h''$ are solutions with respect to the same background metric. 
First observe that if $||h'||_X, ||h''||_X < \gamma < 1$, then (\ref{eq:inverseexpansion1}) and (\ref{eq:inverseexpansion3}) imply
\begin{align*}
|(\bar g + h')^{-1}\star (\bar g + h')^{-1} &- (\bar g + h'')^{-1}\star (\bar g + h'')^{-1}| \leq |(\bar g + h')^{-1}\star (\bar g + h')^{-1} - (\bar g + h')^{-1}\star (\bar g + h'')^{-1}| 
\\& \qquad + |(\bar g + h')^{-1}\star (\bar g + h'')^{-1} - (\bar g + h'')^{-1}\star (\bar g + h'')^{-1}|
\\& \leq |(\bar g + h')^{-1}||(\bar g + h')^{-1} - (\bar g + h'')^{-1}| + |(\bar g + h')^{-1} - (\bar g + h'')^{-1}||(\bar g + h'')^{-1}|
\\& \leq |(\bar g + h')^{-1}|^2 |(\bar g + h'')^{-1}| |h' - h''| + |(\bar g + h'')^{-1}|^2 |(\bar g + h')^{-1}| |h' - h''|
\\& \leq c(n,\gamma)|h' - h''|.
\end{align*}

Thus we have
\begin{align*}
|Q^0[h'] - Q^0[h'']| & \leq |(\bar g + h')^{-1}\star(\bar g + h')^{-1} \star \nabla h' \star \nabla h' - (\bar g + h'')^{-1}\star(\bar g + h'')^{-1} \star \nabla h'' \star \nabla h''| 
\\& + |[(\bar g + h')^{-1} - \bar g]\star \Rm^{\bar g}\star h' \star h' - [(\bar g + h'')^{-1} - \bar g]\star\Rm^{\bar g}\star h'' \star h''|
\\& \leq |(\bar g + h')^{-1}\star(\bar g + h')^{-1} \star \nabla h' \star \nabla h' - (\bar g + h')^{-1}\star(\bar g + h')^{-1} \star \nabla h' \star \nabla h''| 
\\& + |(\bar g + h')^{-1}*(\bar g + h')^{-1} \star \nabla h' \star \nabla h'' - (\bar g + h'')^{-1}\star(\bar g + h'')^{-1} \star \nabla h' \star \nabla h''|
\\& + |(\bar g + h'')^{-1}\star(\bar g + h'')^{-1} \star \nabla h' \star \nabla h'' - (\bar g + h'')^{-1}\star(\bar g + h'')^{-1} \star \nabla h'' \star \nabla h''|
\\& + |[(\bar g + h')^{-1} - \bar g]\star\Rm^{\bar g}\star h' \star h' - [(\bar g + h')^{-1} - \bar g]\star\Rm^{\bar g}\star h' \star h''| 
\\& + 
|[(\bar g + h')^{-1} - \bar g]\star \Rm^{\bar g}\star h' \star h'' - [(\bar g + h'')^{-1} - \bar g]\star \Rm^{\bar g}\star h' \star h''|
\\& + |[(\bar g + h'')^{-1} - \bar g]\star\Rm^{\bar g}\star h' \star h'' - [(\bar g + h'')^{-1} - \bar g]\star \Rm^{\bar g}\star h'' \star h''|
\\& \leq c(n, \gamma)|\nabla h'||\nabla(h' - h'')| + c(n)|h' - h''||\nabla h'||\nabla h''| + c(n,\gamma)|\nabla(h' - h'')||\nabla h''|
\\& + c(n,\gamma)|\Rm^{\bar g}||h'||h' - h''| + c(n)|h' - h''||\Rm^{\bar g}||h'||h''| + c(n,\gamma)|\Rm^{\bar g}||h'-h''||h''|,
\end{align*}
appealing to (\ref{eq:inverseexpansion1}). We also have
\begin{align*}
|Q^1[h'] - Q^1[h'']| & \leq |[(\bar g + h')^{-1} - \bar g^{-1}]\star h' \star \nabla h' - [(\bar g + h'')^{-1} - \bar g^{-1}]\star h'' \star \nabla h''|
\\& \leq |[(\bar g + h')^{-1} - \bar g^{-1}]\star h' \star \nabla h' - [(\bar g + h')^{-1} - \bar g^{-1}]\star h' \star \nabla h''| 
\\& + |[(\bar g + h')^{-1} - \bar g^{-1}]\star h' \star \nabla h'' - [(\bar g + h')^{-1} - \bar g^{-1}]\star h'' \star \nabla h''|
\\& + |[(\bar g + h')^{-1} - \bar g^{-1}]\star h'' \star \nabla h'' - [(\bar g + h'')^{-1} - \bar g^{-1}]\star h'' \star \nabla h''|
\\& \leq c(n,\gamma)|h'||\nabla (h' - h'')| + c(n,\gamma)|h' - h''||\nabla h''| + c(n,\gamma)|h''||\nabla h''||h' - h''|.
\end{align*}
Then the estimate (\ref{eq:easylemunweighted}) follows from the definitions of the $X$ and $Y$ norms, much as in the proof of Lemma \ref{lemma:easy}.
\end{proof}

\begin{lemma}\label{lemma:iterationcontraction}
Suppose $\bar g(t)$ is a smooth Ricci flow background defined for $t\in [0,T]$. Suppose $g_0', g_0''\in L^{\infty}(M)$ are two initial metrics such that  $||g_0' - \bar g(0)||_{L^\infty(M)}, ||g_0''-\bar g(0)||_{L^\infty(M)} < \varepsilon$, where $\varepsilon$ is given by Lemma \ref{lemma:RDTexistence} and reduced, if necessary, so that $2c_1c_3C\varepsilon < 1$, where $C$ is the constant from Lemma \ref{lemma:RDTexistence} and $c_1$ and $c_3$ are as in (\ref{eq:easylemunweighted}) and (\ref{eq:hardlemunweighted}) respectively. Let $h'(t)$ and $h''(t)$ be solutions to the integral equation (\ref{eq:integraleq}) in $X^\gamma$ for some $0<\gamma < 1$, given by Lemma \ref{lemma:RDTexistence}, starting from $g_0'$ and $g_0''$ respectively. Then
\begin{equation}\label{eq:C0convunweighted}
||h' - h''||_{X} \leq c||g_0' - g_0''||_{L^\infty(M)},
\end{equation}
where $c = c(n, \gamma, T\sup_{[0,T]}|\Rm|(\bar g_t))$.
\end{lemma}
\begin{proof}
Observe that $||h'||_{X}, ||h''||_{X} \leq C\varepsilon$, by Lemma \ref{lemma:RDTexistence}. Moreover, the proof of Lemma \ref{lemma:RDTexistence} implies that $F[\cdot, h_0']$ and $F[\cdot, h_0'']$ are contraction mappings $X^{C\varepsilon}\to X^{C\varepsilon}$. For $i\in \N$ let $h_i' = F^i[0, h_0']$, i.e. $F[\cdot, h_0']$ applied to the $0$-tensor $i$ times, and let $h_i'' = F^i[0,h_0'']$, so that $h_i'\to h'$ in $X$ as $i\to \infty$, and similarly for $h''$, by the Banach fixed point theorem. We show
\begin{equation}\label{eq:iterationcontraction}
||h_i' - h_i''||_{X} \leq c_2\sum_{k=0}^{i-1}(2c_1c_3C\varepsilon)^{k}||h_0' - h_0''||_{L^\infty(M)}.
\end{equation}
To prove (\ref{eq:iterationcontraction}) we induct. We have
\begin{equation*}
||h_1' - h_1''||_{X} = \left|\left|\int_M\bar K(x,y) (h_0' - h_0'')(y)dy\right|\right|_{X} \leq c_2||h_0' - h_0''||_{L^\infty(M)},
\end{equation*}
by (\ref{eq:homogeneouslemunweighted}). Moreover, supposing (\ref{eq:iterationcontraction}) holds for $i-1$, we appeal to (\ref{eq:hardlemunweighted}), (\ref{eq:homogeneouslemunweighted}), and (\ref{eq:easylemunweighted}) to find
\begin{align*}
||h_i' - h_i''||_X & \leq ||F[h_{i-1}', h_0'] - F[h_{i-1}'', h_0']||_X + ||F[h_{i-1}'', h_0'] - F[h_{i-1}'', h_0'']||_X
\\& = \left|\left|\int_0^t\int_M\bar K_{t-s}(x,y)(Q[h_{i-1}'] - Q[h_{i-1}''])(y,s)dyds\right|\right|_X + \left|\left|\int_M\bar K(x,y)(h_0' - h_0'')(y)dy\right|\right|_X
\\& \leq c_3||Q[h_{i-1}'] - Q[h_{i-1}'']||_Y + c_2||h_0' - h_0''||_{L^\infty(M)}
\\& \leq c_1c_3(||h_{i-1}'||_{X} + ||h_{i-1}''||_{X})||h_{i-1}' - h_{i-1}''||_X + c_2||h_0' - h_0''||_{L^\infty(M)}
\\& \leq 2c_1c_3C\varepsilon c_2\sum_{k=0}^{i-2}(2c_1c_3C\varepsilon)^{k}||h_0' - h_0''||_{L^\infty(M)} + c_2||h_0' - h_0''||_{L^\infty(M)} .
\end{align*}
Taking limits, we obtain (\ref{eq:C0convunweighted}), with $c = c_2\sum_{k=0}^{\infty}(2c_1c_3C\varepsilon)^k$.
\end{proof}
\end{appendices}

\bibliography{C0SCCbib}

\begin{bibdiv}
\begin{biblist}

\bib{Bam2}{article}{
      author={Bamler, Richard~H},
       title={Stability of hyperbolic manifolds with cusps under ricci flow},
        date={2014},
     journal={Adv. Math.},
      volume={263},
       pages={412 \ndash  467},
        note={MR3239144},
}

\bib{Bam}{article}{
      author={Bamler, Richard~H.},
       title={A {R}icci flow proof of a result by {G}romov on lower bounds for
  scalar curvature},
        date={2016},
     journal={Math. Res. Lett.},
      volume={23},
      number={2},
       pages={325 \ndash  337},
        note={MR3512888},
}

\bib{BK}{unpublished}{
      author={Bamler, Richard~H.},
      author={Kleiner, Bruce},
       title={Uniqueness and stability of {R}icci flow through singularities},
        date={2018},
        note={\texttt{arXiv 1709.04122}},
}

\bib{Xphdthesis}{thesis}{
      author={Bamler, Richard~Heiner},
       title={Stability of {E}instein metrics of negative curvature},
        type={Ph.D. Thesis},
        date={2011},
}

\bib{CCG+}{book}{
      author={Chow, Bennett},
      author={Chu, Sun-Chin},
      author={Glickenstein, David},
      author={Guenther, Christine},
      author={Isenberg, James},
      author={Ivey, Tom},
      author={Knopf, Dan},
      author={Lu, Peng},
      author={Luo, Feng},
      author={Ni, Lei},
       title={The {R}icci low: techniques and applications. {P}art ${I}{I}{I}$:
  {G}eometric-analytic aspects},
      series={Mathematical Surveys and Monographs},
   publisher={American Mathematical Society},
     address={Providence, RI},
        date={2010},
      volume={163},
        note={MR2604955},
}

\bib{De}{article}{
      author={DeTurck, Dennis~M.},
       title={Deforming metrics in the direction of their {R}icci tensors},
        date={1983},
     journal={J. Differential Geom.},
      volume={18},
      number={1},
       pages={157\ndash 162},
        note={MR0697987},
}

\bib{GL}{article}{
      author={Gromov, Mikhael},
      author={Jr., H. Blaine~Lawson{,}},
       title={Spin and scalar curvature in the presence of a fundamental group.
  {I}},
        date={1980},
     journal={Ann. of Math. (2)},
      volume={111},
      number={2},
       pages={209\ndash 230},
        note={MR0569070},
}

\bib{Gro}{article}{
      author={Gromov, Misha},
       title={Dirac and {P}lateau billiards in domains with corners},
        date={2014},
     journal={Cent. Eur. J. Math.},
      volume={12},
      number={8},
       pages={1109\ndash 1156},
        note={MR3201312},
}

\bib{Kar}{article}{
      author={Karcher, H.},
       title={Riemannian center of mass and mollifier smoothing},
        date={1977},
     journal={Comm. Pure Appl. Math.},
       pages={509\ndash 541},
        note={MR0442975},
}

\bib{KL1}{article}{
      author={Koch, Herbert},
      author={Lamm, Tobias},
       title={Geometric flows with rough initial data},
        date={2012},
     journal={Asian J. Math.},
      volume={16},
       pages={209\ndash 236},
        note={MR2916362},
}

\bib{KL2}{article}{
      author={Koch, Herbert},
      author={Lamm, Tobias},
       title={Parabolic equations with rough data},
        date={2015},
     journal={Math. Bohem.},
      volume={140},
      number={4},
       pages={457\ndash  477},
        note={MR3432546},
}

\bib{Kry2}{book}{
      author={Krylov, N.V.},
       title={Lectures on elliptic and parabolic equations in {H}{\"o}lder
  spaces},
      series={Graduate Studies in Mathematics},
   publisher={American Mathematical Society},
     address={Providence, RI},
        date={1996},
      volume={12},
        note={MR1406091},
}

\bib{Kry}{book}{
      author={Krylov, N.V.},
       title={Lectures on elliptic and parabolic equations in {S}obolev
  spaces},
      series={Graduate Studies in Mathematics},
   publisher={American Mathematical Society},
     address={Providence, RI},
        date={2008},
      volume={96},
        note={MR2435520},
}

\bib{SY}{article}{
      author={Schoen, R.},
      author={Yau, S.T.},
       title={On the structure of manifolds with positive scalar curvature},
        date={1979},
     journal={Manuscripta Math.},
      volume={28},
      number={1-3},
       pages={159\ndash 183},
        note={MR0535700},
}

\bib{Sim}{article}{
      author={Simon, Miles},
       title={Deformation of ${C}^0$ {R}iemannian metrics in the direction of
  their {R}icci curvature},
        date={2002},
     journal={Comm. Anal. Geom.},
      volume={10},
      number={5},
       pages={1033\ndash 1074},
        note={MR1957662},
}

\bib{Sogge}{book}{
      author={Sogge, Christopher~D.},
       title={Fourier integrals in classical analysis},
      series={Cambridge Tracts in Mathematics},
   publisher={Cambridge University Press},
     address={Cambridge},
        date={2017},
        note={MR3645429},
}

\bib{Top}{book}{
      author={Topping, Peter},
       title={Lectures on the {R}icci flow},
      series={London Mathematical Society Lecture Note Series},
   publisher={Cambridge University Press},
     address={Cambridge},
        date={2006},
      volume={325},
        note={MR2265040},
}

\end{biblist}
\end{bibdiv}
\bibliographystyle{plain}

\end{document}